%% file: _arxiv.tex
\documentclass[bj,preprint]{imsart}

\usepackage{xr-hyper}

\RequirePackage[utf8]{inputenc}
\RequirePackage[OT1]{fontenc}
\RequirePackage[usenames,dvipsnames]{xcolor}

\usepackage{titletoc}

\RequirePackage{amsthm,amsmath,amssymb,bm}
\RequirePackage[round]{natbib}
\RequirePackage[hypertexnames=false,colorlinks,citecolor=blue,urlcolor=blue,linkcolor=blue,filecolor=cyan,breaklinks=true]{hyperref}

\startlocaldefs
\numberwithin{equation}{section}
\newtheorem{theorem}{Theorem}[section]
\newtheorem{proposition}{Proposition}[section]
\newtheorem{lemma}{Lemma}[section]
\newtheorem{corollary}{Corollary}[section]
\newtheorem{definition}{Definition}[section]
\newtheorem{assumption}{Assumption}

\usepackage[sectionbib]{bibunits}
\usepackage{nameref}
\usepackage[nameinlink,noabbrev,capitalise]{cleveref}
\crefname{equation}{equation}{equations}
\crefname{assumption}{Assumption}{Assumptions}
\creflabelformat{enumi}{(#2#1#3)}
\usepackage{autonum}

\usepackage{mathtools}
\usepackage{bbm}
\usepackage{cases}
\input{macros.tex}

\providecommand\given{} % so it exists
\newcommand\SetSymbol[1][]{
  \nonscript\,#1:\nonscript\,\mathopen{}\allowbreak}
\DeclarePairedDelimiterX\Set[1]{\lbrace}{\rbrace}%
{ \renewcommand\given{\SetSymbol[]} #1 }

\endlocaldefs

\begin{document}

%\defaultbibliographystyle{imsart-supnorm}
\defaultbibliography{supnorm}

\begin{frontmatter}
  \title{Adaptive Bayesian density estimation in sup-norm}%
  \runtitle{Adaptive Bayesian density estimation in sup-norm}%
  %\thankstext{T1}{Footnote to the title with the ``thankstext'' command.}

  \begin{aug}
    \author[A]{\fnms{Zacharie} \snm{Naulet}\ead[label=e1]{zacharie.naulet@universite-paris-saclay.fr}}%

    % \thankstext{t1}{Some comment}%
    % \thankstext{t2}{First supporter of the project}%
    % \thankstext{t3}{Second supporter of the project}%
    %\runauthor{Zacharie Naulet}%

    %\affiliation{Universit{\'e} Paris-Saclay\thanksmark{m1}}%

    \address[A]{Universit{\'e} Paris-Saclay\\
      Laboratoire de math{\'e}matiques d'Orsay\\
      91405, Orsay, France\\
      \printead{e1} }
  \end{aug}

  \begin{abstract}
    \input{content-abstract}%
  \end{abstract}


\end{frontmatter}

%\setcounter{tocdepth}{1}%
%\tableofcontents

% This is the content of the document
% ------------------------------------------------------------------------------
\begin{bibunit}[imsart-nameyear]
  \startcontents[main]%
  %\printcontents[main]{l}{1}{\section*{Contents}\setcounter{tocdepth}{1}}%
  \input{content-main}%
  \putbib%
  \stopcontents[main]%
\end{bibunit}

\newpage

% Reset counters for supplemental
% ------------------------------------------------------------------------------
\setcounter{page}{1}
\setcounter{section}{0}
\setcounter{table}{0}
\setcounter{figure}{0}
\setcounter{theorem}{0}
\setcounter{proposition}{0}
\setcounter{lemma}{0}
\setcounter{equation}{0}

% Change numbering for supp mat
% ------------------------------------------------------------------------------
\renewcommand{\thepage}{S\arabic{page}}
\renewcommand{\thesection}{S\arabic{section}}
\renewcommand{\thetable}{S\arabic{table}}
\renewcommand{\thefigure}{S\arabic{figure}}
\renewcommand{\thetheorem}{S\arabic{section}.\arabic{theorem}}
\renewcommand{\theproposition}{S\arabic{section}.\arabic{proposition}}
\renewcommand{\thelemma}{S\arabic{section}.\arabic{lemma}}
\renewcommand{\theequation}{S\arabic{section}.\arabic{equation}}

% Changes the color of the link for more coherence with the journal version
% -----------------------------------------------------------------------------
\hypersetup{hypertexnames=false,colorlinks,citecolor=blue,urlcolor=blue,linkcolor=cyan,filecolor=blue,breaklinks=true}

% Fronmatter for the supplemental
% ------------------------------------------------------------------------------
\begin{frontmatter}
  \title{Adaptive Bayesian density estimation in sup-norm : Supplementary Material}%
  \runtitle{Adaptive Bayesian density estimation in sup-norm}%
  %\thankstext{T1}{Footnote to the title with the ``thankstext'' command.}

  \begin{aug}
    \author{\fnms{Zacharie} \snm{Naulet}\ead[label=e1]{zacharie.naulet@universite-paris-saclay.fr}}%

    % \thankstext{t1}{Some comment}%
    % \thankstext{t2}{First supporter of the project}%
    % \thankstext{t3}{Second supporter of the project}%
    \runauthor{Zacharie Naulet}%

    \affiliation{Universit{\'e} Paris-Saclay\thanksmark{m1}}%

    \address{Universit{\'e} Paris-Saclay\\
      Laboratoire de math{\'e}matiques d'Orsay\\
      91405, Orsay, France\\
      \printead{e1} }
  \end{aug}
  % \begin{abstract}
  %   We investigate the problem of deriving adaptive posterior rates of
  %   contraction on $L^{\infty}$ balls in density estimation. Although it is
  %   known that log-density priors can achieve optimal rates when the true
  %   density is sufficiently smooth, adaptive rates were still to be
  %   proven. Recent works have shown that the so called \textit{spike-and-slab}
  %   priors can achieve optimal rates of contraction under $L^{\infty}$ loss in
  %   white-noise regression and multivariate regression with normal errors. Here
  %   we show that a spike-and-slab prior on the log-density also allows for
  %   (nearly) optimal rates of contraction in density estimation under
  %   $L^{\infty}$ loss. Interestingly, our results hold without lower bound on
  %   the smoothness of the true density.%
  %   \fPROBLEM{zn: The abstract must be rewritten.}
  % \end{abstract}

  % \begin{keyword}[class=MSC]
  %   \kwd[Primary ]{62G07}%
  %   %\kwd{}%
  %   \kwd[; secondary ]{62G20}
  % \end{keyword}

  % \begin{keyword}
  %   \kwd{Bayesian density estimation}%
  %   \kwd{supremum norm}%
  %   \kwd{adaptation}%
  % \end{keyword}
\end{frontmatter}

% The supplemental
% ------------------------------------------------------------------------------
\begin{bibunit}[imsart-nameyear]
  \startcontents[supp]%
  %\printcontents[supp]{l}{1}{\section*{Contents}\setcounter{tocdepth}{1}}%
  \input{content-supp}
  \putbib%
  \stopcontents[supp]%
\end{bibunit}

% \bibliographystyle{imsart-number}
% \bibliography{bnpnetwork}

\end{document}

%% file: macros.tex
\usepackage{stmaryrd}
\usepackage{mathrsfs}
\newcommand{\obs}{\mathbf{X}_n}
\newcommand{\Inner}[1]{\langle #1 \rangle}

\newcommand{\kset}{\mathcal{V}}
\newcommand{\klm}{v}

\newcommand{\basis}{\varphi}
\newcommand{\cbasis}{\theta}
\newcommand{\domain}{[0,1]}
\newcommand{\dclass}{\Sigma} % density class

\newcommand{\LP}[1]{\mathbb{L}_{#1}}

\newcommand{\smooth}{s}
\newcommand{\minsmooth}{\tilde{\smooth}}

\newcommand{\zygmund}{B_{\infty,\infty}^{\smooth}}

\renewcommand{\L}{L}
\newcommand{\dens}{p}                       % density
\newcommand{\pmes}{P}                       % Default letter for proba mes.
\newcommand{\empmes}{\mathbb{P}}            % Empirical measure

\newcommand{\refrate}{\varepsilon_n}                % Reference rate
\newcommand{\natrate}{\mathcal{E}} % 'Natural' rate

\newcommand{\minimaxrate}{\varepsilon_{n}^*(\smooth)} % Minimax rate L^\infty

\newcommand{\sasrate}{\tilde{\varepsilon}_n^{*}(\smooth)}
\newcommand{\He}{\mathsf{H}}
\newcommand{\KL}{\mathsf{KL}}
\newcommand{\VKL}{\mathsf{V}}

\newcommand{\1}{\bm{1}}
\newcommand{\rSPF}{\SPF}
\newcommand{\rSPFC}{\SPFC}

\newcommand{\PF}{\bm{F}_I}
\newcommand{\PFalt}{\bm{F}_I'}
\newcommand{\PFC}{\bar{\bm{F}}_I}
\newcommand{\SPF}{S_{\PF}}
\newcommand{\SPFC}{S_{\PFC}}
\newcommand{\SPFalt}{S_{\PF'}}

\newcommand{\iid}{\overset{iid}{\sim}}
\newcommand{\ind}{\overset{ind}{\sim}}
\newcommand{\EE}{\mathbb{E}}
\newcommand{\PP}{\mathbb{P}}

\global\long\def\Reals{\mathbb{R}}
\global\long\def\Nats{\mathbb{N}}
\global\long\def\Ints{\mathbb{Z}}
\global\long\def\NNInts{\Ints_{+}}
\global\long\def\NNReals{\Reals_{+}}

\global\long\def\intd{\mathrm{d}} % differential d
\newcommand{\Ind}{\boldsymbol{1}}

\newcommand{\supp}{\mathrm{supp}\,}

\newcommand{\lspan}{\mathrm{span}}

\newcommand{\consistencyset}{\mathscr{C}_{*}}
\newcommand{\sliceset}{\mathscr{S}}

\newcommand{\psliceset}{\mathcal{A}}

\newcommand{\LKL}{\tilde{\mathcal{A}}}

\newcommand{\covlog}{\mathsf{C}}

\newcommand{\mustar}{\mu_{\star}}

%% file: content-abstract.tex
We investigate the problem of deriving adaptive posterior rates of contraction
on $\LP{\infty}$ balls in density estimation. Although it is known that
log-density priors can achieve optimal rates when the true density is
sufficiently smooth, adaptive rates were still to be proven. Here we establish
that the so-called \textit{spike-and-slab} prior can achieve adaptive and
optimal posterior contraction rates. Along the way, we prove a generic
$\LP{\infty}$ contraction result for log-density priors with independent wavelet
coefficients.  Interestingly, our approach is different from previous works on
$\LP{\infty}$ contraction and is reminiscent of the classical test-based approach
used in Bayesian nonparametrics. Moreover, we require no lower bound on the
smoothness of the true density, albeit the rates are deteriorated by an extra
$\log(n)$ factor in the case of low smoothness.

%%% Local Variables:
%%% mode: latex
%%% TeX-master: "bernoulli-main"
%%% End:

%% file: content-main.tex
\section{Introduction}
\label{sec:introduction}

\par We consider the problem of estimating a density $\dens$ with respect to
Lebesgue measure on $[0,1]$ given $n$ independent and identically distributed
samples $\obs \coloneqq (X_1,\dots,X_n)$ from the corresponding distribution
$\pmes$. We adopt the Bayesian paradigm and put a joint distribution on
the log-density and the observations.%

\smallskip%
\par Over the decades, there has been a growing interest for the understanding
of the frequentist behaviour of posterior distributions initiated by the seminal
papers of \citet{schwartz-1965,barron-1999,ghosal-ghosh-vaart-2000-conver}. In
particular \citet{ghosal-ghosh-vaart-2000-conver} states generic sufficient
conditions for obtaining rates of concentration of the posterior distribution
near the true model in some distance. The approach relies on the well-known
existence of exponentially powerful test functions. The existence of such tests
depends on the distance considered, and is guaranteed for the $\LP{1}$ or
Hellinger distance between densities, and also for the $\LP{2}$ metric under
supplementary assumptions. It is, however, now well understood that the test
approach fails to give optimal rates when the risk is measured with respect to
the $\LP{\infty}$ distance, see
\citet{castillo2014bayesian,hoffmann2015adaptive,yoo2017adaptive}.

\clearpage%

\smallskip%
\par The failure of the classical approach for $\LP{\infty}$ rates is unfortunate
because one has in general a better intuition of the shape of $\LP{\infty}$ balls
rather than Hellinger balls, making the $\LP{\infty}$ risk a more natural
distance for evaluating performance of estimators. From a frequentist point of
view, density estimation in sup-norm is now well understood. Minimax lower
bounds can be found in \citet{khas1979lower} while upper bounds can be found for
instance in \citet{ibragimov1980estimate,goldenshluger2014adaptive}.

\smallskip%
\par For Bayesian procedures, concentration on $\LP{\infty}$ balls is much less
understood. For the non-adaptive case, the first result goes back to
\citet{gine2011rates} where optimal rates are obtained in white-noise regression
using conjugacy arguments. In the same paper, the authors obtained (possibly
adaptive) rates for density estimation in sup-norm using a testing approach, but
failed to achieve optimality. Using conjugacy arguments, \citet{yoo2016supremum}
also obtain non-adaptive but optimal rates for estimating a regression
function. \Citet{scricciolo-2014-adapt-bayes} adapts the techniques of
\citet{gine2011rates} to obtain optimal rates when the true density is
analytic. The first non-adaptive optimal result in density estimation for non
ultra-smooth densities is to be credited to \citet{castillo2014bayesian}, where
the author uses techniques based on semi-parametric Bernstein--von Mises
theorems. His approach, however, requires a minimal smoothness to be
applicable. Recently, \citet{castillo2017polya} obtained non-adaptive but
optimal rates for density estimation in sup-norm using P{\'o}lya trees prior,
with no lower bound required on the smoothness.

\smallskip%
\par The existence of adaptive and optimal results is, to our knowledge, even
more limited. The first successful result is in \citet{hoffmann2015adaptive}
where the authors get adaptive optimal rates in $\LP{\infty}$ norm for
white-noise regression using a \textit{spike-and-slab} prior. They also obtain
adaptive and optimal rates in density estimation, though their result is rather
an existence result as their abstract sieve prior is not computable. More
recently, \citet{yoo2017adaptive} obtained adaptive optimal rates in
$\LP{\infty}$ norm for estimating a regression function, using a white-noise
approximation of the likelihood to adapt the techniques developed in
\citet{hoffmann2015adaptive}. Since the first version of the current paper,
\citet{castillo2019spike} have introduced spike-and-slab P{\'o}lya trees and
built upon the results of \citet{castillo2017polya} to obtain adaptive
contraction rates, though the arguments we use are different.

\smallskip%
\par In density estimation, it is not obvious to proceed as
\citet{yoo2017adaptive} and reduce the problem to white-noise regression,
although it is known those models are equivalent (in the Le Cam sense) under
certain assumptions. Here, instead, we propose a different approach. We obtain
in \cref{sec:general-approach} a general contraction result for log-density
priors with independent wavelet coefficients. This result is the building block
of our main \cref{thm:3} about the \textit{spike-and-slab} log-density
prior. The posterior \textit{spike-and-slab} is known to be the Bayesian
analogue of \textit{hard thresholding} \citep{haerdle-2000-wavel}, as already
noticed by \citet{hoffmann2015adaptive,yoo2017adaptive}. As such, it constitutes
the prototypical example of model for which we expect adaptive and optimal
$\LP{\infty}$ contraction. Unlike \citet{hoffmann2015adaptive,yoo2017adaptive},
however, the present paper does not exploit the thresholding property of the
spike-and-slab posterior to establish the posterior support of the wavelet
coefficients, but uses a more classical approach.

% contraction, but is less specific. It
% seems hard, however, to find examples that can accommodate independence of the
% coefficients and adaptivity beyond the spike-and-slab prior.
% This is discussed
% more thoroughly, with other implications of our results, in
% \cref{sec:discussion}.%
% \fPROBLEM{so in fact the revision looks more and more like it is specific to the
% s-a-s prior...}%

\smallskip%
\par In the case of the spike-and-slab prior, we show that our method can be
applied to obtain \textit{minimax} optimal and adaptive posterior
contraction. More precisely, we show that if
$\L \coloneqq \log \dens \in \zygmund[0,1]$, where $\zygmund[0,1]$ denotes the
Hölder-Zygmund space with smoothness $\smooth > 1/2$ (see \cref{def:1}), then
there exists $M > 0$ such that as $n \to \infty$,
\begin{equation}
  \label{eq:26}
  \EE_{\L}\Pi(\L'\,:\, \|\L' - \L\|_{\infty} \leq
  M\minimaxrate \mid \obs) = 1 + o(1),
\end{equation}
where $\minimaxrate$ is the minimax rate over bounded balls in $\zygmund[0,1]$
under $\LP{\infty}$ loss \citep{donoho1996density}
\begin{equation}
  \label{eq:31}
  \minimaxrate%
  \coloneqq%
  \Big( \frac{\log n}{n} \Big)^{\frac{\smooth}{2\smooth + 1}}.
\end{equation}
Interestingly, our method can be applied to obtain adaptive rates in the region
$0 < \smooth \leq 1/2$, which to the best of our knowledge is the first result
of this type in the Bayesian literature for methods not relying on conjugacy
arguments. The rates we obtain in this region are, however, slightly
deteriorated by a factor $\log(n)$.

\smallskip%
\par In contrast with previous results in $\LP{\infty}$ loss, the approach used
in this paper is somewhat less specific and uses the same kind of arguments as
for the \textit{master theorem of Bayesian nonparametrics}
\citep{ghosal-ghosh-vaart-2000-conver,ghosal2007b,ghosal2007}. In particular, it
relies on the existence of suitable test functions and proving the prior
positivity of some neighborhoods, in apparent contradiction with the folk wisdom
that no test for the $\LP{\infty}$ loss has enough power to obtain optimal rates
\citet{hoffmann2015adaptive,yoo2017adaptive}. This contradiction is only
apparent, as here we require to test only very specific kind of alternatives,
and exponentially consistent tests are not needed. Although hard to generalize,
we believe the present paper shows that the traditional approach of
\citep{ghosal-ghosh-vaart-2000-conver,ghosal2007b,ghosal2007} is more powerful
than we believed, giving hope for the existence of general contraction results
in strong norms.

\smallskip%
This article comes with a supplementary material, which contain additional
proofs and various classical results about the spike-and-slab log-density
prior. We adopt the convention that every section, subsection, theorem, etc. of
the supplemental has label prefixed by S and is cited in cyan. References to the
main document are cited in blue with no prefix.

\section{Exponentiated random wavelet series}
\label{sec:expon-rand-wavel}

\subsection{Log-density priors}
\label{sec:prior-definition}

We use the $S$-regular, orthogonal, boundary corrected wavelets of
\citet{cohen1993wavelet}, referred to as the CDV basis. We denote this basis by
$\Set{\basis_{j,k} \given (j,k) \in \kset}$, where $\kset \subseteq \NNInts^2$, and refer
to \citet{cohen1993wavelet,gine2016mathematical,castillo2014bayesian} for
details. Each index $(j,k)$ is a pair where
$j \geq 0$ is the wavelet level and $k$ the location index. The CDV basis is an
orthogonal basis for $\LP{2}\domain$ equipped with the Lebesgue measure. We
endow $\LP{2}\domain$ with the inner product
$\Inner{f,g} \coloneqq \int_{\domain} f g$. If $f \in \zygmund\domain$ for some
$\smooth > 0$, then the wavelet series
$\sum_{(j,k) \in \kset}\Inner{f,\basis_{j,k}}\basis_{j,k}$ converges uniformly to
$f$. Conversely, for a given $(\cbasis_{j,k})_{(j,k)\in \kset}$, the series
$\sum_{(j,k)\in \kset}\cbasis_{j,k}\basis_{j,k}$ converges uniformly if $\bm\cbasis$
is in\begin{equation}
  \Theta \subseteq%
  \Set*{ \bm\cbasis \in \Reals^{\kset} \given%
    \textstyle \sup_{x\in \domain}\sum_{(j,k)\in \kset} |\cbasis_{j,k}||\basis_{j,k}(x)| < \infty }.
\end{equation}
Thus, we shall consider prior distributions over the space $\Theta$. Such prior
distribution induces a prior distribution on the space of densities on $\domain$
(with respect to the Lebesgue measure) through the mapping
$\bm\cbasis \mapsto p_{\bm\cbasis}$ such that
\begin{equation}
  \label{eq:10}
  p_{\bm\cbasis}(x)%
  \coloneqq%
  \frac{\exp\{\sum_{(j,k)\in\kset}\cbasis_{j,k}\basis_{j,k}(x)
    \}}{\int_{\domain}
    \exp\{\sum_{(j,k)\in\kset}\cbasis_{j,k}\basis_{j,k}\} },\qquad%
  x \in \domain.
\end{equation}

The coefficients $\bm\cbasis$ in \cref{eq:10} are immediately seen to be the
basis coefficients of an unnormalized version of the log-density
$\log \dens_{\bm\cbasis}$. By defining the log-normalizing constant
\begin{equation}
  \label{eq:5}
  \xi^{\bm\cbasis} \coloneqq -\log \Big(\int_{\domain}\exp\big\{
  \textstyle\sum_{(j,k)\in\kset}\cbasis_{j,k}\basis_{j,k}\big\}\Big)
  \in \Reals,
\end{equation}
we can rewrite the log-density $\L_{\bm\cbasis} \coloneqq \log
\dens_{\bm\cbasis}$ as
\begin{equation}
  \label{eq:45}
  \L_{\bm\cbasis}%
  = \xi^{\bm\cbasis} + \sum_{(j,k) \in \kset}\cbasis_{j,k}\basis_{j,k}%
  =\sum_{(j,k)\in \kset}\big( \cbasis_{j,k} +
  \Inner{\xi^{\bm\cbasis},\basis_{j,k}})\basis_{j,k}.
\end{equation}
% \st{The last display shows that in general the coefficients $\bm\cbasis$ are not
% identifiable. However, it is known that only low-level, \textit{i.e.} father
% wavelets, are non orthogonal to constants. Hence
% $\kset_0 \coloneqq \Set{\klm \in \kset \given \Inner{1,\,\basis_{\klm}} \ne 0}$
% is a finite set, and $\Inner{\xi^{\bm\cbasis},\basis_{\klm}} = 0$ whenever
% $\klm \notin \kset_0$: the coefficients $(\cbasis_{\klm})_{\klm\notin \kset_0}$
% are indeed identifiable.}%
% \fPROBLEM{maybe review that non identifiability issues paragraph in light of new
% notations (and in fact we maybe don't need this so much...)}

\subsection{Spike-and-Slab log-density priors}
\label{sec:spike-slab-log}

To obtain adaptive and optimal rates of contraction, we consider the so-called
\textit{spike-and-slab} prior distribution over $\Theta$
\citep{mitchell1988bayesian}. For some weights
$(\omega_1,\omega_2,\dots) \in [0,1]^{\Nats}$%
\begin{equation}
  \cbasis_{j,k} \ind%
  \begin{cases}
    (1 - \omega_j)\delta_0 + \omega_j Q_j(\cdot) &\mathrm{if}\ 0
    \leq j \leq \frac{\log(n)}{\log(2)},\\
    \delta_0 &\mathrm{if}\ j > \frac{\log(n)}{\log(2)}.
  \end{cases}
\end{equation}
Here $\delta_0$ is the point mass at zero and $Q_j$ are probability
distributions on $\Reals$. We assume $Q_j$ have densities $q_j$ such that for
some $0 < \smooth_0 < 1/2$ and for some density $f$, we have
$q_j(x) \coloneqq 2^{j(\smooth_0 + 1/2)}f(2^{j(\smooth_0 + 1/2)}x)$ for every
$j \geq 0$. We write $F$ the probability distribution with density $f$. We
further assume that there are $a_1,b_1,b_2,\mustar,x_0 > 0$ such that
\begin{gather}
  \label{eq:38}
  a_1e^{-j b_1} \leq \omega_j \leq \min\{1/2,\,2^{-j(1 + \mustar)}\},\\
  \label{eq:39}
  \sup_{x > 1\vee x_0}e^{b_2x}\big(1 - F[-\log(x),\log(x)]\big) \leq 1.
\end{gather}
In order to ensure that the prior puts enough mass on neighborhoods of the true
log-density $\L$, we also assume that for all $G > 0$ there is $g > 0$ such that
\begin{equation}
  \label{eq:40}
  \inf_{x\in [-G,G]} f(x) \geq g.
\end{equation}
We note that the assumptions of \cref{eq:38,eq:40} are classical in the
literature for rates of contraction in supremum loss. The \cref{eq:39} is
however very strong, but guarantees that \textit{a priori}
$\L_{\bm\cbasis}$ has wavelet coefficients of reasonable magnitude,
which guarantees that the posterior concentrates on nice neighborhoods of $\L$,
see in particular \cref{supp-sec:sm:uniform-consistency}. As an example of
distribution $F$ that satisfy the requirements of \cref{eq:39,eq:40}, one could
take the distribution of the random variable $\log(Z)$ where $Z$ has an
inverse-Gaussian distribution (or any distribution on $\NNReals$ with an
exponential behaviour both near $0$ and $\infty$).

\subsection{Adaptation and optimality under supremum loss}
\label{sec:rates-contr-spike}

This paper considers adaptation over bounded Hölder-Zygmund balls, which we
define below. First, we give a precise definition for the Hölder-Zygmund spaces
of smoothness.

\begin{definition}[Hölder-Zygmund spaces]
  \label{def:1}
  For any $\smooth > 0$, the Hölder-Zygmund space $\zygmund\domain$ is the space of
  uniformly continuous functions $f : \domain \to \Reals$ such that
  $\|f\|_{\infty,\infty,\smooth} < \infty$, where $\|f\|_{\infty,\infty,\smooth}
  \coloneqq \sup_{(j,k) \in \kset}2^{j(s+1/2)}|\Inner{f,\basis_{j,k}}|$.
\end{definition}

We are now in position to define the bounded Hölder-Zygmund ball of
log-densities with radius $R > 0$ and smoothness $\smooth > 0$
\begin{equation}
  \label{eq:90}
  \dclass(R,\smooth)%
  \coloneqq \Set*{\L \in \zygmund\domain \given \|\L\|_{\infty,\infty,\smooth} \leq
    R,\, \textstyle \int_{\domain} \exp(\L) = 1}.
\end{equation}

We prove in the supplemental that spike-and-slab log-density priors satisfying
\cref{eq:38,eq:39,eq:40} achieve adaptive and nearly optimal posterior contraction rates
$\minimaxrate$ over $\dclass(R,\smooth)$ under Hellinger loss. In particular,
the next theorem is proven in \cref{supp-sec:sm:hell-contr}.
\begin{theorem}
  \label{thm:hellinger-adapt}
  Let $\Pi$ be the spike-and-slab log-density prior satisfying
  \cref{eq:38,eq:39,eq:40} and let $\He(P,Q)$ denote the Hellinger distance
  between probability distributions $P$ and $Q$. Then for all
  $0 < s_0 \leq \smooth \leq S$ and for all $R > 0$ there exists a constant $M > 0$
  such that
  \begin{equation}
    \lim_{n\to \infty}\sup_{\L\in \dclass(R,\smooth)}
    \EE_{\L}\Pi(\bm\cbasis\,:\, \He(P_{\bm\cbasis},P_{\L}) >
    M\minimaxrate \mid \obs) = 0.
  \end{equation}
\end{theorem}

The main theorem of the paper establishes that spike-and-slab log-density priors
can achieve the optimal posterior contraction rates if we consider the
$\LP{\infty}$ loss. The rate is optimal when $\smooth > 1/2$, with a slight
deterioration in the region of small smoothness.
\begin{theorem}
  \label{thm:3}
  Let $\Pi$ be the spike-and-slab log-density prior satisfying
  \cref{eq:38,eq:39,eq:40}. Also let $\sasrate \coloneqq \minimaxrate$ if
  $\smooth > 1/2$ or $\sasrate \coloneqq \log(n) \minimaxrate$ if
  $0 < \smooth \leq 1/2$. Then for all $0 < s_0 \leq  \smooth \leq S$ and for all $R > 0$
  there exists a constant $M > 0$ such that
  \begin{equation}
    \lim_{n\to \infty}\sup_{\L\in \dclass(R,\smooth)}
    \EE_{\L}\Pi(\bm\cbasis\,:\, \|\L_{\bm\cbasis} - \L\|_{\infty} >
    M\sasrate \mid \obs) = 0.
  \end{equation}
\end{theorem}

We emphasize that the \cref{thm:3} also entails posterior concentration of
$\|p_{\bm\cbasis} - p_{\L}\|_{\infty}$ at rate $\sasrate$, and thus because
$\L \in \Sigma(R,\smooth)$ implies that $e^{-R} \leq p_{\L} \leq e^R$. Hence,
$\|p_{\bm\cbasis} - p_{\L}\|_{\infty}$ and $\|\L_{\bm\cbasis} - \L\|_{\infty}$ are
equivalent distances when the latter is small enough. This assumption that $\L$
belongs to a Hölder ball of smoothness is stronger than the classical
frequentist assumption that only $p_{\L}$ does. In particular, we see that it
entails that $p_{\L}$ is bounded from below, which is of great help in the
proofs. Of course, this begs the question of what can be said when $p_{\L}$ is
smooth but not bounded from below, which is outside the scope of this paper. We
note that assuming smoothness on $\L$ rather than $p_{\L}$ is classical in the
Bayesian community \citep[see][]{castillo2014bayesian}.

\smallskip%
The rest of the paper is organized as follows. In the
\cref{sec:general-approach} we establish the main notations and give the main
ideas behind the proof of the \cref{thm:3}. In particular, we give guidelines
for the proof and state a central contraction result in \cref{thm:dr:1} which is
at the core of the proof of \cref{thm:3}. Then, in \cref{sec:discussion} we
discuss the main implications of our results. Finally, proofs are given in
\cref{sec:proof-crefthm:2,sec:proof-thm:dr:1,sec:remaining-proofs}, respectively
for the \cref{thm:3}, the \cref{thm:dr:1}, and for auxiliary results. Many
secondary proofs are deferred to the \cref{supp-sec:missing-proofs-main} of
the supplemental.

\section{Heuristic and main ideas behind the proof of
  \texorpdfstring{\cref{thm:3}}{Theorem \ref{thm:3}}}
\label{sec:general-approach}

\subsection{Notations}
\label{sec:notations-1}

\par We let $\Nats \coloneqq \Set{1,2,\dots}$ denote the set of natural numbers,
and we let $\NNInts \coloneqq \Set{0,1,\dots}$ denote the set of positive
integers. The symbols $\lesssim$ and $\gtrsim$ are used to denote inequalities
up to generic constants. If $a \lesssim b$ and $b \lesssim a$, we write
$a \asymp b$. For two sequences $(a_n)_{n\in \NNInts}$ and
$(b_n)_{n\in \NNInts}$, the notation $a_n = o(b_n)$ means
$\limsup_{n\rightarrow \infty} |a_n / b_n| = 0$, and $a_n = O(b_n)$ means
$\limsup_{n \rightarrow \infty}|a_n/b_n| = C$ for some $C \geq 0$.  For
$a,b \in \Reals$, we let $a \wedge b$ denote the minimum of $a$ and $b$, and
$a \vee b$ stands for the maximum. % \st{We adopt the convention that for sequence of
% function, $f_n(z) = o(a_n)$ means $\limsup_{n\to\infty} \sup_z|f_n(z)/a_n| = 0$,
% and similarly for the $O(\cdot)$ notation, \textit{i.e.} without explicit mention the remainders
% are uniformly bounded.}%
% \fPROBLEM{I am not so sure we are doing that anymore...}
%

\smallskip%
\par Densities are understood with respect to the Lebesgue measure.  Lower-case
notations $p,q,\dots$ are used to denote densities, while upper-case $P,Q,\dots$
denote the corresponding distributions. Given a log-density $\L$ on $\domain$,
we write $\dens_{\L} \coloneqq \exp\{\L\}$ the corresponding density and
$P_{\L}$ the corresponding distribution. When $\L = \L_{\bm\cbasis}$ for some
$\cbasis \in \Theta$, we abbreviate $P_{\bm\cbasis}$ for $P_{\L_{\bm\cbasis}}$,
etc.

\smallskip%
We see $\obs = (X_1,\dots,X_n)$ as the beginning of an infinite sequence
$\mathbf{X}_{\infty} = (X_1,X_2,\dots)$ defined on a measurable space
$(\Omega,\mathcal{A})$ and such that under $L$, the variables $X_1,X_2,\dots$
are independent and identically distributed (iid) with distribution $P_{\L}$. We
write $\PP_{\L}$ the distribution of $\mathbf{X}_{\infty}$, and we write
indistinctly $\EE_{\L}$ the expectation under $\PP_{\L}$ or under $P_{\L}$. We
write $\PP_n = n^{-1}\sum_{i=1}^n\delta_{X_i}$ the empirical measure of $\obs$.

\smallskip%
We use the standard definitions for the $\LP{p}$ spaces of (equivalence classes)
of functions with finite $\|\cdot\|_p$ norm, with
$\|f\|_p^p \coloneqq \int |f|^p$ if $1\leq p < \infty$, and
$\|f\|_{\infty} \coloneqq \mathrm{ess\,sup}_x|f(x)|$. We will also make use of
the \textit{Hellinger} distance between two probability distributions $P$ and
$Q$ having respective densities $p$ and $q$, defined as
$\He(P,Q) \coloneqq \frac{1}{\sqrt{2}}\|\sqrt{p} - \sqrt{q}\|_2$.

\subsection{Change of parameterization}
\label{sec:change-param}

For some integer $J_0$ to be chosen sufficiently large, we define
$B_0 \coloneqq \Set{(j,k) \in \kset \given j \leq J_0}$. The indices in $B_0$
corresponds to small scales wavelets and will require special cares. To ease the
proof, it is convenient to relabel the wavelets with indices not in $B_0$. We
let $\psi : \Nats \to \kset \backslash B_0$ be the bijection corresponding to the
lexicographical reordering of the index set $\kset \backslash B_0$; i.e. writing
$\psi(m) = (j,k)$ and $\psi(m') = (j',k')$
\begin{align}
  \label{eq:77}
  m \leq m' \iff (j < j')\ \mathrm{or}\ (j = j'\ \mathrm{and}\ k \leq k').
\end{align}
For all $m \geq 1$ we write $J_m \coloneqq \psi(m)_1$ the scale-index of
the wavelet $\basis_{\psi(m)}$. By construction $J_1 = J_0 + 1$ and
$J_1 \leq J_2 \leq \dots$. For proofs, it is also convenient to define
$B_m \coloneqq \Set{\psi(m)}$ for all $m \geq 1$.

\smallskip%
Given this re-indexing of the wavelets, we are now in position to
define a change or parameterization which is convenient for proofs. We pick an
arbitrary reference log-density $\L \in \dclass(R,\smooth)$, and we
establish the concentration under $L$ by taking care that the results are
uniform over $\dclass(R,\smooth)$. Given $\L$, we let
$\cbasis_{j,k}^{\L} \coloneqq \Inner{\L,\basis_{j,k}}$, and we define % for all $m\geq 0$
\begin{equation}
  \label{eq:98}
  F_{m}^{\bm\cbasis} \coloneqq%
  \sum_{(j,k)\in B_m}(\cbasis_{j,k} - \cbasis_{j,k}^{\L})(\basis_{j,k} - \EE_{\L}[\basis_{j,k}]).
  % \begin{cases}
  %   \sum_{(j,k)\in B_0}(\cbasis_{j,k} - \cbasis_{j,k}^{\L})(\basis_{j,k} - \EE_{\L}[\basis_{j,k}]) &\mathrm{if}\ m = 0,\\
  %   (\cbasis_{\psi(m)} - \cbasis_{\psi(m)}^{\L})(\basis_{\psi(m)} - \EE_L[\basis_{\psi(m)}]) &\mathrm{if}\ m \geq 1.
  % \end{cases}
\end{equation}
% For proofs, it is convenient to define $B_m \coloneqq \Set{\psi(m)}$ for $m \geq 1$, so
% that we can write
% $F_m^{\bm\cbasis} \coloneqq $
% for all $m\geq 0$.
Clearly $\L_{\bm\cbasis} -\L$ can be written uniquely in term of the
$(F_{m}^{\bm\cbasis})_{m\geq 0}$, so that we might as well consider
$(F_{m}^{\bm\cbasis})_{m\geq 0}$ as the parameter of the model. For each $m \geq 0$ we
will write
$\mathcal{F}_{m} \coloneqq \lspan\Set{\basis_{j,k} - \EE_{\L}[\basis_{j,k}] \given (j,k) \in B_m }$,
and $\mathcal{F} \coloneqq \mathcal{F}_1 \times \mathcal{F}_2 \times \dots$ the infinite
cartesian product of these spaces. We also let
$\mathcal{F}_0 \coloneqq \Set{\bm{F} \in \mathcal{F} \given \EE_{\L}[\exp\{\sum_{m\geq 0}F_m \}] <\infty }$
denote the subset of proper parameters. Then, we now parameterize the model by
$\bm{F} \in \mathcal{F}_0$. Using the constraint that
$\EE_{\L}[\exp\{\L_{\bm{F}} -\L\}] = 1$, we determine that the log-likehood of
the model $\bm{F} \in \mathcal{F}_0$ is given by
\begin{equation}
  \label{eq:78}
  \L_{\bm{F}} - \L = \sum_{m\geq 0}F_{m} - \log
  \EE_{\L}\big[\exp\{\textstyle\sum_{m\geq 0}F_{m} \}\big].
\end{equation}

\subsection{Guidelines for the proof of \texorpdfstring{\cref{thm:3}}{Theorem \ref{thm:3}} and intermediate
  contraction results}
\label{sec:heuristic}

Here we present the main ideas behind the proof of \cref{thm:3} and the main
intermediate results that are used in the proof. First of all, it is convenient
to assume that the posterior concentrates on nice neighborhoods of $\L$. We will
prove that the posterior concentrates on the set%
\begin{equation}
  \label{eq:42}
  \consistencyset%
  \coloneqq%
  \Set*{ \bm{F} \in \mathcal{F} \given%
    \textstyle\sup_{x\in \domain}\sum_{m\geq 0} |F_{m}(x)|  \leq \delta },
\end{equation}
where $0 < \delta \leq 1$ is a constant to be chosen sufficiently small. Once it has
been shown the posterior is concentrated on $\consistencyset$, the analysis of
the log-likehood difference $\L_{\bm{F}} - \L$ is easier. Posterior contraction
on $\consistencyset$ can be obtained by the classical machinery \textit{à la
  Ghosal and van der Vaart}
\citep{ghosal-ghosh-vaart-2000-conver,ghosal2007,ghosal2007b} and is essentially
a corollary of the \cref{thm:hellinger-adapt}. It can be done using similar
arguments as those already found in \citet{castillo2014bayesian,rivoirard2012},
as we do in \cref{supp-sec:sm:proof-lem:consistency} to prove the following
lemma.%

\begin{lemma}
  \label{lem:consistencyset}
  Let $\Pi$ be the prior described in \cref{sec:spike-slab-log}. Then for all
  $0 < s_0 \leq  \smooth \leq S$, all $R > 0$, and all $\delta > 0$,% and all $\beta_{0} > 0$,
  \begin{equation}
    \label{eq:138}
    \lim_{n\to \infty}\sup_{\L \in \dclass(R,\smooth)}\EE_{\L}\Pi(\consistencyset^{c} \mid \obs) = 0.
  \end{equation}
\end{lemma}

To obtain $\LP{\infty}$ rates, the goal is to relate the distance
$\|\L_{\bm{F}} - \L\|_{\infty}$ to the parameter $\bm{F}$. In particular, we shall
seek to relate $\|\L_{\bm{F}} - \L\|_{\infty}$ to $\Set{\|F_m\|_2 \given m\geq 0}$,
which is motivated by the fact that $\Set{\|F_m\|_2 \given m\geq 0}$ essentially
drives the behaviour of the posterior distributions. The following lemma serves
this purpose.

\begin{lemma}
  \label{lem:distance-bound}
  Let $\bm{F} \in \consistencyset$. Then, there exists a universal constant
  $C > 0$ such that for all choice of $J_0$ we have
  \begin{equation}
    \label{eq:139}
    \|\L_{\bm{F}} - \L\|_{\infty}%
    \leq C\sum_{j \geq J_{0}}2^{j/2} \sup_{m\,:\, J_{m}=j} \|F_m\|_2.
  \end{equation}
\end{lemma}

In view of \cref{lem:consistencyset,lem:distance-bound}, to prove the
\cref{thm:3} it is enough to prove that the posterior concentrates on a set
where the rhs in \cref{lem:distance-bound} is smaller than a multiple constant
of $\sasrate$. Our stretagy is to build a partition of $\mathcal{F}$, where on
each part we have a fine control of $\Set{\|F_m\|_2 \given m \geq 0}$. We build the
partition $(\sliceset_I)_{I \subseteq \NNInts}$, such that for every $I \subseteq \NNInts$,
\begin{equation}
  \label{eq:13}
  \sliceset_I \coloneqq%
  \Set*{ \bm{F} \in  \mathcal{F} \given%
      m \in I \implies%
      \|F_m\|_2 > H_I(m),\quad %
          m \notin I \implies %
    \|F_m\|_2 \leq H_I(m)
  },%
\end{equation}
where we choose $H_m(I)$ as follows. We let $\Gamma,\gamma > 0$ be constants to be
determined, and we define the optimal truncation level $j_n \equiv j_n(\smooth)$ as the only integer
satisfying%
\begin{equation}
  \label{eq:44}
  \begin{cases}
    \gamma 2^{-(j_n+1)(\smooth + 1/2)} < \Gamma \sqrt{\log(n)/n} \leq \gamma 2^{-j_n(\smooth + 1/2)} &\mathrm{if}\ \smooth > 1/2,\\
    \gamma 2^{-(j_n+1)(\smooth + 1/2)} < \Gamma 2^{-j_n/2}\minimaxrate \leq \gamma 2^{-j_n(\smooth + 1/2)} &\mathrm{if}\ 0 < \smooth \leq 1/2.
  \end{cases}%
\end{equation}
Then, for $\xi > 1$ also to be chosen accordingly,
\begin{align}
  \label{eq:88}
  H_I(m)%
  &\coloneqq%
    \begin{cases}
      \Gamma \xi^{-\1_{0\in I}\1_{m\ne 0}}\rho_m &\mathrm{if}\ J_m \leq j_n,\\
      \gamma 2^{-J_m(\smooth + 1/2) } &\mathrm{if}\ J_m > j_n\\
    \end{cases},\qquad%
  \rho_m%
  \coloneqq%
  \begin{cases}
    \sqrt{\log(n)/n} &\mathrm{if}\ \smooth > 1/2,\\
    2^{-J_m/2}\minimaxrate &\mathrm{if}\ 0 < \smooth \leq 1/2.\\
  \end{cases}
\end{align}
At this point, it might look obscure why the definition of $H_I(m)$ differs
according to whether $\smooth > 1/2$ or not, and also according to whether
$0 \in I$ or not. The subtle reason of this choice will be found when proving the
\cref{thm:dr:1}. In fact, we will require to control some covariance terms
involving $F_m$ and $F_{m'}$ (see also \cref{sec:contr-vari-covar}). The control
of these covariance terms can get tricky, and this particular choice of $H_I(m)$
permits to obtain the desired control.

Since $H_I(m)$ is function of $I$,
it is not immediate that $(\sliceset_I)_{I\subseteq \Nats}$ is a proper partition. We
establish this fact in the next lemma, together with a useful property of this
partition.
\begin{lemma}
  \label{lem:partition}
  The collection $(\sliceset_I)_{I \subseteq \NNInts}$
  is a partition of $\mathcal{F}$ and
  \begin{equation}
    \label{eq:55}
    \bm{F} \in \Big( \bigcup_{I \ne \varnothing}\sliceset_I \Big )^{c} \implies%
    \|F_m\|_2 \leq
    \begin{cases}
      \Gamma \rho_m &\mathrm{if}\ J_m \leq j_n,\\
       \gamma 2^{-J_m(\smooth + 1/2)} &\mathrm{if}\ J_m > j_n.
    \end{cases}
  \end{equation}
\end{lemma}
The previous lemma is one of the key result. In conjunction with
\cref{lem:consistencyset,lem:distance-bound}, it implies the following corollary
which is the starting point of the proof of \cref{thm:3}.

\begin{corollary}
  \label{cor:strategybound}
  For all choice of $J_0,\Gamma,\gamma,\xi$ and for all $(R,\smooth)$ there exists $M> 0$
  such that the following bound is true.%
  \begin{equation}
    \label{eq:G3}
    \EE_{\L}\Pi(\bm\cbasis\,:\, \|\L_{\bm\cbasis} - \L\|_{\infty} >
    M\sasrate \mid \obs)%
    \leq%
    \sum_{\substack{I \subseteq \NNInts\\I\ne \varnothing}}\EE_{\L}\Pi(\consistencyset \cap \sliceset_I \mid
    \obs).
  \end{equation}
\end{corollary}

Our strategy is then to bound each of the terms
$\EE_{\L}\Pi(\consistencyset \cap \sliceset_I \mid \obs)$ for $I \ne \varnothing$; which
is done in the \cref{thm:dr:1} below. Interestingly, the technique is
reminiscent to the classical testing approach of
\citep{ghosal-ghosh-vaart-2000-conver,ghosal2007b,ghosal2007} with extra cares.
Large parts of the proofs of \cref{thm:3,thm:dr:1} rely on the fine tuning of
the constants $\delta,J_0,\Gamma,\gamma,\xi$, as well as the relation between those constants, and
also on taking $n$ sufficiently large. Since the proofs are quite long, it can
be challenging to keep track all along of the constraints those constants must
satisfy. To facilitate the understanding of the theorems and their proof, we
summarize in the next assumption how $\delta,J_0,\Gamma,\gamma,\xi$ and $n$ must be taken at the
end of the day for the theorems to hold true.
\begin{assumption}
  \label{ass:1}
  We assume that there are constants $K_0,K_1,K_2,K_3,K_4 > 0$ eventually large and
  eventually depending on $(R,\smooth)$ but solely on $(R,\smooth)$, such that
  \begin{enumerate}
    \item $J_0 \geq K_0\max\{1,\, \log(1/\delta),\, \log(\gamma)\}$;
    \item
          $\log(n) \geq K_1\max\{2^{J_0},\, \delta^{-2},\, \log(\Gamma),\, \log(\gamma),\, \log(\xi) \}$;
    \item  $\xi \geq K_2 \delta^{-1}\max\{1,\gamma\}2^{J_0}$;%
    \item $\gamma \geq K_3\max\{1,\,\delta^{-1}\}$;%
    \item $\Gamma \geq \max\{\gamma,\, K_42^{J_0},\, K_4\xi\}$.
  \end{enumerate}
  The constant $\delta$ will be taken as small as needed.
\end{assumption}

Finally, bounding $\EE_{\L}\Pi(\consistencyset \cap \sliceset_I \mid \obs)$ rely on
splitting the parameter $\bm{F}$ into two parts $\bm{F} = (\PF,\PFC)$ where
$\PF \coloneqq (F_m)_{m\in I}$ and $\PFC \coloneqq (F_m)_{m\notin I}$. The following
functions will also be used:
\begin{equation}
  \label{eq:50}
  \SPF
  \coloneqq \sum_{m \in I} F_{m},\
  \mathrm{and},\ %
  \SPFC \coloneqq \sum_{m \notin I} F_{m}.
\end{equation}

\begin{theorem}
  \label{thm:dr:1}
  Suppose \cref{ass:1} is satisfied with constants $K_0,K_1,K_2 > 0$
  sufficiently large and let $\Pi$ be a prior such that $\PF$ and $\PFC$ are
  independent for all $I \subseteq \NNInts$. Then, there are constants
  $c_0,c_1,c_2,\delta_0 > 0$ such that for all $0 < \delta \leq \delta_0$, for all
  $\frac{512\cdot c_0\delta}{1 + 512\cdot c_0\delta} < \alpha \leq 1/2$, and for all $I \subseteq \NNInts$,
  \begin{equation}
    \label{eq:dr:20}
    \EE_{\L}\Pi(\consistencyset \cap \sliceset_I  \mid \obs)%
    \leq \Bigg\{ \frac{3\exp\big\{-c_2n\natrate_I^2 + c_12^{J_0}|I|\}^{\alpha}}{(1- e^{-c_2 n\natrate_I^2})^{2\alpha}}%
      \Big\{\frac{\Pi(\psliceset_I)}{\Pi(\LKL_I)} \Big\}^{1 - \alpha}
    \Bigg\}^{\frac{1}{1+2\delta}},
  \end{equation}
  where
  $\natrate_I \coloneqq \inf\Set{\EE_{\L}[\SPF^2]^{1/2} \given \bm{F} \in \sliceset_I}$,
  $\mathcal{A}_I \coloneqq \Set{\PF \given \|F_m\|_2 > H_I(m) }$,
  and
  $\LKL_I \coloneqq \Set{\PF \given \EE_{\L}[\SPF^2] \leq \delta^2 \natrate_I^2,\, \|\SPF\|_{\infty} \leq
    \delta}$; provided that $\Pi(\LKL_I) > 0$ and $\natrate_I > 0$ for all $I \subseteq \NNInts$.
\end{theorem}

We note that the fact that $\natrate_I > 0$ and $\Pi(\LKL_I) > 0$ for the
spike-and-slab prior are consequences of \cref{lem:2,pro:1} that will be
established later. Also, we point out that in the whole paper we make the abuse
of notations of writing $\Pi$ to denote the prior on $\bm\cbasis$,
$\L_{\bm\cbasis}$, $\L_{\bm{F}}$, $\bm{F}$, as well as for the restricted
parameters $\PF$ or $\PFC$. The proof of the \cref{thm:3} consists on
specializing the bound of \cref{thm:dr:1} to the spike-and-slab prior and using
it in conjunction with \cref{cor:strategybound} to conclude.

\section{Discussion}
\label{sec:discussion}

\paragraph*{The master theorem of Bayesian nonparametrics}

The current state-of-the-art method in calculating posterior contraction rates
is the master theorem developed by
\citet{ghosal-ghosh-vaart-2000-conver,ghosal2007,shen2001rates}. This theorem
relies on two main ingredients:
\begin{itemize}
  \item The existence of tests for the hypothesis $H_0 : \L' = \L$ against the
  alternative $H_1 : \|\L' - \L\|_{\infty} > M\minimaxrate$, with Type I and
  Type II errors decreasing as $\exp\{- Kn \minimaxrate^2\}$;
  \item The prior puts enough mass on certain Kullback-Leibler neighborhoods of
  $\L$.
\end{itemize}
In the context of $\LP{\infty}$ contraction, it is known that the master theorem
yields suboptimal contraction rates
\citep{gine2011rates,hoffmann2015adaptive,yoo2017adaptive}. The issue is
discussed thoroughly in \citep{hoffmann2015adaptive,yoo2017adaptive}: no test
has enough power to obtain the optimal rate of contraction in $\LP{\infty}$. In
particular, the Type II error has to decay polynomially in $n$, unless we
deteriorate the rate. It is known that not all the alternative $H_1$ has to be
tested -- only a suitable sieve -- but this does not help either to get optimal
rates, the root of the problem being deeper.

\smallskip%
The arguments in \citet{hoffmann2015adaptive,yoo2017adaptive} are strong, and it
is natural to ask what is wrong in the current paper such that the tests we use
in the proofs of \cref{thm:3,thm:dr:1} permit optimal contraction rates. This
indeed relies on the nature of the alternative we test. We are not constructing
tests for $H_1 : \|\L' - \L\|_{\infty} > M \minimaxrate$, but instead for each
$I \subseteq \Nats$, we build a test for
$H_1 : \L' \in \Set{\L_{\bm{F}} \given \bm{F} \in \sliceset_I}$. Those
tests (see the proof of \cref{lem:dr:step2}) have Type I and Type II errors
decreasing as $\exp\{-K n \natrate_I^2 \}$, which is typically
polynomial in $n$ when $|I|$ is small, and thus not in contradiction with the
arguments of the aforementioned papers. We remark that $|I|$ small corresponds
exactly to those log-densities $\L'$ that can be far from $\L$ in $\LP{\infty}$
but close in $\LP{2}$, and thus hard to separate. When $|I|$ gets large, however,
the powers of the tests increase, which is what saves us.

\smallskip%
The main drawback of the method is getting a sharp enough bound on the
denominator of the Bayes rule, which seems hard to do beyond the scope of
independent wavelet coefficients, or at least having a nice structure. Anyhow,
we believe the approach of the current paper shows that the master theorem of
Bayesian nonparametrics can be still useful for $\LP{\infty}$ contraction, giving
some hope toward a general $\LP{\infty}$ contraction result of the same flavour.

\paragraph*{Suboptimality when $0 < \smooth \leq 1/2$}

\par The rates of \cref{thm:3} are slightly suboptimal in the region
$0 < \smooth \leq 1/2$. The problem is indeed not inherent to the spike-and-slab
prior, and as such not surprising as it is known density estimation on the
interval behave very differently when $0 < \smooth \leq 1/2$ or $\smooth > 1/2$,
see for instance \citet{brown1998asymptotic}. Our troubles come from the
impossibility of taking $\rho_m = \sqrt{\log(n)/n}$ when $\smooth \leq 1/2$ and we
have instead to take a much larger threshold $\rho_m = 2^{-J_m/2} \minimaxrate$.
The reasons for this impossibility are to be found in controlling some
covariance terms when decomposing the likelihood, see
\cref{sec:contr-vari-covar}. In fact, this exhibits a major difference on the
strength of the result we prove here: in the case $\smooth > 1/2$ the control of
the posterior is much tighter. In particular, we prove that every wavelet
coefficients of $\L_{\bm\cbasis}$ at level $j \leq j_n$ is within
$\sqrt{\log(n)/n}$ distance of the coefficients of $\L$ if $\smooth > 1/2$,
while we are only able to get a distance of $2^{-j/2}\minimaxrate$ otherwise
(which is much larger when $j$ is small).

\smallskip%
\par To the best of our knowledge, no method based on asymptotic expansions of
the log-likelihood succeeded before in getting posterior $\LP{\infty}$ rates
when $\smooth \leq 1/2$. Thereby, the strategy developed here shed new light on
our understanding of $\LP{\infty}$ contraction. In view of the recent result of
\citet{castillo2019spike}, however, methods based on conjugacy arguments are
able to obtain adaptivity and optimality over all $0 < \smooth \leq 1$, with no
extra $\log(n)$ factor. This shows that we don't really understand yet enough
the behaviour of the log-likelihood when $0 < \smooth \leq 1/2$, which should be
investigated in a near future by the author.

\paragraph*{Estimation of the derivatives}

The spike-and-slab prior of \cref{sec:spike-slab-log} also achieves optimal
contraction rates for estimating the derivatives of the density. We remark that
if $\L$ has derivatives $\L^{(r)}$, $r\geq 1$ integer, then
$\smooth \geq 1 > 1/2$. Then, in this case, investigation of the proof of
\cref{thm:3} shows that the posterior contracts on the set
$\Set{\L' \given |\Inner{\L' - \L,\basis_{j,k}}| \lesssim \sqrt{\log(n)/n}\1_{j\leq j_n} + 2^{-j(\smooth + 1/2)}\1_{j> j_n}\ \forall (j,k)\in \kset }$.
Then, it
is a classical result that this implies for all $1\leq r \leq \smooth \leq S$
with $r$ integer,
\begin{equation}
  \sup_{\L\in \dclass(R,\smooth)}
  \EE_{\L}\Pi(\bm\cbasis\,:\, \|\L_{\bm\cbasis}^{(r)} - \L^{(r)}\|_{\infty} >
  M\minimaxrate^{\frac{\smooth - r}{\smooth}} \mid \obs) = o(1).
\end{equation}

% \paragraph{Other priors}

% The \cref{thm:dr:1} can be applied to any prior satisfying the independence
% requirements. In the non-adaptive case (\textit{i.e}, we can choose the optimal
% truncation level by hand), it can for instance be used to extend the results of
% \citet{castillo2014bayesian} to the region $0 < \smooth \leq 1$. When it comes
% to obtain adaptivity, however, things seem to get a lot trickier. When bounding
% the posterior mass of $\sliceset{I}{\rho_n}$, the contribution of the likelihood
% is only $\exp\{-nC\natrate_I^2 \}$ for some $C > 0$, which might not be
% enough to kill the term $\exp\{K_2|I|\}$, coming from the stochastic control of
% the log-likelihood. Indeed, letting
% $\kset_{n,\smooth} \coloneqq \Set{(j,k)\in \kset \given j \leq J_n(\smooth)}$,
% with $J_n(\smooth)$ the optimal truncation level (see proofs for details), we
% have $\natrate_I^2 \approx |I \cap \kset_{n,\smooth}| \log(n)$ for
% reasonable choice of $(\rho_n)$. Hence, whenever
% $|I| \gg |I\cap\kset_{n,\smooth}| \log(n)$, the prior has to play a big role in
% regularizing the coefficients with indices in $\kset_{n,\smooth}^c$. This is
% only possible if there is a very high prior probability of having those
% coefficients being zero (\textit{i.e.} in the non-adaptive case, or the
% spike-and-slab example), or if the prior variance for those coefficients is
% ridiculously small.

\section{Proof of the \texorpdfstring{\cref{thm:3}}{Theorem \ref{thm:3}}}
\label{sec:proof-crefthm:2}

As explained in \cref{sec:heuristic}, the proof of \cref{thm:3} consists on
plugging the bound of \cref{thm:dr:1} into the bound of
\cref{cor:strategybound}. The first step is to obtain an upper estimate on
$\Pi(\mathcal{A}_I)/\Pi(\LKL_I)$. The next lemma is proved in \cref{supp-sec:proofs-lem:2-lem:3}.

\begin{lemma}
  \label{lem:2}
  Suppose \cref{ass:1} is satisfied with constants $K_0,K_1,K_2,K_3 > 0$
  sufficiently large. Then, there are universal constants $\nu_1,\nu_2 > 0$ such
  that for all $I \subseteq \NNInts$,%
  \begin{equation}
    \label{eq:151}
    \frac{\Pi(\mathcal{A}_I )}{\Pi(\LKL_I)}%
    \le %
    \nu_1 \exp\Big\{ \nu_2\log(n)\sum_{m\in I}2^{J_0\1_{m=0}}\1_{J_m \leq j_n}%
    - (1+\mustar)\log(2)\sum_{m\in I}J_m\1_{J_m>j_n}%
    \Big\}.
  \end{equation}
  Furthermore, if $I \cap \Set{m \given J_m > \log(n)/\log(2)} \ne \varnothing$, then
  $\Pi(\mathcal{A}_I) / \Pi(\LKL_I) = 0$.
\end{lemma}

Then, we can  plug the bound of \cref{lem:2} into the
\cref{thm:dr:1} and fine-tune $\alpha$ in function of $I$ to obtain a clean bound
on $\EE_{\L}\Pi(\consistencyset \cap \sliceset_I \mid \obs)$. We do so in the following
lemma, also proved in \cref{supp-sec:proofs-lem:2-lem:3}.

\begin{lemma}
  \label{lem:3}
  Suppose \cref{ass:1} is satisfied with constants $K_0,K_1,K_2,K_3,K_4 > 0$
  sufficiently large and $\delta >0$ is taken sufficiently small. Then, there are
  universal constants $\nu_3,\nu_4 > 0$ such that for all $I \subseteq \NNInts$,
  \begin{equation}
    \label{eq:129}
    \EE_{\L}\Pi(\consistencyset \cap \sliceset_I \mid \obs)%
    \leq
    \nu_4 \prod_{\substack{m\in I\\J_m \leq j_n}}n^{-\nu_3K_4^2}%
    \prod_{\substack{m\in I\\J_m> j_n}}2^{-J_m(1 + \mustar/2)}.
    % \nu_4\exp\Big\{ - \nu_3 K_4^2\log(n) \sum_{m\in I}\1_{J_m \leq j_n}%
    % - \Big(1 + \frac{\mustar}{2}\Big)\log(2) \sum_{m\in I}J_m\1_{J_m>j_n}
    % \Big\}.
  \end{equation}
\end{lemma}

Now we are in position to use the bound established in \cref{lem:3} with the
inequality of \cref{cor:strategybound} to finish the proof. Define for
simplicity $g_m \coloneqq \nu_3 K_4^2\log(n)$ if $0 \leq J_m \leq j_n$ and
$g_m \coloneqq J_m(1 + \mustar/2)\log(2)$ if $J_m > j_n$. Then,%
\begin{align}
  \sum_{\substack{I \subseteq \NNInts\\I\ne \varnothing} }\EE_{\L}\Pi(\consistencyset \cap \sliceset_I\mid
  \obs)
  &\leq \nu_4 \sum_{\mathbf{b} \in \Set{0,1}^{\NNInts}}%
    \Ind\Set{\textstyle\sum_mb_m\geq 1\displaystyle }
    \prod_{m\in \NNInts}e^{-g_m b_m}\\
  &\leq \nu_4 \sum_{m'\in \NNInts}%
    \sum_{\mathbf{b} \in \Set{0,1}^{\NNInts}}b_{m'}\prod_{m\in \NNInts}e^{-g_m b_m}\\%
  &= \nu_4 \sum_{m'\in \NNInts}%
    \sum_{\mathbf{b} \in \Set{0,1}^{\NNInts}}e^{-g_{m'}b_{m'}}b_{m'}\prod_{\substack{m\in \NNInts\\m\ne m'}}e^{-g_m b_m}\\
  &= \nu_4  \sum_{m' \in \NNInts}e^{-g_{m'}}\prod_{\substack{m\in\NNInts\\m\ne m'}}(1 + e^{-g_m}).
\end{align}
Hence we get the bound,
\begin{equation}
  \label{eq:72}
  \sum_{\substack{I \subseteq \NNInts\\I\ne \varnothing} }\EE_{\L}\Pi(\consistencyset \cap \sliceset_I\mid
  \obs)
  \leq \nu_4\exp\Big(\sum_{m\in \NNInts}e^{-g_m} \Big)
  \sum_{m \in \NNInts}e^{-g_m}.
\end{equation}
The previous display is $o(1)$ whenever
$\sum_{m\in \NNInts}e^{-g_m} = o(1)$, which we prove now. Indeed,
\begin{align}
  \sum_{m \in \NNInts}e^{-g_m}%
  &= \sum_{m\in \NNInts}e^{-g_m}\1_{J_m\leq j_n} + \sum_{m\in \NNInts}e^{-g_m}\1_{J_m > j_n}\\
  &\lesssim n^{-\nu_3K_4^2 }\sum_{m\in \NNInts}\1_{J_m\leq j_n} + \sum_{m\in \NNInts}2^{-j(1 + \mustar/2)}\1_{J_m > j_n}\\
  &\lesssim n^{-\nu_3K_4^2}\sum_{j= 0}^{j_n}2^j + \sum_{j > j_n}2^{-j(1 + \mustar/2)}2^j\\
  &\lesssim n^{- \nu_3K_4^2}2^{j_n} + 2^{-j_n \mustar / 2},
\end{align}
where the third line follows as there are no more than $\lesssim 2^j$ wavelets at level
$j$. Now, we remark that $2^{j_n} \asymp (n / \log(n))^{2/(2\smooth + 1)}$.
Hence, if $K_4$ is taken large enough, $\sum_{m\in \NNInts}e^{-g_m}= o(1)$, as
claimed.

\section{Proof of the \texorpdfstring{\cref{thm:dr:1}}{Theorem \ref{thm:dr:1}}}
\label{sec:proof-thm:dr:1}

\subsection{Main ideas}
\label{sec:main-ideas}

We already know that
$\sliceset_I \subseteq \Set{\PF \given \EE_{\L}[\SPF^2] \geq \natrate_I^2 }$ by
construction. We obtain a finer result by further slicing the set $\sliceset_I$.
For $y \geq 1$ integer, we let
\begin{align}
  \label{eq:4}
  \sliceset_I^y%
  &\coloneqq
    \sliceset_I \cap \Set{\bm{F}\given y \natrate_I^2 \leq \EE_{\L}[\SPF^2] < (y+1)\natrate_I^2}.
\end{align}
Similarly, we define we define
$\mathcal{A}_I^y \coloneqq \mathcal{A}_I \cap \Set{\PF \given y \natrate_I^2 \leq \EE_{\L}[\SPF^2] < (1+y)\natrate_I^2}$.
Clearly $(\mathcal{A}_I^y)_{y\geq 1}$ is a partition of $\mathcal{A}_I$. The first
lemma establishes a first bound on the posterior mass of $\sliceset_I^y$.

\begin{lemma}
  \label{lem:dr:step1}
  Suppose \cref{ass:1} is satisfied with constants $K_0,K_1,K_2 > 0$
  sufficiently large. Also suppose $\Pi$ is such that $\PF$ and $\PFC$ are
  independent for all $I\subseteq \Nats$, and $\Pi(\LKL_I) > 0$ for all $I\subseteq \NNInts$. Then, there are universal constants
  $c_0,\delta_0 > 0$ such that for all $t > 0$ there is an event $\Omega_{t}$ with
  $\PP_{\L}^n(\Omega_t^c) \leq e^{-t}$ and if $\obs \in \Omega_t$, for all $0 < \delta \leq \delta_0$, for
  all $I \subseteq \NNInts$, for all $y \geq 1$,
  \begin{equation}
    \label{eq:48}
    \Pi(\consistencyset \cap \sliceset_I^y \mid \obs )%
    \leq 2e^{2\delta t + c_0\delta yn\natrate_I^2} \int_{\mathcal{A}_I^y \cap \Set*{\|\SPF\|_{\infty} \leq \eta }} %
    \prod_{i=1}^n \frac{q_{\PF}(X_i)}{\dens_{\L}(X_i)}\,
     \frac{\Pi(\intd \PF)}{\Pi(\LKL_I)},
  \end{equation}
  where $q_{\PF}$ is a probability density on $[0,1]$ whose exact expression is
  known but deferred to the proof of the lemma for convenience.%
\end{lemma}
The last lemma is the key result of the proof. Interestingly, the classical
approach to concentration rates \textit{à la}
\citet{ghosal-ghosh-vaart-2000-conver} consists on establishing a similar
relation, but with $q_{\PF}$ replaced by $\exp\{\L_{\bm{F}}\}$ and $\LKL_I$
replaced by a \textit{Kullback-Leibler} neighborhood of $\dens_{\L}$. We use the
estimate of \cref{lem:dr:step1} to bound
$\EE_{\L}\Pi(\consistencyset \cap \sliceset_I^y \mid \obs)$ using the standard testing
approach \textit{à la} \citet{ghosal-ghosh-vaart-2000-conver}, coupled with the
\textit{square-root trick} of
\citet{lijoi2005consistency,walker2007rates,ghosal2007b}. This step is rather
immediate in view of the existing literature and it boils down to bound
$\inf\Set{\He(Q_{\SPF},P_{\L})^2 \given \PF \in \psliceset_I^y,\, \|\SPF\|_{\infty} \leq \delta }$
and the metric entropy (in the Hellinger distance) of the set of densities
$\mathcal{P}_I^y \coloneqq \Set{q_{\PF} \given \PF \in \psliceset_I^{y},\, \|\SPF\|_{\infty} \leq \delta}$.
This gives the following lemma.%
\begin{lemma}
  \label{lem:dr:step2}
  Suppose \cref{ass:1} is satisfied with constants $K_0,K_1,K_2 > 0$
  sufficiently large,  and let everything as in
  \cref{lem:dr:step1}. Then, there are universal constants $c_1,\delta_0 > 0$ such that for all
  $0 < \delta \leq \delta_0$, for all $I \subseteq \NNInts$, for all $y \geq 1$, for all $t > 0$,
  \begin{equation}
    \label{eq:29}
    \EE_{\L}[\Pi(\consistencyset \cap \sliceset_I^y \mid \obs)\Ind_{\Omega_t}]%
    \leq  \Bigg\{ \frac{8\exp(- \frac{y n\natrate_I^2}{256} + 2\delta t + c_12^{J_0}|I|) }{1 -
      e^{-y n\natrate_I^2 / 256}}\frac{\Pi(\psliceset_I^y)}{\Pi(\LKL_I)} \Bigg\}^{1/2}.
  \end{equation}
\end{lemma}

We can obtain a bound on
$\EE_{\L}[\Pi(\consistencyset \cap \sliceset_I \mid \obs)\Ind_{\Omega_t}]$ by summing over
$y \geq 1$ the bound obtained in \cref{lem:dr:step2}. This gives a valid bound, but
it is in not sharp enough in cases where $|I|$ gets too large or $n\natrate_I^2$
is too small. Indeed, in those cases, we can improve the bound to give more
importance to the prior by remarking that taking the expectation both sides of the expression
in \cref{lem:dr:step1} and applying Fubini's theorem gives%
\begin{equation}
  \label{eq:32}
  \EE_{\L}[\Pi(\consistencyset \cap \sliceset_I^y \mid \obs)\Ind_{\Omega_t}]%
  \leq 2e^{2 \delta t + c_0\delta y n \natrate_I^2}\frac{\Pi(\psliceset_I^y)}{\Pi(\LKL_I)}.
\end{equation}
This improvement permits to assume only $\mustar > 0$, otherwise we would have
to assume $\mustar > 1$, which may be undesirable in practice (as it may cause
over-shrinkage). The next lemma leverages that
$\EE_{\L}[\Pi(\consistencyset \cap \sliceset_I^y \mid \obs)\Ind_{\Omega_t}]$ is bounded above
by the minimum between the expression in \cref{lem:dr:step2} and the last
display to get a sharp bound on
$\EE_{\L}[\Pi(\consistencyset \cap \sliceset_I \mid \obs)\Ind_{\Omega_t}]$.

\begin{lemma}
  \label{lem:1}
  Suppose \cref{ass:1} is satisfied with constants $K_0,K_1,K_2 > 0$
  sufficiently large, and let everything as in
  \cref{lem:dr:step1,lem:dr:step2}. Then, there are universal constants $c_2,\delta_0 > 0$ such that for all
  $0 < \delta \leq \delta_0$, for all $I \subseteq \NNInts$, for all $\frac{512\cdot c_0\delta}{1 + 512\cdot c_0\delta} < \alpha \leq 1/2$, for all $t > 0$,
  \begin{align}
    \label{eq:83}
    \EE_{\L}[\Pi(\consistencyset \cap \sliceset_I \mid \obs)\Ind_{\Omega_t}]%
    &\leq \frac{\sqrt{8}e^{2\delta t}\exp\big\{-c_2n\natrate_I^2 + c_12^{J_0}|I|\}^{\alpha}}{(1- e^{-c_2  n\natrate_I^2})^{2\alpha}}%
      \Big\{\frac{\Pi(\psliceset_I)}{\Pi(\LKL_I)} \Big\}^{1 - \alpha}.
  \end{align}
\end{lemma}
% We point out that summing the bound of \cref{lem:dr:step2} over $y\geq 1$
% essentially gives the bound of \cref{lem:1} if we choose $\alpha = 1/2$. It might
% seem anecdotal to have improved for allowing $\alpha \ll 1/2$ but this has a certain
% importance.

Finally, to obtain the bound in the statement of the theorem, we note that,
\begin{align}
  \label{eq:2}
  \EE_{\L}\Pi(\consistencyset \cap \sliceset_I \mid \obs)%
  &\leq%
  \inf_{t>0}\Big\{e^{-t} + \EE_{\L}[\Pi(\consistencyset \cap \sliceset_I \mid \obs)\1_{\Omega_t}] \Big\}.% \\
  % &\leq%
  %   \inf_{t>0}\Big\{e^{-t}%
  %   + \frac{\sqrt{8}e^{2\delta t}\exp\big\{- \mu n \natrate_I^2 + \frac{\alpha K |I|}{2} \big\}}{(1 - e^{-n\natrate_I^2/256})^{\alpha/2}(1 - e^{-\mu n\natrate_I^2})}%
  %   \Big\{\frac{\Pi(\mathcal{A}_I)}{\Pi(\LKL_I)} \Big\}^{1-\alpha/2} \Big\}
\end{align}
Plugging the bound obtained in \cref{lem:1} into the previous display and
solving to find the infimum gives the bound of the theorem when choosing $\delta$
small enough.

\subsection{Proofs of
  \texorpdfstring{\cref{lem:dr:step1,lem:dr:step2,lem:1}}{Lemmas
    \ref{lem:dr:step1} to \ref{lem:1}}}
\label{sec:proofs-lem:dr:st-lem}

\begin{proof}[Proof of \texorpdfstring{\cref{lem:dr:step1}}{Lemma \ref{lem:dr:step1}}]
Let define $\Phi(f) \coloneqq f - \EE_{\L}[f] - \log \EE_{\L}[e^{f - \EE_{\L}[f]}]$ and
$\covlog(f,g) \coloneqq \log \EE_{\L}[e^{\Phi(f)}e^{\Phi(g)}]$. We will see that
$-\covlog(\SPF,\SPFC)$ is asymptotically equivalent to the covariance of $\SPF$
and $\SPFC$, and thus we will refer abusively to this term as the covariance
from now on. It is easily seen that the log-likelihood can be rewritten as
\begin{equation}
  \label{eq:24}
  \L_{\bm{F}} - \L%
  = \Phi(\rSPF) + \Phi(\rSPFC)%
  - \covlog(\rSPF,\rSPFC).
\end{equation}
Then, by the Bayes rule,
\begin{align}
  \label{eq:173}
  \Pi(\consistencyset \cap \sliceset_I^y \mid \obs)%
  &=%
      \frac%
  {%
    \int_{\consistencyset\cap \sliceset_I^y}e^{n\empmes_n\Phi(\SPF)}e^{n\empmes_n\Phi(\SPFC)}e^{-n\covlog(\SPF,\SPFC)}
    \Pi(\intd\bm{F}) }%
  {%
  \int e^{n\empmes_n\Phi(\SPF)}e^{n\empmes_n\Phi(\SPFC)}e^{-n\covlog(\SPF,\SPFC)}
    \Pi(\intd\bm{F})}.
\end{align}
Recall that $\PF \coloneqq (F_m)_{m\in I}$ and
$\PFC \coloneqq (F_m)_{m\notin I}$. Also, in addition to $\mathcal{A}_I^y$, we let
\begin{align}
  \label{eq:172}
  \mathcal{N}_I%
  \coloneqq \Set*{\PFC\given \forall m \notin I,\  \|F_m\|_2 \leq H_I(m)}.
\end{align}
It is immediate that if $\bm{F} \in \sliceset_I^y$ then
$\PF \in \mathcal{A}_I^y$ and
$\PFC \in \mathcal{N}_I$. Also, if $\bm{F} \in \consistencyset$, then
\begin{equation}
  \max\{|\SPF(x)|,|\SPFC(x)|\}%
  \leq \max\Big\{ \sum_{m\in I}|F_m(x)|,\, \sum_{m\notin I}|F_m(x)| \Big\}%
  \leq \sum_{m\geq 0}|F_m(x)|%
  \leq \delta.
\end{equation}
So for all $\bm{F} \in \consistencyset$ and for all $I \subseteq \NNInts$, we have
$\|\SPF\|_{\infty} \leq \delta$ and $\|\SPFC\|_{\infty}  \leq \delta$. It follows by \cref{eq:173} that
$\Pi(\consistencyset \cap \sliceset_I^y \mid \obs)$ is bounded from above by
\begin{equation}
  \label{eq:174}
  \frac{%
      \int \1_{\mathcal{A}_I^y}(\PF)\1_{\mathcal{N}_I}(\PFC)\1_{\|\SPF\|_{\infty}\leq\delta}\1_{\|\SPFC\|_{\infty}\leq\delta} e^{n\empmes_n\Phi(\SPF)}e^{n\empmes_n\Phi(\SPFC)}e^{-n\covlog(\SPF,\SPFC)}
    \Pi(\intd\bm{F}) }%
  {%
  \int \1_{\tilde{\mathcal{A}}_I}(\PF)\1_{\mathcal{N}_I}(\PFC)\1_{\|\SPF\|_{\infty}\leq\delta}\1_{\|\SPFC\|_{\infty}\leq\delta}e^{n\empmes_n\Phi(\SPF)}e^{n\empmes_n\Phi(\SPFC)}e^{-n\covlog(\SPF,\SPFC)}
    \Pi(\intd\bm{F})}.
\end{equation}
The main challenge in the proof of the theorem is to control the term
$\covlog(\SPF,\SPFC)$ both in the numerator and denominator, which is deferred
to \cref{sec:contr-vari-covar}. In fact, by
\cref{cor:cov-control-num,cor:cov-control-denom}, if the constants $K_0,K_1,K_2$
in \cref{ass:1} are taken sufficiently large, there is a universal $C > 0$ such
that taking $\delta$ small enough gives
\begin{equation}
  \label{eq:175}
  2e^{C\delta y n\natrate_I^2}
    \frac{%
      \int\1_{\mathcal{A}_I^y}(\PF)\1_{\mathcal{N}_I}(\PFC)\1_{\|\SPF\|_{\infty}\leq\delta}\1_{\|\SPFC\|_{\infty}\leq\delta} e^{n\empmes_n\Phi(\SPF)}e^{n\empmes_n\Phi(\SPFC)}
    \Pi(\intd\bm{F}) }%
  {%
  \int \1_{\tilde{\mathcal{A}}_I}(\PF)\1_{\mathcal{N}_I}(\PFC)\1_{\|\SPF\|_{\infty}\leq\delta}\1_{\|\SPFC\|_{\infty}\leq\delta}e^{n\empmes_n\Phi(\SPF)}e^{n\empmes_n\Phi(\SPFC)}
    \Pi(\intd\bm{F})},
\end{equation}
But $\PF$ is independent of $\PFC$ and $\SPF$ is solely function of $\PF$
(respectively $\SPFC$ and $\PFC$), thus
\begin{align}
  \Pi( \consistencyset \cap \sliceset_I^y \mid \obs)%
  &\leq%
  2e^{C\delta y n\natrate_I^2}
  \frac%
  {\int_{\mathcal{A}_I^y}\1_{\|\SPF\|_{\infty}\leq \delta} e^{n\empmes_n\Phi(\SPF)}
    \Pi(\intd\PF) }%
    {\int_{\tilde{\mathcal{A}}_I}\1_{\|\SPF\|_{\infty}\leq \delta} e^{n\empmes_n\Phi(\SPF)} \Pi(\intd\PF) }\\
    \label{eq:3}
  &=2 e^{C\delta y n\natrate_I^2}
  \frac%
  {\int_{\mathcal{A}_I^y}\1_{\|\SPF\|_{\infty}\leq \delta} e^{n\empmes_n\Phi(\SPF)}
    \Pi(\intd\PF) }%
    {\int_{\tilde{\mathcal{A}}_I} e^{n\empmes_n\Phi(\SPF)} \Pi(\intd\PF) },
\end{align}
where the second line follows because $\PF \in \LKL_I \implies \|\SPF\|_{\infty} \leq \delta$
by construction. It is interesting that $x \mapsto \dens_{\L}(x)e^{\Phi\SPF (x)}$ is
indeed a proper density function, \textit{i.e} it is non-negative and integrates
to $1$. We write $q_{\PF}(x) \coloneqq \dens_{\L}(x)e^{\Phi\SPF (x)}$. We then can
bound the expectation of \cref{eq:3} using the standard approach \textit{à la}
\citeauthor{ghosal-ghosh-vaart-2000-conver}. In particular, we arrive at the
bound of \cref{eq:48} by controlling the $\PP_{\L}$-probability of the event
\begin{equation}
  \label{eq:dr:12}
  \Omega_t \coloneqq
  \Set*{ \obs \given \textstyle\int_{\LKL_I}
    \prod_{i=1}^n \frac{q_{\PF}(X_i)}{p_{\L}(X_i)}%
    \frac{\Pi(\intd\PF)}{\Pi(\LKL_I)}
    \geq e^{-n \delta^2 \natrate_I^2}e^{-\sqrt{2n\delta^2 \natrate_I^2t} - \delta t}}.
\end{equation}

\begin{proposition}
  \label{pro:dr:4}
  Let $\Pi$ be any probability measure supported on the set $\LKL_I$. For all
  $I \subseteq \NNInts$, for all $0 < \delta \leq \log(2)$, for all $t > 0$, and for all
  $n > 0$, $\PP_{\L}^n(\Omega_t) \geq 1 - e^{-t}$.
\end{proposition}

Then, on the event that $\obs \in \Omega_t$, the \cref{eq:3} becomes
\begin{align}
  \Pi( \consistencyset \cap \sliceset_I^y \mid \obs)%
    &\leq%
    2 e^{C\delta ny\natrate_I^2 + n \delta^2\natrate_I^2  + \sqrt{2n \delta^2 \natrate_I^2 t} + \delta t}
      \int_{\mathcal{A}_I^y \cap \Set{\|\SPF\|_{\infty}\leq \eta}}%
    \prod_{i=1}^n \frac{q_{\PF}(X_i)}{p_{\L}(X_i)}
    \frac{\Pi(\intd\PF)}{\Pi(\LKL_I)}.
\end{align}
The conclusion follows because $y \geq 1$, and because if $t \leq 2n\natrate_I^2$ then
$\sqrt{2n \delta^2 \natrate_I^2t} + \delta t \le 2\delta n\natrate_I^2 + \delta t$, while if
$t > 2n\natrate_I^2$ then $\sqrt{2n \delta^2\natrate_I^2 t} + \delta t \leq 2\delta t$. Hence we
can take $\delta_0 = \log(2)$ and $c_0 = C + 2 + \delta_0$.
\end{proof}

\begin{proof}[Proof of \texorpdfstring{\cref{lem:dr:step2}}{Lemma \ref{lem:dr:step2}}]
First we obtain a lower bound on
$\inf\Set{ \He(Q_{\PF},P_{\L})^2 \given \PF \in \psliceset_I^y,\, \|\SPF\|_{\infty} \leq \delta}$.
The following proposition helps.
\begin{proposition}
  \label{pro:lem:2}
  As $\eta \to 0$ it holds
  $\He(Q_{\PF},P_{\L})^2 \geq \frac{1}{8}\EE_{\L}[\SPF^2]e^{O(\eta)}$ for
  all $\SPF$ satisfying $\|\SPF\|_{\infty} \leq \eta$. Then
  $\inf\Set{ \He(Q_{\PF},P_{\L})^2 \given \PF \in
    \psliceset_I^y,\, \|\SPF\|_{\infty}\leq \delta } \geq \frac{y \natrate_{I}^2}{16}$ for
  $\delta$ small enough (but not depending on $I$ nor on $y$).%
\end{proposition}

For $\epsilon > 0$ and any subset $A$ of a metric space equipped with metric $d$, we
let $N(\epsilon, A, d)$ denote the $\epsilon$-covering number of $A$, \textit{i.e.} the
smallest number of balls of radius $\epsilon$ needed to cover $A$. if $d$ is induced by
some norm $\|\cdot\|$, we write $N(\epsilon, A, \|\cdot\|)$. By
\citet[Corollary~1]{ghosal2007b} our \cref{pro:lem:2} implies that for all
$D > 0$, all $y \geq 1$, and all $n \geq 1$ there exists a test $\phi_{n,y}$ such
that%
\begin{equation}
  \EE_{\L}[\phi_{n,y}]%
  \leq \frac{N(\frac{\sqrt{y}\natrate_I }{16}, \mathcal{P}_I^y,\He)}{D} \frac{e^{- \frac{y n \natrate_I^2}{256}}}{1 -
    e^{- \frac{y n\natrate_I^2}{256}} },\quad%
  \sup_{\substack{\PF\in \psliceset_I^y\\\|\SPF\|_{\infty}\leq \delta} }\EE_{\PF}[1 - \phi_{n,y}] \leq D
  e^{- \frac{y n\natrate_I^2}{256}},
\end{equation}
where $\EE_{\PF}$ is understood as the expectation under $Q_{\PF}^{\otimes n}$, and
where
$\mathcal{P}_I^y \coloneqq \Set{q_{\PF} \given \PF \in \mathcal{A}_I^y,\ \|\SPF\|_{\infty} \leq \delta}$.
Using the estimate of \cref{lem:dr:step1} we find that
$\EE_{\L}[\Pi(\consistencyset \cap \sliceset_I^y \mid \obs)\Ind_{\Omega_t}]$ is bounded
by%
\begin{multline}
  \EE_{\L}[\phi_{n,y}\Pi(\consistencyset \cap \sliceset_I^y  \mid \obs)]%
  + \EE_{\L}[(1 - \phi_{n,y})\Pi(\consistencyset \cap \sliceset_I^y  \mid \obs)\Ind_{\Omega_t} ]\\
  \begin{aligned}
    &\leq \EE_{\L}[\phi_{n,y}]%
    + 2e^{2\delta t+ c_0\delta y n\natrate_I^2} \int_{\psliceset_I^y\cap \Set{\|\SPF\|_{\infty} \leq \eta} } \EE_{\PF}[1 -
    \phi_{n,y}]%
    \frac{\Pi(\intd \PF)}{\Pi(\LKL_I)}\\
    &\leq \frac{N(\frac{\sqrt{y}\natrate_I }{16}, \mathcal{P}_I^y,\He)}{D} \frac{e^{- \frac{y n \natrate_I^2}{256}}}{1 -
      e^{- \frac{y n\natrate_I^2}{256}} }%
    + 2D e^{2\delta t +c_0\delta y n \natrate_I^2 - \frac{yn\natrate_I^2}{256}}  \frac{\Pi(\mathcal{A}_I^y)}{\Pi(\LKL_I)}.
  \end{aligned}
\end{multline}
The
previous display is true for any $D > 0$ and thus we can optimize over $D$,
which is also known as the \textit{square-root trick}
\citep{lijoi2005consistency,walker2007rates}. Doing so gives the bound,
\begin{equation}
  \label{eq:16}
  \EE_{\L}[\Pi(\consistencyset \cap \sliceset_I^y \mid \obs)\1_{\Omega_t}]%
  \leq \Big\{ \frac{8e^{2\delta  t + c_0\delta y n\natrate_I^2} N(\frac{\sqrt{y}\natrate_I }{16},\mathcal{P}_I^y,\He)}{1 -
    e^{-y n\natrate_I^2 / 256}}\frac{\Pi(\psliceset_I^y)}{\Pi(\LKL_I)} \Big\}^{1/2}e^{-\frac{yn\natrate_I^2}{256}}.
\end{equation}

To obtain the bound in the statement of the lemma, it is enough to prove
that
$\sup_{y\geq 1}N(\frac{\sqrt{y}\natrate_I}{16},\mathcal{P}_I^y,\He) \leq \exp(c_12^{J_0}|I|)$.
The following lemma helps.
\begin{proposition}
  \label{pro:lem:3}
  There exists $\delta_0 > 0 $ such that for all $\delta \leq \delta_0$, all $I \subseteq \NNInts$, and
  all
  $\PF,\PFalt \in\psliceset_I \cap \Set{\|\SPF\|_{\infty} \leq \delta}$ it holds
  $\He(Q_{\PF},Q_{\PFalt})^2 \leq \frac{1}{2}\EE_{\L}[(\SPF - \SPFalt)^2]$.%
\end{proposition}
Observe that for $\PF,\PFalt \in \psliceset_I$ we have
$\SPF - \SPFalt = \sum_{m\in I}(F_{m} - F_{m}')$, where by construction each
$F_{m} - F_{m}'$ is in $\mathcal{F}_{m}$. Then, by
\cref{pro:coeff:new}-\eqref{item:coeff:4}, and then by
\cref{pro:coeff:new}-\eqref{item:coeff:3},
\begin{align}
  \label{eq:35}
  \EE_{\L}[(\SPF - \SPFalt)^2]%
  &\lesssim \sum_{m\in I}\|F_{m} - F_{m}'\|_2^2%\\
  \lesssim \sum_{m\in I}\sum_{\klm\in B_{m}}\Inner{F_{m} - F_{m}',\basis_{\klm}}^2.
\end{align}
On the other hand, for all $\PF \in \psliceset_I^y$, we have by
\cref{pro:coeff:new} that,
\begin{align}
  \label{eq:150}
  \sum_{m\in I}\sum_{\klm\in B_{m}}\Inner{F_{m},\basis_{\klm}}^{2}%
  \lesssim \sum_{m\in I}\|F_{m}\|_{2}^2%
  \lesssim \EE_{\L}[\SPF^2]%
  \lesssim y \natrate_I^2.
\end{align}
By \cref{eq:35,eq:150} and \cref{pro:lem:3}, we find that
$N(\frac{\sqrt{y}\natrate_I}{16},\mathcal{P}^y_I,\He)$ is no more than the
covering number of a ball of radius $\lesssim \sqrt{y}\natrate_I$ with balls of radius
$\asymp \sqrt{y}\natrate_I$ in $\Reals^{p}$
equipped with the euclidean distance, with
$p = \sum_{m\in I}|B_m| \leq |B_0|\cdot |I|$.
By
\citet[Lemma~4.1]{pollard1990empirical}, this implies that there is a universal
$K > 0$ such that
\begin{equation}
  \label{eq:20}
  N\Big(\frac{\sqrt{y}\natrate_I}{16}, \mathcal{P}_I^y,\He\Big)%
  \leq%
  \max\Big\{1,\,\Big( \frac{3K \sqrt{y}\natrate_I }{\sqrt{y}\natrate_I } \Big)^p\Big\}%
  \leq \max\{1,\,(3K)^{p}\}.
\end{equation}
Finally, by construction it is true that $|B_0| \lesssim 2^{J_0}$.
\end{proof}

\begin{proof}[Proof of \texorpdfstring{\cref{lem:1}}{Lemma \ref{lem:1}}]
We use the fact that for all $u,v \geq 0$ we have
$\min\{u,v\} = \inf_{\beta\in(0,1)}\{u^{1-\beta}v^{\beta}\}$. Then, combining the bounds of
\cref{lem:dr:step2} with the bound of \cref{eq:32}, we get that for any
$\beta \in (0,1)$ and $y \geq 1$%
\begin{align}
  \label{eq:34}
  \frac{\EE_{\L}[\Pi(\consistencyset \cap \sliceset_I^y \mid \obs)\1_{\Omega_t}]}{\sqrt{8}e^{2\delta t}}%
  &\leq \min\Big\{e^{c_0\delta y n\natrate_I^2} \frac{\Pi(\mathcal{A}_I^y)}{\Pi(\LKL_I)},\,%
    \Big(\frac{\exp(-\frac{y n \natrate_I^2}{256} + c_12^{J_0}|I|)}{1 - e^{-yn\natrate_I^2/256}} \frac{\Pi(\mathcal{A}_I^y)}{\Pi(\LKL_I)} \Big)^{1/2}
    \Big\}\\
  &\leq \frac{\exp\big\{-yn\natrate_I^2(\frac{\beta}{512} +c_0\delta\beta - c_0\delta) + \frac{\beta c_12^{J_0} |I|}{2} \big\} }{(1 - e^{-n\natrate_I^2/256})^{\beta/2}}%
    \Big\{\frac{\Pi(\mathcal{A}_I^y)}{\Pi(\LKL_I)} \Big\}^{1-\beta/2}.%\\
  % &\leq \frac{\exp\big\{-yn\natrate_I^2(\frac{\beta}{512} +c_0\delta\beta - c_0\delta) + \frac{\beta c_1 2^{J_0}|I|}{2} \big\} }{(1 - e^{-n\natrate_I^2/256})^{\beta/2}}%
  %   \Big\{\frac{\Pi(\mathcal{A}_I^y)}{\Pi(\LKL_I)} \Big\}^{1-\beta/2}.
\end{align}
Hence, for any $\frac{512\cdot c_0\delta}{1 + 512\cdot c_0\delta} < \beta \leq 1$, writing
$\mu\coloneqq \frac{\beta}{512} +c_0\delta\beta - c_0\delta > 0$ for simplicity, by Hölder's inequality,
\begin{align}
  \label{eq:84}
  \sum_{y \geq 1} e^{-y \mu n \natrate_I^2} \Big\{\frac{\Pi(\mathcal{A}_I^y)}{\Pi(\LKL_I)} \Big\}^{1-\beta/2}
  &\leq \Big\{ \sum_{y\geq 1}e^{-2y\mu n \natrate_I^2/\beta} \Big\}^{\beta/2} \Big\{ \sum_{y\geq 1}\frac{\Pi(\mathcal{A}_I^y)}{\Pi(\LKL_I)}  \Big\}^{1-\beta/2}\\
  &= \frac{e^{-\mu n \natrate_I^2}}{(1 - e^{-2y\mu n \natrate_I^2/\beta} )^{\beta/2}} \Big\{ \frac{\Pi(\mathcal{A}_I)}{\Pi(\LKL_I)}  \Big\}^{1-\beta/2},
\end{align}
where the second line follows since $(\mathcal{A}_I^y)_{y\geq 1}$ is a partition of
$\mathcal{A}_I$. Therefore,\begin{align}
  \label{eq:43}
  \sum_{y\geq 1}\frac{\EE_{\L}[\Pi(\consistencyset \cap \sliceset_I^y \mid \obs)\1_{\Omega_t}]}{\sqrt{8}e^{2\delta t}}%
  &\leq
    \frac{\exp\big\{- \mu n \natrate_I^2 + \frac{\beta c_12^{J_0} |I|}{2} \big\}}{(1 - e^{-n\natrate_I^2/256})^{\beta/2}(1 - e^{-2\mu n\natrate_I^2/\beta})^{\beta/2}}%
    \Big\{\frac{\Pi(\mathcal{A}_I)}{\Pi(\LKL_I)} \Big\}^{1-\beta/2}.%\\
  % &\leq \frac{\exp\big\{\frac{\beta K |I|}{2} \big\}}{(1 - e^{-n\natrate_I^2/256})^{\beta/2} \mu n\natrate_I^2}%
  %   \Big\{\frac{\Pi(\mathcal{A}_I)}{\Pi(\LKL_I)} \Big\}^{1-\beta/2}
\end{align}
By taking $\alpha = \beta/2$ and $\beta > \frac{1024 \cdot C\delta}{1 + 512\cdot C\delta}$ we have that
$\mu > \frac{1}{2}(\frac{1}{512} + C\delta)\beta$, whence the conclusion.
\end{proof}

\subsection{Control of the covariance terms}
\label{sec:contr-vari-covar}

The major difficulty in establishing the \cref{thm:dr:1} is to prove estimates on
the covariance terms $\covlog(\SPF,\SPFC)$ that are sharp enough. The estimate
used in the proof of \cref{thm:dr:1} are established in the
\cref{cor:cov-control-num,cor:cov-control-denom} below, which are consequences
of the next lemma. The proof of the \cref{lem:4:new} is quite long and is
deferred to \cref{supp-sec:relat-contr-covar}.

\begin{lemma}
  \label{lem:4:new}
  Suppose \cref{ass:1} is
  satisfied with constants $K_0,K_1,K_2 > 0$ sufficiently large. Then, there is
  a constant $C > 0$ such that for all $1 < \delta \leq 1$, for all $I \subseteq \NNInts$, for
  all $\|\SPF\|_{\infty} \leq \delta$, and for all $\PFC \in \mathcal{N}_I$, if $\smooth > 1/2$
  \begin{align}
    \label{eq:65}
    |\covlog(\rSPF,\rSPFC)|%
    &\leq C \delta \EE_{\L}[\SPF^{2}]%
      +\delta \Gamma \sqrt{\frac{\log(n)}{n}}\|F_0\|_2\1_{0\in I}\\%
    &\quad +\delta \Gamma\xi^{-\1_{0\in I}} \Big\{ \sum_{m\in I}\|F_m\|_2^2 \1_{J_{m}\leq j_{n}}\Big\}^{1/2} \sqrt{\frac{\log(n)}{n}}\\ %
    &\quad %
      + \frac{\delta}{\sqrt{n}} \Big\{ \sum_{m\in I} \|F_m\|_2^2 \1_{J_{m}> j_{n}}\Big\}^{1/2},
  \end{align}
  and if $0 < \smooth \leq 1/2$,
  \begin{align}
    \label{eq:127}
    |\covlog(\rSPF,\rSPFC)|%
    &\leq C \delta \EE_{\L}[\SPF^{2}]%
      +\delta \Gamma 2^{-J_0/2}\minimaxrate \|F_0\|_2\1_{0\in I}\\
    &\quad+%
      \delta \Big\{ \sum_{m\in I}\|F_m\|_2^2\1_{J_{m}\leq j_{n}}\Big\}^{1/2}%
      \Big\{\Gamma^2\xi^{-2\1_{0\in I}}\minimaxrate^{2} \sum_{m\in I}2^{-J_{m}}\1_{J_{m}\leq  j_{n}} \Big\}^{1/2}\\
    &\quad%
      + \delta\Big\{\sum_{m\in I}\|F_m\|_2^2\1_{J_{m}>j_{n}}\Big\}^{1/2}%
      \Big\{ \gamma \sum_{m\in I}2^{-J_{m}(2\smooth + 1)}\1_{J_{m}>j_{n}} \Big\}^{1/2}.
  \end{align}
\end{lemma}

\begin{corollary}
  \label{cor:cov-control-num}
  Suppose \cref{ass:1} is satisfied with constants $K_0,K_1,K_2 > 0$
  sufficiently large. Then, there is a constant $C > 0$ such that
  for all $1 < \delta \leq 1$, for all $I \subseteq \NNInts$, for all $y \geq 1$, for all
  $\PF \in \mathcal{A}_I^y$, and for all $\PFC \in \mathcal{N}_I$
  \begin{equation}
    \label{eq:64}
    \covlog(\SPF,\SPFC) \geq - C\delta y\natrate_I^2 - \frac{\delta}{n}.
  \end{equation}
\end{corollary}

\begin{corollary}
  \label{cor:cov-control-denom}
  Suppose \cref{ass:1} is satisfied with constants $K_0,K_1,K_2 > 0$
  sufficiently large. Then, there is a constant $C > 0$ such that for all
  $1 < \delta \leq 1$, for all $I \subseteq \NNInts$,
  \begin{equation}
    \label{eq:11}
    \sup\Set{\covlog(\SPF,\SPFC) \given \PF \in \LKL_{I},\ \PFC \in \mathcal{N}_{I}}%
    \leq C \delta^2 \natrate_{I}^{2}.
  \end{equation}
\end{corollary}

\subsection{Proofs of
  \texorpdfstring{\cref{pro:dr:4,pro:lem:2,pro:lem:3}}{Propositions
    \ref{pro:dr:4} to \ref{pro:lem:3}}}
\label{sec:proof-pro}

\begin{proof}[Proof of \texorpdfstring{\cref{pro:dr:4}}{Proposition \ref{pro:dr:4}}]

The proof is an adaptation of the classical
\citet[Lemma~8.1]{ghosal-ghosh-vaart-2000-conver}. The first step if to remark
that by Jensen's inequality applied to the logarithm%
\begin{align}
  \label{eq:dr:31}
  \log \int_{\LKL_I} e^{n\empmes_n\Phi(\SPF)}\Pi(\intd\PF)%
  % &\geq \sum_{i=1}^n\int_{\LKL_I} \Phi \SPF(X_i)\,\Pi(\intd\PF)\\
  &\geq \sum_{i=1}^n\int_{\LKL_I}\Big(\Phi \SPF(X_i) - \EE_{\L}[\Phi
    \SPF] \Big) \Pi(\intd\PF)\\%
  &\quad%
    + n \int_{\LKL_I}\EE_{\L}[\Phi \SPF]\,\Pi(\intd\PF).
\end{align}
For all $\PF \in \LKL_I$, since $\EE_L[\SPF] = 0$, we have whenever
$0 < \delta \leq \log(2)$,
\begin{align}
  \label{eq:15}
  \EE_{\L}[\Phi \SPF]%
  &= - \log \EE_{\L}[e^{\SPF}]\\
  &\geq - \log \EE_{\L}\Big[1 + \SPF%
    + \frac{1}{2}\SPF^2e^{\|\SPF\|_{\infty}}\Big]\\
  &\geq - \log\Big(1%
    + \EE_{\L}[\SPF^2]\Big).
\end{align}
By definition
$\EE_{\L}[\SPF^2] \leq \delta^2 \natrate_I^2$ whenever $\PF \in \LKL_I$. That is
$\EE_{\L}[\Phi \SPF] \geq
-\EE_{\L}[\SPF^2] \geq - \delta^2 \natrate_I^2$ for any
$\PF \in \LKL_I$. Hence,
\begin{equation}
  \label{eq:dr:34}
  \log \int_{\LKL_I} e^{n\empmes_n\Phi(\SPF)}\Pi(\intd\PF)%
  \geq \sum_{i=1}^n\int_{\LKL_I}\Big(\Phi \SPF(X_i) -
  \EE_{\L}[\Phi \SPF] \Big) \Pi(\intd\PF)%
  -  n \delta^2 \natrate_I^2.
\end{equation}
Now we define the random variables
$Z_i \coloneqq \int_{\LKL_I}\big(\Phi \SPF(X_i) - \EE_{\L}[\Phi\SPF] \big)\Pi(\intd\PF)$.
Observe that $\EE_{\L}[Z_i] = 0$, and $\Phi \SPF - \EE_{\L}[\Phi \SPF] = \SPF$, so we
have $|Z_i| \le \delta$ because $\Pi$ is a probability measure. Further, by an
application of Jensen's inequality and Fubini's theorem
\begin{align}
  \label{eq:dr:32}
  \EE_{\L}[Z_i^2]%
  &= \EE_L\Big[\Big( \int_{\LKL_I} \SPF(X_i)\, \Pi(\intd \PF) \Big)^2\Big]\\
  &\leq \EE_{\L}\Big[\int_{\LKL_I}\SPF(X_i)^2\,\Pi(\intd\PF) \Big]\\
  &= \int_{\LKL_I} \EE_{\L}[\SPF^2]\,\Pi(\intd\PF).
\end{align}
Therefore $\EE_{\L}[Z_i^2] \leq \delta^2 \natrate_I^2$, because of the definition
of $\LKL_I$, and because $\Pi$ is a probability measure. By the
\cref{eq:dr:34}, the probability of $\Omega_t^c$ is no more than
the probability of having
$\sum_{i=1}^n Z_i \leq - \delta t - \sqrt{2n \delta^2 \natrate_I^2t}$. The
conclusion of the proposition then follows by Bernstein's inequality
\citep[Theorem~2.10]{boucheron2013concentration}.
\end{proof}

\begin{proof}[Proof of \texorpdfstring{\cref{pro:lem:2}}{Proposition \ref{pro:lem:2}}]

Observe that by definition $q_{\PF}(x) = \dens_{\L}(x)e^{\Phi(\SPF)}$. Also we have
$\EE_{\L}[\SPF] = 0$ and thus $\Phi(\SPF) = \SPF - \log \EE_{\L}[e^{\SPF}]$. We
lower bound $\He(Q_{\PF},P_{\L})$ by obtaining an upper bound on the
Hellinger affinity
$\mathsf{R}(Q_{\PF},P_{\L}) \coloneqq \int_{\domain}
\sqrt{q_{\PF}\dens_{\L}}$ and using that
$\He(Q_{\PF},P_{\L})^2 = 1 -
\mathsf{R}(Q_{\PF},P_{\L})$. Clearly
$\mathsf{R}(Q_{\PF},P_{\L}) =
\EE_{\L}[e^{\frac{1}{2}\Phi(\SPF)}] =
\EE_{\L}[e^{\frac{1}{2}\SPF}]/\EE_{\L}[e^{\SPF}]^{1/2}$. But $\|\SPF\|_{\infty} \leq \eta$, thus
$\EE_{\L}[e^{\SPF}] \geq  1 + \EE_{\L}[\SPF] +
\frac{1}{2}\EE_{\L}[\SPF^2]e^{-\eta} = 1 +
\frac{1}{2}\EE_{\L}[\SPF^2]e^{-\eta}$. Consequently,
$\EE_{\L}[e^{\SPF}]^{1/2} \geq 1 +
\frac{1}{4}\EE_{\L}[\SPF^2]e^{O(\eta)}$. Similarly,
$\EE_{\L}[e^{\frac{1}{2}\SPF}] \leq 1 +
\frac{1}{8}\EE_{\L}[\SPF^2]e^{\eta}$. It follows
$\mathsf{R}(Q_{\PF},P_{\L}) \leq 1 -
\frac{1}{8}\EE_{\L}[\SPF^2]e^{O(\eta)}$.
\end{proof}

\begin{proof}[Proof of \texorpdfstring{\cref{pro:lem:3}}{Proposition \ref{pro:lem:3}}]

Remark that $\EE_{\L}[\SPF] = 0$, so
$\Phi(\SPF) = \SPF - \log\EE_{\L}[e^{\SPF}]$, and similarly for
$\Phi(\SPFalt)$. Since $\|\SPF\|_{\infty} \leq \delta$ then $\|\Phi(\SPF)\|_{\infty} \leq 2\delta$, similarly
for $\SPFalt$. By a
Taylor expansion there is $u \in (\Phi(\SPF),\Phi(\SPFalt))$, and hence
$|u| \leq 2\delta$, such that
$e^{\frac{1}{2}\Phi(\SPF)} = e^{\frac{1}{2}\Phi(\SPFalt)} +
\frac{1}{2}(\Phi(\SPF) - \Phi(\SPFalt)e^u$. That is,
$(e^{\frac{1}{2}\Phi(\SPF)} - e^{\frac{1}{2}\Phi(\SPFalt)})^2 \leq
\frac{1}{4}e^{4\delta}(\Phi(\SPF) - \Phi(\SPFalt))^2$. Then we can bound the
Hellinger distance as follows.
\begin{align}
  \label{eq:37}
  \He(Q_{\PF},Q_{\PFalt})^2%
  &= \frac{1}{2}\int \Big(\sqrt{\dens_{\L} e^{\Phi(\SPF)}} - \sqrt{\dens_{\L}
    e^{\Phi(\SPFalt)}}  \Big)^2\\
  &= \frac{1}{2}\EE_{\L}\big[\big(e^{\frac{1}{2}\Phi(\SPF)} -
    e^{\frac{1}{2}\Phi(\SPFalt)} \big)^2]\\
  &\leq \frac{1}{8}e^{4\delta}\EE_{\L}[(\Phi(\SPF) - \Phi(\SPFalt) )^2 ].
\end{align}
Expanding the square in the last equation and using that
$\EE_{\L}[\SPF] = \EE_{\L}[\SPFalt] = 0$, we find that
\begin{equation}
  \label{eq:52}
  \He(Q_{\PF},Q_{\PFalt})^2%
  \leq \frac{1}{8}e^{4\delta}\EE_{\L}[(\SPF - \SPFalt)^2]%
  +
  \frac{1}{8}e^{4\delta}\log^2\frac{\EE_{\L}[e^{\SPF}]}{\EE_{\L}[e^{\SPFalt}]}.
\end{equation}
Now remark that
$\EE_{\L}[e^{\SPF}] = \EE_{\L}[e^{\SPFalt}e^{\SPF - \SPFalt}] =
\EE_{\L}[e^{\SPFalt}] + \EE_{\L}[e^{\SPFalt}(e^{\SPF - \SPFalt} -1 )]$, and
hence
$|\EE_{\L}[e^{\SPF}]/\EE_{\L}[e^{\SPFalt}] - 1| \leq e^{O(\delta)}\EE_{\L}[|\SPF
- \SPFalt|] \leq e^{O(\delta)}\EE_{\L}[(\SPF - \SPFalt)^2]^{1/2}$. It follows that
the second term of the rhs of \cref{eq:52} is bounded by
$\frac{1}{8}e^{O(\delta)} \EE_{\L}[(\SPF - \SPFalt)^2]$, and hence
$\He(Q_{\PF},Q_{\PFalt}) \leq \frac{1}{4}e^{O(\delta)}\EE_{\L}[(\SPF -
\SPFalt)^2]$.
\end{proof}

\subsection{Proofs of
  \texorpdfstring{\cref{cor:cov-control-num,cor:cov-control-denom}}{Corollaries
    \ref{cor:cov-control-num} and \ref{cor:cov-control-denom}}}

\begin{proof}[Proof of \cref{cor:cov-control-num}: the case where $\smooth > 1/2$]

By construction, if $m\in I$ and $\bm{F} \in \psliceset_{I}$, then
$\|F_m\|_2 \geq H_I(m)$.
Thus, if $I \cap \Set{m \given J_{m} \leq j_{n}} \ne \varnothing$, for all $\bm{F}\in \psliceset_I$,
\begin{equation}
  \label{eq:53}
 (\Gamma \xi^{-\1_{0\in I}})^2 \frac{\log(n)}{n} \leq \sum_{m\in I}H_I(m)^2\1_{J_{m}\leq j_{n}} \leq  \sum_{m\in I}\|F_m\|_2^2\1_{J_{m}\leq j_{n}}.
\end{equation}
Also if $0\in I$, then $\Gamma^2\log(n)/n \leq \|F_0\|_2^2$ . Therefore
\cref{lem:4:new} and Young's inequality imply
\begin{multline}
  \label{eq:66}
  |\covlog(\SPF,\SPFC)|%
  \leq C\delta\EE_{\L}[\SPF^{2}]%
    + \delta \|F_0\|_2^2\1_{0\in I}%
    + \delta \sum_{m\in I}\|F_m\|_2^2\1_{J_{m}\leq j_{n}}\\
  %&\quad%
    + \frac{\delta}{2} \sum_{m\in I}\|F_m\|_2^2\1_{J_{m}>j_{n}}+ \frac{\delta}{2n}.
\end{multline}
The conclusion follows since $\sum_{m\in I}\|F_m\|_2^2 \asymp \EE_{\L}[\SPF^2] $ by
\cref{pro:coeff:new}-\eqref{item:coeff:4}.
\end{proof}

\begin{proof}[Proof of \cref{cor:cov-control-num}: the case where $0 < \smooth \leq 1/2$]
By construction, if $m\in I$ and $\bm{F} \in \psliceset_{I}$, then
$\|F_m\|_2 \geq H_I(m)$.
Thus, if $I \cap \Set{m \given J_{m} \leq j_{n}} \ne \varnothing$, for all $\bm{F}\in \psliceset_I$,
\begin{align}
  \label{eq:141}
  (\Gamma \xi^{-\1_{0\in I}})^2\minimaxrate^{2}\sum_{m\in I}2^{-J_{m}}\1_{J_m\leq j_n}%
  \leq \sum_{m\in I}\|F_m\|_2^2\1_{J_{m}\leq j_{n}}.
\end{align}
Similarly if $I \cap \Set{m \given J_{m}> j_{n}} \ne \varnothing$, we have by
construction for all $\bm{F}\in \psliceset_{I}$
\begin{align}
  \label{eq:142}
  \gamma^2\sum_{m\in I}2^{-J_{m}(2\smooth + 1)}\1_{J_{m}>j_{n}}%
  \leq \sum_{m\in I}\|F_m\|_2^2\1_{J_{m}> j_{n}}.
\end{align}
Also of $0\in I$, then $\Gamma^2 2^{-J_{0}}\minimaxrate^{2} \leq \|F_0\|_2^2$.
Therefore \cref{lem:4:new} implies
\begin{align}
  \label{eq:143}
  |\covlog(\SPF,\SPFC)|%
  &\leq C\delta\EE_{\L}[\SPF^{2}] + \delta\|F_0\|_2^2\1_{0\in I}%
    + \delta  \sum_{m\in I}\|F_m\|_2^2.
\end{align}
The conclusion follows since $\sum_{m\in I}\|F_m\|_2^2 \asymp \EE_{\L}[\SPF^2] $ by
\cref{pro:coeff:new}-\eqref{item:coeff:4}.
\end{proof}

\begin{proof}[Proof of \cref{cor:cov-control-denom}]
The proof follows immediately from \cref{lem:4:new}, from the fact that
$\sum_{m\in I}\|F_m\|_2^2 \asymp \EE_{\L}[\SPF^2] $ by
\cref{pro:coeff:new}-\eqref{item:coeff:4}, and from the definition of
$\LKL_I$.
\end{proof}

\section{Auxiliary results and remaining proofs}
\label{sec:remaining-proofs}

\subsection{Relations between norms}
\label{sec:dr:auxiliary-results}

In many places we need to relate norm of various functions. In this section we
collect the propositions that serve this purpose.

\begin{proposition}
  \label{pro:coeff:new}
  Let $J_{0}$ be chosen large enough. Then the following are true.%
  \begin{enumerate}
    \item\label{item:coeff:1} For all $m \in \NNInts$ and all
          $F \in \mathcal{F}_{m}$,
          $F = \sum_{(j,k)\in B_{m}}\Inner{F,\basis_{j,k}}(\basis_{j,k} - \EE_{\L}[\basis_{j,k}])$;
    \item\label{item:coeff:2} For all $m \in \NNInts$, all $F\in \mathcal{F}_{m}$,
          and all $(j,k) \notin B_{0} \cup B_{m} \implies \Inner{F,\basis_{j,k}} = 0$.
    \item \label{item:coeff:3} There exist constants $C_{1},C_{2}> 0$ such that
          for all $m \in \NNInts$ and all $F\in \mathcal{F}_{m}$,
          $C_{1}\sum_{\klm\in B_{m}}\Inner{F,\basis_{\klm}}^{2} \leq \|F\|_{2}^2 \leq C_{2}\sum_{\klm\in B_{m}}\Inner{F,\basis_{\klm}}^{2}$.
    % \item \label{item:coeff:3b} There exist constants $C_1,C_2 > 0$ such that
    %       for all $m \in \NNInts$ and all $F \in \mathcal{F}_m$,
    %       $C_1\|F\|_2 \leq \sum_{\klm\in B_m}|\Inner{F,\basis_{\klm}}| \leq C_2 J_m^{1/2}\|F\|_2$.
    % \item \label{item:coeff:3c} The exist constants $C_1,C_2 > 0$ such that for
    %       all $m \in \NNInts$ and all $F \in \mathcal{F}_m$,
    %       $ C_1 J_m^{-1/2}\|F\|_2 \leq \sup_{\klm\in B_m}|\Inner{F,\basis_{\klm}}| \leq C_2\|F\|_2$.
    \item\label{item:coeff:4} There exist constants $C_1,C_2> 0$ such that for
          all $J \subseteq \NNInts$, for all collections
          $\Set{F_m \in \mathcal{F}_m \given m \in J }$,
          $C_{1}\sum_{m\in J}\|F_{m}\|_2^2 \leq \EE_{\L}[(\sum_{m\in J}F_{m})^2] \leq C_{2}\sum_{m\in J}\|F_{m}\|_{2}^{2}$.
    \item\label{item:norm:2} There exist constants $C_1,C_2 > 0$ such that
          $\sup_x\sum_k|\basis_{j,k}(x)| \leq C_12^{j/2}$ for all $j \geq 0$, and
          $\sum_k|\EE_{\L}[\basis_{j,k}]| \leq C_22^{j/2}$ for all $j\geq 0$. Consequently,
          $\sup_x\sum_k|\basis_{j,k}(x) - \EE_{\L}[\basis_{j,k}]| \leq 2\max\{C_1,C_2\} 2^{j/2}$
          for all $j \geq 0$;
  \end{enumerate}
\end{proposition}

\begin{proposition}
  \label{pro:1}
  Suppose \cref{ass:1} is valid. There exists a universal constant $C>0$ such that for all $I \subseteq \NNInts$,
  \begin{align}
    \label{eq:110}
    \natrate_I^2%
    &\geq C\Big(
       \frac{\Gamma^2\log(n)}{n}\sum_{m\in I }\xi^{-2\1_{0\in I}\1_{m\ne 0}}\1_{J_m\leq j_n}%
      + \gamma^2  \sum_{m\in I}2^{-J_m(2\smooth + 1)}\1_{J_m>j_n} \Big).
  \end{align}
\end{proposition}

\subsection{Proofs of the lemmas used in the guidelines of
  \texorpdfstring{\cref{sec:general-approach}}{Section
    \ref{sec:general-approach}} and proof of the
  \texorpdfstring{\cref{cor:strategybound}}{Corollary \ref{cor:strategybound}}}
\label{sec:proofs-lemma-used}

\begin{proof}[Proof of \texorpdfstring{\cref{lem:distance-bound}}{Lemma \ref{lem:distance-bound}}]
Write
$g = \sum_{m \geq 0}F_m$ for simplicity. Then, remark that
$\EE_{\L}[g] = 0$, and thus
$\EE_L[e^g] \geq \EE_L[1 + g + \frac{1}{2}g^2e^{- \|g\|_{\infty}}] = 1 +
\frac{1}{2}\EE_{\L}[g^2]e^{-\|g\|_{\infty}} \geq 1$, and with the same argument
$1 \leq \EE_{\L}[e^g] \leq 1 + \frac{1}{2}\EE_L[g^2]e^{\|g\|_{\infty}}$. It
follows from \cref{eq:78} that $|\L_{\bm{F}} - \L| \leq |g| + |\log
\EE_{\L}[e^g]| = |g| + \log \EE_L[e^g] \leq |g| +
\frac{1}{2}\EE_L[g^2]e^{\|g\|_{\infty}}$. Since
$\EE_{\L}[g^2] \leq \|g\|_{\infty}^2$,
\begin{align}
  \|\L_{\bm{F}} - \L\|_{\infty}%
  &\leq \Big\|\sum_{m\geq 0}F_m\Big\|_{\infty}\Big(1 + \frac{1}{2}\Big\|\sum_{m\geq 0}F_m\Big\|_{\infty} e^{\|\sum_{m\geq 0}F_m\|_{\infty}} \Big)\\
  \label{eq:112}
  &\lesssim \Big\|\sum_{m\geq 0}F_m\Big\|_{\infty},
\end{align}
because by assumption $\bm{F} \in \consistencyset$. Further, by
\cref{pro:coeff:new},
% \begin{align}
%   \|F_0\|_{\infty}%
%   &\leq \sup_{x\in\domain}\sum_{(j,k)\in B_0}|\Inner{F_0,\basis_{j,k}}||\basis_{j,k}(x) - \EE_{\L}[\basis_{j,k}]|\\
%   &\lesssim \|F_0\|_2 \sum_{j=0}^{J_0}2^{j/2}\\
%   \label{eq:132}
%   &\lesssim \|F_0\|_2 2^{J_0/2},
% \end{align}
% and, since for $m \geq 1$ we have that $(j,k)\in B_m \implies j = J_m > J_0$ and $B_m$ is a
% singleton,
\begin{align}
  \Big\|\sum_{m\geq 0}F_m\Big\|_{\infty}%
  &\leq \sup_{x\in \domain}\sum_{m\geq 0}\sum_{(j,k)\in B_m}|\Inner{F_m,\basis_{j,k}}||\basis_{j,k}(x) - \EE_{\L}[\basis_{j,k}]|\\
  &\leq \sup_{x \in \domain} \sum_{j\geq J_0} \sup_{m : J_m=j}\sup_{(j,k)\in B_m}|\Inner{F_m,\basis_{j,k}}| \sum_k|\basis_{j,k}(x) - \EE_{\L}[\basis_{j,k}]|\\
  \label{eq:133}
  &\lesssim \sum_{j \geq J_0} \sup_{m:J_m=j}\|F_m\|_2 2^{j/2}.
\end{align}
The conclusion follows by combining \cref{eq:112,eq:133}.
\end{proof}

\begin{proof}[Proof of \texorpdfstring{\cref{lem:partition}}{Lemma \ref{lem:partition}}]
We first establish that $(\sliceset_I)_{I\subseteq \NNInts}$ is a partition of $\mathcal{F}$.
Pick $\bm{F} \in \mathcal{F}$ arbitrary. We want to show
that there exists a unique $I \subseteq \NNInts$ such that $\bm{F} \in \sliceset_I$. We
have the following two possibilities:
\begin{itemize}
  \item If $\|F_0\|_2 \geq \rho_0 \Gamma$, choose
  $I = \Set{0} \cup \Set{m \geq 1 \given \|F_m\|_2 > \rho_m \Gamma \xi^{-1},\, J_m \leq j_n } \cup \Set{m \geq 1 \given \|F_m\|_2 > \gamma 2^{-J_m(\smooth + 1/2)},\, J_m>j_n}$.
  \item If $\|F_0\|_2 < \rho_0 \Gamma$, choose
  $I = \Set{m\geq 1 \given \|F_m\|_2 > \rho_m \Gamma,\, J_m \leq j_n} \cup \Set{m \geq 1 \given \|F_m\|_2 > \gamma2^{-J_m(\smooth +1/2)},\, J_m >j_n}$.
\end{itemize}
The index set $I$ is uniquely defined by $\bm{F}$, and $\bm{F} \in \sliceset_I$.
We now prove the second claim. Let $A \coloneqq \bigcup_{I \ne \varnothing}\sliceset_I$.
We can decompose $A$ as $A_{1}\cup A_{2}$ where
$A_{1} \coloneqq \bigcup_{I \ne \varnothing,0\in I}\sliceset_I = \bigcup_{I\subseteq \NNInts,0\in I}\sliceset_I$,
and $A_{2} \coloneqq \bigcup_{I\ne\varnothing, 0\notin I}\sliceset_I$. Remark that
$A_{1} = \Set{\bm{F}\in \mathcal{F} \given \|F_0\|_2\geq  \rho_0 \Gamma }$, and
$A_2 = A_1^c \cap \Set{\bm{F} \in \mathcal{F} \given \exists m\geq 1,\ \|F_m\|_2> H_I(m)}$.
Note that if $0 \in I$ then $A_1^c$ is empty, so
\begin{equation}
  \label{eq:12}
  A_{2} = A_1^c \bigcap %
  \Set*{\bm{F} \in \mathcal{F} \given \exists m\geq 1,\ \|F_m\|_2\geq \rho_m \Gamma\xi^{-1}\1_{J_m \leq j_n} + \gamma 2^{-J_m(\smooth + 1/2)}\1_{J_m > j_n}
  }.
\end{equation}
The conclusion follows since $A^c = A_1^c \cap A_2^c$ and since $\xi > 1$.
\end{proof}

\begin{proof}[Proof of \cref{cor:strategybound}: the case where $\smooth > 1/2$]
In view of \cref{lem:distance-bound,lem:partition}, it is sufficient to show
that
\begin{equation}
  \label{eq:135}
  \sum_{j=J_0}^{j_n}\Gamma\rho_m 2^{j/2} + \sum_{j>j_n}\gamma 2^{-j\smooth} \lesssim \sasrate.
\end{equation}
But, $\gamma\sum_{j>j_n}2^{-j\smooth} \lesssim \gamma 2^{-j_n\smooth} \lesssim (\Gamma / \gamma)^{\frac{2\smooth}{2\smooth + 1}}\minimaxrate$
by the
definition of $j_n$ in \cref{eq:44}. On the other hand
$\sum_{j=J_0}^{j_n}\Gamma\rho_m 2^{j/2} \lesssim \Gamma \sqrt{\log(n)/n}2^{j_n/2} \lesssim \Gamma (\gamma/ \Gamma)^{\frac{1}{2\smooth + 1}} \minimaxrate$,
still by \cref{eq:44}.
\end{proof}

\begin{proof}[Proof of \cref{cor:strategybound}: the case where $0 < \smooth \leq  1/2$]
  As for the other case, it is enough to show that \cref{eq:135} holds true. In
  this cas, $\gamma \sum_{j> j_n}2^{-j\smooth} \lesssim \gamma 2^{-j_n\smooth} \lesssim \Gamma \minimaxrate^2$
  by \cref{eq:44}. Also,
  $\sum_{j=J_0}^{j_n}\Gamma\rho_m 2^{j/2} \leq \Gamma j_n \minimaxrate \lesssim \log(\Gamma \minimaxrate / \gamma)\Gamma\minimaxrate \lesssim \log(n)\minimaxrate$,
  again by \cref{eq:44}.
\end{proof}

% \subsection{Proof of pro:aa}
% \label{sec:proof-pro:aa}

% \PROBLEM{This proof is new, it is to make explicit the form of the likelihood,
%   at least to convince myself I didn't make mistakes...}

% Remark that by construction $\EE_{\L}[\SPF] = \EE_{\L}[\SPFC] = 0$ for any $I$,
% so in fact $\Phi(\SPF) = \SPF -\log \EE_{\L}[e^{\SPF}]$, similarly for $\SPFC$.
% Injecting this in the \cref{eq:78new} we obtain that
% \begin{align}
%   \label{eq:21}
%   \L_{\bm{F}} - \L%
%   &= \SPF + \SPFC - \log \EE_{\L}\big[\exp\{\SPF + \SPFC \}]\\
%   &= \Phi(\SPF) + \log \EE_{\L}[e^{\SPF}] + \Phi(\SPFC) + \log \EE_{\L}[e^{\SPFC}] - \log  \EE_{\L}\big[\exp\{\SPF + \SPFC \}]\\
%   &= \Phi(\SPF) + \Phi(\SPFC) - \log \EE_{\L}\big[e^{\SPF -\log \EE_{\L}[e^{\SPF}]}e^{\SPFC - \log\EE_{\L}[e^{\SPFC}]}  ]\\
%   &= \Phi(\SPF) + \Phi(\SPFC) - \log \EE_{\L}\big[e^{\Phi(\SPF)}e^{\Phi(\SPFC)}].
% \end{align}
\subsection{Proofs of \texorpdfstring{\cref{pro:coeff:new,pro:1}}{Propositions
    \ref{pro:coeff:new} and \ref{pro:1}}}
\label{sec:proof-pro:c-pro:1}

\begin{proof}[Proof of \cref{pro:coeff:new}, \cref{item:coeff:1}]
  By construction we know that there are numbers $a_{j,k} \in \Reals$ such that
  $F = \sum_{(j,k) \in B_{m}} a_{j,k}(\basis_{j,k} - \EE_{\L}[\basis_{j,k}])$. We
  note that if $m \geq 1$ the coefficients $(a_{j,k})$ are uniquely determined by
  $a_{j,k} = \Inner{F,\basis_{j,k}}$, because $J_{m} > J_{0}$ is large enough
  such that all $\Inner{\basis_{j,k},1} = 0$ for all $(j,k) \in B_{m}$. Thus
  $\Inner{F,\basis_{j,k}} = \sum_{(j',k')\in B_{m}}a_{j',k'}\Inner{\basis_{j',k'},\basis_{j,k}} = a_{j,k}$,
  for any $(j,k) \in B_{m}$. This establishes the proof for $m\geq 1$. For $m=0$, it
  is the case that $F_{0}$ is in the span of
  $\Set{\basis_{j,k} \given (j,k) \in B_0}$ (because the constants are included in
  the span), and thus $F_{0}$ can be uniquely written as
  $F_{0} = \sum_{(j,k)\in B_{0}}\Inner{F_{0},\basis_{j,k}} \basis_{j,k}$. But by
  construction, $\EE_{\L}[F_{0}] = 0$, so in fact
  $F_{0} = \sum_{(j,k) \in B_{0}}\Inner{F_{0},\basis_{j,k}}(\basis_{j,k} - \EE_{\L}[\basis_{j,k}])$.
\end{proof}

\begin{proof}[Proof of \cref{pro:coeff:new}, \cref{item:coeff:2}]
This follows from the \cref{item:coeff:1} and because for $(j,k) \notin B_{0}$ and $J_{0}$
large enough, we have $\Inner{1,\basis_{j,k}} = 0$. Therefore,
it it the case that $\Inner{F,\basis_{j,k}} = \sum_{(j',k')\in B_{m}}\Inner{F,\basis_{j',k'}}\Inner{\basis_{j',k'},\basis_{j,k}}$.
By orthogonality of the wavelet basis, the previous is either $0$ if
$(j,k) \notin B_{m}$, or $\Inner{F,\basis_{j,k}}$ otherwise.
\end{proof}

\begin{proof}[Proof of \cref{pro:coeff:new}, \cref{item:coeff:3}]
  The lower bound is immediate because
$\|F|_2^2 = \sum_{(j,k) \in \kset}\Inner{F,\basis_{j,k}}^2 \geq \sum_{(j,k)\in B_{m}}\Inner{F,\basis_{j,k}}^2$,
so indeed $C_1 = 1$ works. For the upper bound, we note that because
$\|\L\|_{\infty} \lesssim 1$ we have $\|F_{m}\|_2^2 \asymp \EE_{\L}[F_{m}^2]$, and by
\cref{item:coeff:1}
\begin{align}
  \label{eq:9}
  \EE_{\L}[F_m^2]%
  &= \EE_{\L}\Big[\Big(\sum_{(j,k) \in B_{m}}\Inner{F,\basis_{j,k}}\basis_{j,k} - \EE_{\L}\Big[\sum_{(j,k) \in B_{m}}\Inner{F,\basis_{j,k}}\basis_{j,k}\Big] \Big)^{2} \Big]\\
  &\leq \EE_{\L}\Big[\Big(\sum_{(j,k) \in B_{m}}\Inner{F,\basis_{j,k}}\basis_{j,k} \Big)^{2} \Big]\\
  &\lesssim \Big\|\sum_{(j,k) \in B_{m}}\Inner{F,\basis_{j,k}}\basis_{j,k} \Big\|_{2}^2\\
  &= \sum_{(j,k)\in B_{m}}\Inner{F,\basis_{j,k}}^2,
\end{align}
where the last line follows by the orthogonality of the wavelet basis.
\end{proof}

% \paragraph{Proof of \cref{item:coeff:3b}}
% $\|F\|_2^2 \lesssim \sum_{\klm\in B_m}\Inner{F,\basis_{\klm}}^{2} \leq \sup_{\klm\in B_m}|\Inner{F,\basis_{\klm}}| \sum_{\klm \in B_m}|\Inner{F,\basis_{\klm}}|$,
% because of the \cref{item:coeff:3}.
% Since
% $\sup_{\klm\in B_m}|\Inner{F,\basis_{\klm}}| \leq \sum_{\klm\in B_m}\Inner{F,\basis_{\klm}}$,
% we obtain that $\|F\|_2 \lesssim \sum_{\klm\in B_m}|\Inner{F,\basis_{\klm}}|$. On the other
% hand, by Cauchy-Schwarz'
% \begin{align}
%   \label{eq:80}
%   \sum_{\klm\in B_m}|\Inner{F,\basis_{\klm}}|%
%   &\leq \Big\{\sum_{\klm\in B_m}|\Inner{F,\basis_{\klm}}|^2\Big\}^{1/2} \sqrt{|B_m|}\\
%   &\leq \Big\{\sum_{\klm\in \kset}|\Inner{F,\basis_{\klm}}|^2\Big\}^{1/2} \sqrt{|B_m|}\\%
%     &= \|F\|_2 \sqrt{|B_m|}.
% \end{align}
% The conclusion follows because $|B_m| \asymp J_m$ by construction.

% \paragraph{Proof of \cref{item:coeff:3c}}

% $\|F\|_2^2 \lesssim \sum_{\klm\in B_m}\Inner{F,\basis_{\klm}}^2 \leq \sup_{\klm\in B_m}\Inner{F,\basis_{\klm}}^2|B_m|$
% because of the \cref{item:coeff:3}. Since $|B_m| \asymp J_m$, this gives the
% lower bound. For the upper bound, it is immediate that
% $\sup_{\klm\in B_m}\Inner{F,\basis_{\klm}}^2 \leq \sum_{\klm\in \kset}\Inner{F,\basis_{\klm}}^2 = \|F\|_2^2$.

\begin{proof}[Proof of \cref{pro:coeff:new}, \cref{item:coeff:4}]
We start with the upper bound, which follows from similar arguments than those
of the \cref{item:coeff:3}. Indeed, recall that $\EE_{\L}[g^2] \asymp \|g\|_2^2$ for all
$g$ because $\|L\|_{\infty} \lesssim 1$, hence
\begin{align}
  \label{eq:104}
  \EE_{\L}\Big[\Big(\sum_{m\in J}F_{m}\Big)^2\Big]%%
  &= \EE_{\L}\Big[\Big( \sum_{m\in J}\sum_{(j,k) \in B_{m}}\Inner{F_{m},\basis_{j,k}}\basis_{j,k} - \EE_{\L}\Big[\sum_{m\in J}\sum_{(j,k) \in B_{m}}\Inner{F_{m},\basis_{j,k}}\basis_{j,k} \Big]  \Big)^{2} \Big]\\
  &\leq \EE_{\L}\Big[\Big( \sum_{m\in J}\sum_{(j,k) \in B_{m}}\Inner{F_{m},\basis_{j,k}}\basis_{j,k}\Big)^{2}\Big]\\
  &\lesssim \Big\|\sum_{m\in J}\sum_{(j,k) \in B_{m}}\Inner{F_{m},\basis_{j,k}}\basis_{j,k} \Big\|_{2}^{2}\\
  &= \sum_{m\in J}\sum_{(j,k) \in B_{m}}\Inner{F_{m},\basis_{j,k}}^{2}
\end{align}
Then, by the \cref{item:coeff:2},
\begin{align}
  \label{eq:27}
  \EE_{\L}\Big[\Big(\sum_{m\in J}F_{m}\Big)^2\Big]%
  \lesssim \sum_{m\in J}\|F_{m}\|_2^{2}.
\end{align}
We now proceed with the lower bound. By the \cref{item:coeff:1},
\begin{align}
  \Big\|\sum_{m\in J}F_{m}\Big\|_2^2%
  &= \sum_{(j,k)\in \kset}\Big\langle\sum_{m'\in J}F_{m'},\basis_{j,k}\Big\rangle^{2}\\
  &\geq \sum_{m \geq 1}\sum_{(j,k)\in B_{m}}\Big\langle\sum_{m'\in J}\sum_{(j',k') \in B_{m'}}\Inner{F_{m'},\basis_{j',k'}}(\basis_{j',k'} - \EE_{\L}[\basis_{j',k'}]),\basis_{j,k}\Big\rangle^{2}\\
  &= \sum_{m \geq 1}\sum_{(j,k)\in B_{m}}\Big\langle\sum_{m'\in J}\sum_{(j',k') \in B_{m'}}\Inner{F_{m'},\basis_{j',k'}}\basis_{j',k'},\basis_{j,k}\Big\rangle^{2}\\
  &= \sum_{m\in J}\sum_{(j,k)\in B_{m}}\Inner{F_{m},\basis_{j,k}}^2\1_{m\ne 0},
\end{align}
where the third line follows because $\Inner{1,\basis_{j,k}} = 0$ for all
$(j,k) \notin B_{0}$, and the last line by orthogonality of the wavelet basis.
Therefore by the \cref{item:coeff:3} it must be the case that
\begin{align}
  \label{eq:134}
  \EE_{\L}\Big[\Big(\sum_{m\in J}F_{m}\Big)^2\Big]%
  \gtrsim \sum_{m\in J}\|F_{m}\|_2^2\1_{m\ne 0}.
\end{align}
The last display gives the proof in the case where $0 \notin J$. We now assume that
$0 \in J$, which is a more delicate case. In this situation, we have that
$F_{0}= \sum_{m\in J}F_{m} - \sum_{m\in J}F_{m}\1_{m\ne 0}$, and thus
\begin{align}
  \EE_{\L}[F_{0}^{2}]%
  &\leq 2\EE_{\L}\Big[\Big( \sum_{m\in J}F_m\Big)^{2} \Big] +  2\EE_{\L}\Big[\Big( \sum_{m\in J}F_m\1_{m\ne 0} \Big)^{2} \Big]\\
    \label{eq:103}
  &\lesssim  \EE_{\L}\Big[\Big( \sum_{m\in J}F_m\Big)^{2} \Big]%
    + \sum_{m\in J}\|F_{m}\|_2^2\1_{m\ne 0},
\end{align}
where the second line follows from the upper bound of \cref{eq:27} applied to
the index set $J \backslash \Set{0}$. Combining \cref{eq:134,eq:103},
\begin{align}
  \label{eq:105}
  \|F_0\|_2^2 \lesssim \EE_{\L}[F_0^2]%
  \lesssim \EE_{\L}\Big[\Big( \sum_{m\in J}F_m\Big)^{2} \Big].
\end{align}
Now if we combine the \cref{eq:134,eq:105}, we have indeed
\begin{align}
  \label{eq:128}
 \EE_{\L}\Big[\Big( \sum_{m\in J}F_m\Big)^{2} \Big]%
  &\gtrsim \max\Big\{\sum_{m\in J}\|F_m\|_2^2\1_{m\ne 0},\, \|F_0\|_2^2 \Big\}\\%
  % &= \frac{1}{2}\max\Big\{\sum_{m\in I}\EE_{\L}[F_{m}^{2}]\1_{m\ne 0},\, \EE_{\L}[F_{0}^{2}] \Big\}%
  %   + \frac{1}{2}\max\Big\{\sum_{m\in I}\EE_{\L}[F_{m}^{2}]\1_{m\ne 0},\, \EE_{\L}[F_{0}^{2}] \Big\}\\
  &\geq \frac{1}{2}\sum_{m\in J}\|F_m\|_2^2\1_{m\ne 0} +  \frac{1}{2}\|F_0\|_2^2\\
  &= \frac{1}{2} \sum_{m\in J}\|F_m\|_2^2.
\end{align}
\end{proof}

\begin{proof}[Proof of \cref{pro:coeff:new}, \cref{item:norm:2}]
The first claim is a well-known localization properties of the wavelet
basis. The second fact follows because $\EE_L[\basis_{j,k}] \leq
\|p_L\|_{\infty}\|\basis_{j,k}\|_1 \lesssim \|p_{\L}\|_{\infty}2^{-j/2}$, and
because there are no more than $2^j$ wavelets at each level $j \geq 0$. The
third fact is obvious.
\end{proof}

\begin{proof}[Proof of \texorpdfstring{\cref{pro:1}}{Proposition \ref{pro:1}}]
From the definition of $\natrate_I$ and from
\cref{pro:coeff:new}-\eqref{item:coeff:4}, it is immediate that
$\natrate_I^2 \gtrsim \sum_{m\in I}H_I(m)^2$. If $\smooth > 1/2$ then
the result is immediate. In case $0 < \smooth \leq 1/2$, then we note that by
definition of $j_n$
\begin{equation}
  \label{eq:70}
  \gamma 2^{-j_n(\smooth +1/2)} \geq \Gamma 2^{-j_n /2} \minimaxrate%
  \implies 2^{-j_n} \geq \Big(\frac{\Gamma}{\gamma}\Big)^{\frac{1}{\smooth}} \Big(\frac{\log(n)}{n}
  \Big)^{\frac{1}{2\smooth + 1}}.
\end{equation}
Therefore,
\begin{align}
  \label{eq:33}
  \sum_{m\in I}\rho_m^2\1_{J_m\leq j_n}%
  &=  \sum_{m\in I} 2^{-J_m} \minimaxrate^2\1_{J_m\leq j_n}\\
  &\geq 2^{-j_n}\minimaxrate^2 \sum_{m\in I}\1_{J_m\leq j_n}\\
  &= \Big(\frac{\Gamma}{\gamma }\Big)^{\frac{1}{\smooth}}\cdot \frac{\log(n)}{n}\sum_{m\in I}\1_{J_m\leq j_n}\\
  &\geq \frac{\log(n)}{n}\sum_{m\in I}\1_{J_m\leq j_n},
\end{align}
where the last line is true under \cref{ass:1}.
\end{proof}

\section*{Acknowledgements}

\par This work was supported by U.S. Air Force Office of Scientific Research
grant \#FA9550-15-1-0074. The author also thanks Daniel M. Roy for helpful
discussions and the opportunity to work on this project.

%%% Local Variables:
%%% mode: latex
%%% TeX-master: "bernoulli-main"
%%% End:

%% file: content-supp.tex
\section{Organization}
\label{sec:introduction}

This document is supplementary material for the article \textit{Adaptive
  Bayesian density estimation in sup-norm}. It contains the missing proofs for
the \textit{spike-and-slab} prior example. We refer to the main document for all
the definitions.
\begin{itemize}
  \item In \cref{sec:sm:notations-1}, we introduce some new notations that were
  not needed in the main document, but which we will need in the supplemental.
  % \item In \cref{sec:sm:expon-rand-wavel}, we recall the prior definition and
  % the assumptions we make.
  \item In \cref{sec:sm:hell-contr}, we prove the posterior concentration on
  small Hellinger neighborhoods of the true density. This is the first step
  toward concentration in stronger distances. %
  \item In \cref{sec:sm:uniform-consistency}, we prove posterior concentration
  on $\LP{2}$ neighborhoods and uniform posterior consistency, which is needed
  to prove the \cref{main-lem:consistencyset} stated in the main document.
  \item In \cref{sec:sm:proof-lem:consistency}, we give the proofs that are
  missing in the main document. In particular, the \cref{main-lem:consistencyset}
  establishing the posterior concentration on $\consistencyset$ is proved in
  \cref{sec:sm:proof-lem:consistency}, the \cref{main-lem:4:new} controlling the
  covariance terms is proved in \cref{sec:relat-contr-covar}, and finally
  \cref{main-lem:2,main-lem:3} are proved in \cref{sec:proofs-lem:2-lem:3}.

  % \item In \cref{sec:adapt-optim-rates} we give the missing part of the proof of
  % \cref{main-thm:3} of the main document, corresponding to the situation where
  % $\smooth > 1/2$.
\end{itemize}

Every section, subsection, theorem, etc. of the supplemental has label prefixed
by S and is cited in cyan. References to the main document are cited in blue
with no prefix.

\section{Notations}
\label{sec:sm:notations-1}

We use the same conventions as in the main paper. We furthermore make use of the
following measures of discrepancy between probability distributions. The
Kullback--Leibler divergence is written
$\KL(P,Q) \coloneqq \int_0^1 p \log(p/q)$. We also use the second-order measure
of discrepancy $\VKL(P,Q) \coloneqq \int_0^1 p \log^2(p/q)$.

\smallskip%
For convenience, we also define the following sequence norms on $\Theta$. We
denote the usual $\ell_2$ norm by
$\|\bm\cbasis\|_2^2 \coloneqq \sum_{(j,k) \in \kset}|\cbasis_{j,k}|^2$. In
addition, we define the mixed $\ell_{1,\infty}$ norm such that
$\|\bm\cbasis\|_{1,\infty} \coloneqq \sum_{j\geq
  0}\max_k|\cbasis_{j,k}|2^{j/2}$. It is a well-known fact that if
$f = \sum_{(j,k)\in \kset}\cbasis_{j,k}\basis_{j,k}$, then
$\|f\|_{\infty} \lesssim \|\bm\cbasis\|_{1,\infty}$, and $\|f\|_2 = \|\bm\cbasis\|_2$.

\section{Posterior concentration on Hellinger balls}
\label{sec:sm:hell-contr}

The starting point to the proof of the concentration of the posterior in strong
distances is to first establish the contraction of the posterior for
$\pmes_{\bm\cbasis}$ on small Hellinger neighborhoods of $\pmes_{\L}$. We
will prove that for
$\minimaxrate \coloneqq (\log(n)/n)^{\smooth/(2\smooth + 1)}$, the
spike-and-slab log-density prior satisfies for $M > 0$ large enough,
\begin{equation}
  \label{eq:sm:1}
  \sup_{\L \in \dclass(R,\smooth)}\EE_{\L}\Pi\big(\bm\cbasis\,:\, \He(\pmes_{\bm\cbasis},\pmes_{\L}) > M\minimaxrate  \mid
  \obs \big) = o(1),\qquad n\rightarrow \infty.
\end{equation}
This will prove the \cref{main-thm:hellinger-adapt}. One can
notice that $\minimaxrate$ is not the optimal rate for the Hellinger distance.
This is a well-known consequence of the fact that prior independence of the
wavelets coefficients cannot yield optimal contraction on Hellinger or $\LP{2}$
balls \citep{hoffmann2015adaptive,cai2008information}.

\smallskip%
\par \Cref{eq:sm:1} is obtained classically, as a consequence of
\citet[Theorem~2.1]{ghosal-ghosh-vaart-2000-conver} combined with
\cref{pro:sm:3,pro:sm:4} below.

\begin{proposition}
  \label{pro:sm:4}
  Let $\Pi$ be the spike-and-slab prior described in
  \cref{main-sec:spike-slab-log}. Assume $\L \in \dclass(R,\smooth)$ for some
  $0 < \smooth_0 \leq \smooth \leq S$ and let
  $\minimaxrate \coloneqq (\log(n)/n)^{\smooth/(2\smooth + 1)}$. Then, there
  exists $C > 0$ such that for $n$ large enough,
  \begin{equation}
    \Pi\big(\KL(\pmes_{\L},\pmes_{\bm\cbasis}) \leq \minimaxrate^2,\,
    \VKL(\pmes_{\L},\pmes_{\bm\cbasis}) \leq \minimaxrate^2\big)%
    \geq \exp\{-C n\minimaxrate^2\}.
  \end{equation}
\end{proposition}
\begin{proof}
  Let $\bm\cbasis$ be such that
  $|\cbasis_{j,k} - \cbasis_{j,k}^{\L}| \leq n^{-1/2}$ if $j \leq \tilde{J}$ and
  $\cbasis_{j,k} = 0$ for $g > \tilde{J}$, with $\tilde{J}$ a truncation level
  to be chosen accordingly. Because
  $|\cbasis_{j,k}^{\L}| \leq R2^{-j(\smooth+1/2)}$ by assumption, this implies that
  $\|\bm\cbasis - \bm\cbasis^{\L}\|_{1,\infty} \lesssim \sum_{j \leq
    \tilde{J}} 2^{j/2}n^{-1/2} + \sum_{j > \tilde{J}}2^{-j\smooth} \lesssim
  2^{\tilde{J}/2}n^{-1/2} + 2^{-\tilde{J}\smooth} \lesssim \minimaxrate$ by
  choosing the optimal truncation level. By \citet[Lemma~3.1]{van2008rates},
  this implies that
  $\KL(\pmes_{\L},\pmes_{\bm\cbasis}) \lesssim \minimaxrate^2$ and
  $\VKL(\pmes_{\L},\pmes_{\bm\cbasis}) \lesssim \minimaxrate^2$. Thus, for
  some $B > 0$,
  \begin{multline}
    \Pi\big(\KL(\pmes_{\L},\pmes_{\bm\cbasis}) \leq B\minimaxrate^2,\,
    \VKL(\pmes_{\L},\bm\cbasis) \leq B\minimaxrate^2\big)\\%
    \begin{aligned}
      &\geq \prod_{j \leq \tilde{J}}\prod_k \omega_j
      \Pi_{j,k}\big(|\cbasis_{j,k} - \cbasis_{j,k}^{\L}| \leq n^{-1/2}
      \big)\prod_{j = \tilde{J}}^{J_n}\prod_k(1 - \omega_j)\\
      &\geq \prod_{j \leq \tilde{J}}\prod_k \omega_j%
      2^{j(\smooth_0 + 1/2)}\int_{\cbasis_{j,k}^{\L} -
        n^{-1/2}}^{\cbasis_{j,k}^{\L}+n^{-1/2}}f(2^{j(\smooth_0 + 1/2)}t)\,\intd t%
      \prod_{j = \tilde{J}}^{J_n}\prod_k(1 - \omega_j)\\
      &= \prod_{j \leq \tilde{J}}\prod_k \omega_j%
      \int_{2^{j(\smooth_0+1/2)}(\cbasis_{j,k}^{\L} -
        n^{-1/2})}^{2^{j(\smooth_0+1/2)}(\cbasis_{j,k}^{\L}+n^{-1/2})}f(t)\,\intd t%
      \prod_{j = \tilde{J}}^{J_n}\prod_k(1 - \omega_j).
    \end{aligned}
  \end{multline}
  Since $\smooth \geq \smooth_0$ and $ j \leq \tilde{J}$, there is a constant $g >
  0$ such that $f \geq g$ always on the domain of integration of the previous
  display. Moreover, $\omega_j \gtrsim e^{-j b_1}$ and there are no more
  than a generic constant times $2^j$ wavelets at each level $j$, then for some
  constants $C > 0$,
  \begin{equation}
    \Pi\big(\KL(\pmes_{\L},\pmes_{\bm\cbasis}) \leq B\minimaxrate^2,\,
    \VKL(\pmes_{\L},\pmes_{\bm\cbasis}) \leq B\minimaxrate^2\big)%
    \geq \exp\{- C 2^{\tilde{J}}\log(n) \}.%%
  \end{equation}
  The conclusion follows since $2^{\tilde{J}}\log(n) \lesssim n \minimaxrate^2$.
\end{proof}

\begin{proposition}
  \label{pro:sm:3}
  Let $\Pi$ be the spike-and-slab prior described in
  \cref{main-sec:spike-slab-log}. Let $n\refrate^2 \rightarrow \infty$ with
  $\log(n\refrate^2) \gtrsim \log(n)$ and let $n$ be large enough. Then for
  every $C > 0$ there exists a sequence of sets $(\Theta_n)_{n\geq 0}$ such that
  $\Pi(\Theta_n) \lesssim \exp\{-Cn\refrate^2)$ and
  $\log N(\refrate,\mathcal{P}_n,\He) \leq n\refrate^2$, where
  $\mathcal{P}_n = \Set{p_{\bm\cbasis} \given \bm\cbasis \in \Theta_n}$.
\end{proposition}
\begin{proof}
  We choose, for some constant $K,K' > 0$,
  \begin{equation}
    \Theta_n \coloneqq \Set*{\bm\cbasis \in \Theta \given%
      \begin{array}{c}
        \sup_{j,k}2^{j(\smooth_0+1/2)}|\cbasis_{j,k}| \leq \log(Kn\refrate^2),\\
        |\Set{\cbasis_{j,k} \ne 0}| \leq K'n
        \refrate^2 / \log(n\refrate^2)
      \end{array}
      }.
  \end{equation}
  Then, by construction of the prior since
  $\cbasis_{j,k} = 2^{-j(\smooth_0 +1/2)} Z_{j,k}$ for $Z_{j,k} \sim F$, and
  since all the coefficients are independent,
  \begin{align}
    \Pi(\Theta_n^c)%
    &\leq \sum_{j,k} \omega_{j}\Pi\big( |Z_{j,k}| > \log(Kn\refrate^2)\big)%
      + \Pi\Big( |\Set{\cbasis_{j,k} \ne 0}| >
      \frac{K'n\refrate^2}{\log(n\refrate^2)}\Big)\\
    &\leq e^{-b_2Kn\refrate^2} + \Pi\Big( |\Set{\cbasis_{j,k} \ne 0}| >
      \frac{K'n\refrate^2}{\log(n\refrate^2)}\Big),
  \end{align}
  where the second line follows because $\omega_j \leq 2^{-j(1+\mustar)}$ for
  $\mustar > 0$ and because there are at most a generic constant times $2^j$
  wavelets at each level $j \in \NNInts$, and because of the assumption of
  \cref{main-eq:39}. Furthermore,
  $\EE_{\Pi}[|\Set{\cbasis_{j,k} \ne 0}|] \leq \sum_{j=0}^{J_n}\sum_k
  2^{-j(1+\mustar)} \lesssim 1$, Hence, by Chernoff's bound, for some constant
  $B > 0$ when $K'n\refrate^2$ gets large enough
  \begin{equation}
    \Pi\Big( |\Set{\cbasis_{j,k} \ne 0}| > \frac{K'n
    \refrate^2}{\log(n\refrate^2)} \Big)%
    \leq e^{-BK' n\refrate^2}.
  \end{equation}
  Thus $\Theta_n$ meets the first requirement of the proposition.

  \smallskip%
  \par We now determine an upper bound on $N(\refrate, \mathcal{P}_n, \He)$. We
  assume without loss of generality that
  $M_n \coloneqq K' n \refrate^2/ \log(n\refrate^2)$ is integer. Furthermore
  there is a generic constant $c > 0$ such that, for any
  $\bm\cbasis,\bm\cbasis' \in \Theta_n$
  \citep[see][Lemma~3.1]{van2008rates},
  \begin{align}
    \He(P_{\bm\cbasis},P_{\bm\cbasis'})%
    &\lesssim \|\bm\cbasis - \bm\cbasis'\|_{1,\infty}e^{-c \|\bm\cbasis
    - \bm\cbasis'\|_{1,\infty}}.
  \end{align}
  But, for any $\bm\cbasis,\bm\cbasis' \in \Theta_n$, writing
  $Z_{j,k} = 2^{j(\smooth_0+1/2)}\cbasis_{j,k}$ and
  $Z_{j,k}' = 2^{j(\smooth_0+1/2)}\cbasis_{j,k}'$, it is clear that
  $\|\bm\cbasis - \bm\cbasis'\|_{1,\infty} \lesssim \|\bm{Z} -
  \bm{Z}'\|_{\infty}$. Since $\bm\cbasis$ has no more than $M_n$ non-zero
  entries, $N(\refrate, \mathcal{P}_n, \He)$ is no more than  the sum over all
  possible subsets of  indices
  $I \subseteq A_{n} \coloneqq \Set{(j,k)\given j \leq \frac{\log(n)}{\log(2)}}$ such that
  $|I| \leq M_{n}$ of the covering numbers of
  \begin{equation}
    \label{eq:28}
    E_{I} \coloneqq%
    \Set*{\bm{z} \in \Reals^I \given \|\bm{z}\|_{\infty} \leq \log(Kn\refrate^{2})}
  \end{equation}
  with balls of radius $B\refrate$, for a universal $B > 0$. A $B\varepsilon_{n}$-net over
  $E_{I}$ has cardinality no more than
  \begin{equation}
    \label{eq:29}
    \Big(\frac{\log(Kn\refrate^{2})}{B\refrate}  \Big)^{|I|}%%
    \leq \Big(\frac{\log(Kn\refrate^{2})}{B\refrate}  \Big)^{M_{n}} \lesssim e^{C_{1}K'n\refrate^{2}},
  \end{equation}
  for some universal $C_{1} > 0$. The number of possible subsets is
  $\sum_{m=0}^{M_{n}}\binom{|A_{n}|}{m}$. Using the well-known inequality
  $\binom{n}{k} \leq \frac{n^{k}}{k!}$, we deduce that,
  \begin{equation}
    \label{eq:30}
    \sum_{m=0}^{M_{n}}\binom{|A_{n}|}{m} \leq \sum_{m=0}^{M_{n}}\frac{|A_{n}|^{m}}{m!}%
    = \frac{e^{|A_{n}|}}{M_{n}!}\int_{|A_{n}|}^{\infty}t^{M_{n}}e^{-t}\intd t.
  \end{equation}
  By Stirling's formula, as $n \to \infty$, because $|A_{n}| \gg M_{n}$,
  \begin{align}
    \sum_{m=0}^{M_{n}}\binom{|A_{n}|}{m}%
    &\leq \frac{(1+o(1))e^{|A_{n}| + M_{n} - M_{n}\log M_{n}}}{\sqrt{2\pi M_{n}}}%
      \int_{|A_{n}|}^{\infty}t^{M_{n}}e^{-t} \intd t\\
    &= \frac{(1+o(1))|A_{n}|e^{|A_{n}| + M_{n} - M_{n}\log M_{n} + M_{n}\log|A_{n}|}}{\sqrt{2\pi M_{n}}}%
      \int_{1}^{\infty}e^{-u|A_{n}| + M_{n}\log(u)} \intd u\\
    &\leq \frac{(1+o(1))|A_{n}|e^{|A_{n}| + M_{n} - M_{n}\log M_{n} + M_{n}\log|A_{n}|}}{\sqrt{2\pi M_{n}}}%
      \int_{1}^{\infty}e^{-u|A_{n}|(1 - \frac{M_{n}}{e|A_{n}|})}\intd u\\
    &= \frac{(1+o(1))e^{M_{n}(1+e^{-1}) - M_{n}\log M_{n} + M_{n}\log|A_{n}|}}{\sqrt{2\pi M_{n}}(1 - \frac{M_{n}}{e|A_{n}|})}.
  \end{align}
  Since $M_{n} \asymp \frac{K'n\refrate^{2}}{\log(n\refrate^{2})}$ and $|A_{n}| \lesssim n$,
  it follows for some universal constant $C_{2} > 0$ that
  $\sum_{m=0}^{M_{n}}\binom{|A_{n}|}{m} \lesssim e^{C_{2}K'n\refrate^{2}}$. The
  conclusion of the proposition follows by picking $K'$ small enough.
\end{proof}

\section{Posterior concentration on $L^2$ balls and uniform consistency}
\label{sec:sm:uniform-consistency}

\par Here we strengthen a little bit the result of \cref{sec:sm:hell-contr} and we
show that the posterior indeed concentrates on $\|\cdot\|_2$ and $\|\cdot\|_{1,\infty}$
neighborhoods of $\bm\cbasis^{\L}$. Indeed, since the coefficients
$\bm\cbasis^{\L}$ is only identifiable up to suitable translation, we don't
expect to concentrates on balls of the form
$\Set{\bm\cbasis \given \|\bm\cbasis - \bm\cbasis^{\L}\| \leq R}$, but instead on
balls of the form
$\Set{\bm\cbasis \given \|\bm\cbasis + \bm\Xi^{\bm\cbasis} - \bm\cbasis^{\L}\| \leq R}$,
where $\bm\Xi^{\bm\cbasis}$ is the vector with entries
$\Xi^{\bm\cbasis}_{j,k} \coloneqq \Inner{\xi^{\bm\cbasis},\basis_{j,k}}$, where
$\xi^{\bm\cbasis}$ is the log-normalizing constant defined in
\cref{main-sec:prior-definition}. We remark that by the properties of the
wavelet basis, there is a $J_0$ such that $\Xi_{j,k} = 0$ for all $j \geq J_0$ and
all $k$. We also prove that concentration on $\|\cdot\|_{1,\infty}$ balls imply the
posterior concentration on $\consistencyset$ as required in the main paper.%

\smallskip%
\par The proof follows a minor adaptation of
\citet[Lemma~4]{castillo2014bayesian}, itself inspired from ideas of
\citet{rivoirard2012}. The argument, however, requires to be adapted to handle
the fact that there is a non zero prior mass of having coefficients
$\cbasis_{j,k} \ne 0$ with $j > \tilde{J}$, with $\tilde{J}$ being the optimal
truncation level.

\begin{proposition}
  \label{pro:sm:1}
  Let $\Pi$ be the spike-and-slab prior described in
  \cref{main-sec:spike-slab-log}. Then, there is a generic constant $C > 0$ such
  that,
  \begin{equation}
    \sup_{J \geq 0}\EE_{\L}\Pi\Big(
    \textstyle\sum_{j > J}\max_k|\cbasis_{j,k}|2^{j/2}
    > C2^{-J\smooth_0}\log(n) \mid \obs\Big)= o(1).
  \end{equation}
\end{proposition}
\begin{proof}
  By construction, $\cbasis_{j,k} = 2^{-j(\smooth_0 + 1/2)}Z_{j,k}$, where
  $Z_{j,k} \iid F$. Then, we see that
  $\sum_{j > J}\max_k|\cbasis_{j,k}|2^{j/2} = \sum_{j>
    J}2^{-j\smooth_0}\max_k|Z_{j,k}| \lesssim 2^{-J
    \smooth_0}\sup_{\klm}|Z_{\klm}|$. Hence, the probability in the statement of
  the proposition is bounded by
  $\EE_{\L}\Pi(\sup_{\klm}|Z_{\klm}| \gtrsim \log(n) \mid \obs)$. But, by the
  assumptions on $F$, we have that for any $K > 0$%
  \begin{align}
    \Pi( \textstyle\sup_{\klm}|Z_{\klm}| > \log(Kn\minimaxrate^2) )%
    &\leq \sum_{j\geq 0}\sum_k\Pi( |Z_{j,k}| > \log(Kn\minimaxrate^2) )\\
    &= \sum_{j\geq 0}\sum_k \omega_j\Pi( |Z_{j,k}| > \log(Kn\minimaxrate^2) \mid
      \cbasis_{j,k} \ne 0)\\
    &\lesssim \exp\{-  b_2 Kn\minimaxrate^2\},
  \end{align}%
  where the last line follows because $\omega_j \leq 2^{-j(1+\mustar)}$ for
  $\mustar > 0$ by assumption, and there are at most a generic constant times
  $2^j$ wavelets at each level $j$. By \citet[see the proof of
  Theorem~2.1]{ghosal-ghosh-vaart-2000-conver} the last display together with
  \cref{pro:sm:4} implies that the proposition is true, because
  $\log(n\minimaxrate^2) \asymp \log(n)$ by definition of $\minimaxrate$.
\end{proof}

We are now in position to adapt \citet[Lemma~4]{castillo2014bayesian} to our
setting. The idea is to leverage the property that the high frequency
coefficients are always small enough (\cref{pro:sm:1}) to obtain that the Hellinger
contraction implies the desired result.
\begin{proposition}
  \label{pro:sm:2}
  Let $\Pi$ be the spike-and-slab prior described in
  \cref{main-sec:spike-slab-log} and assume $0 < \smooth_0 \leq \smooth \leq
  S$. Then, under the assumption of the paper, for all $R > 0$ the following
  holds,
  \begin{gather}
    \exists K > 0,\quad%
    \sup_{\L \in \dclass(R,\smooth)}\EE_{\L}\Pi(\|\bm\cbasis +
    \bm\Xi^{\bm\cbasis} - \bm\cbasis^{\L}\|_2
    > K\minimaxrate \mid \obs) =o(1),\\
    \forall \eta > 0,\quad%
    \sup_{\L \in \dclass(R,\smooth)}\EE_{\L}\Pi(\|\bm\cbasis +
    \bm\Xi^{\bm\cbasis} - \bm\cbasis^{\L}\|_{1,\infty} > \eta \mid \obs) =
    o(1).
  \end{gather}
\end{proposition}
\begin{proof}
  From \citet[Lemma~8]{ghosal2007b}, because $\dens_{\L}$ is bounded from below (as
  $\|\L\|_{\infty} < \infty$ by assumption), we have that
  \begin{align}
    \notag
    \|\L_{\bm\cbasis} - \L\|_2^2%
    &= \int_{\domain}\log^2(\dens_{\bm\cbasis}/\dens_{\L})\\
    \notag
    &\lesssim \int_{\domain}\dens_{\L}\log^2(\dens_{\bm\cbasis}/\dens_{\L})\\
    \label{eq:sm:10}
    &\lesssim \He(\pmes_{\bm\cbasis}, \pmes_{\L})^2\big(1 +
    \|\L_{\bm\cbasis} - \L\|_{\infty}^2 \big).
  \end{align}
  But, for any $J \in \Nats$ large enough so that $\Xi_{j,k} = 0$ for all $j > J$
  and all $k$, we have
  \begin{align}
    \label{eq:2}
    \|\L_{\bm\cbasis} - \L\|_{\infty}%
    &\lesssim%
      \sum_{j=0}^J \max_k|\cbasis_{j,k} +
      \Xi^{\bm\cbasis}_{\klm} - \cbasis^{\L}_{j,k}|2^{j/2}\\
    &\quad%
      +\sum_{j > J}\max_k|\cbasis_{j,k}|2^{j/2}%
      + \sum_{j > J}\max_k|\cbasis_{j,k}^{\L}|2^{j/2}
  \end{align}
  By assumption
  $\sum_{j > J}\max_k|\cbasis_{j,k}^{\L}|2^{j/2} \lesssim 2^{-J\smooth}$, and by
  \cref{pro:sm:1} there is a set $\Theta_n$ of posterior mass $1 + o_p(1)$ such
  that for any $\bm\cbasis \in \Theta_n$ we have
  $\sum_{j > J}\max_k|\cbasis_{j,k}|2^{j/2} \lesssim
  2^{-J\smooth_0}\log(n)$. Hence, whenever $\smooth \geq \smooth_0$, for
  $\bm\cbasis \in \Theta_n$,
  \begin{align}
    \|\L_{\bm\cbasis} - \L\|_{\infty}%
    &\lesssim%
      \sum_{j=0}^J \max_k|\cbasis_{j,k} +
      \Xi^{\bm\cbasis}_{\klm} - \cbasis^{\L}_{j,k}|2^{j/2}%
      + 2^{-J\smooth_0}\log(n) + 2^{-J\smooth}\\
    &\lesssim 2^{J/2}\|\bm\cbasis + \bm\Xi^{\bm\cbasis} -
      \bm\cbasis^{\L}\|_2%
      + 2^{-J\smooth_0}\log(n)\\
    &=2^{J/2}\|\L_{\bm\cbasis} - \L\|_2%
      + 2^{-J\smooth_0}\log(n),
  \end{align}
  where the last line follows by the orthogonality of the wavelet
  basis. Combining the last display with \cref{eq:sm:10} we find that
  \begin{equation}
    \|\L_{\bm\cbasis} - \L\|_2^2\big(1 -
    2^J\He(\pmes_{\bm\cbasis},\pmes_{\L})^2 \big)%
    \lesssim \He(\pmes_{\bm\cbasis},\pmes_{\L})^2\big(1 +
    2^{-2J\smooth_0}\log^2(n) \big).
  \end{equation}
  By \cref{eq:sm:1}, we can furthermore restrict ourselves to the event such
  that
  $\Set{\bm\cbasis \given \He(\pmes_{\bm\cbasis},\pmes_{\L}) \lesssim
    \minimaxrate}$. Then, it is always possible to choose $J$ sufficiently large
  so that $2^{-J\smooth_0}\log(n) = o(1)$, but small enough so that
  $2^J\minimaxrate^2 = o(1)$. Thus
  $\|\L_{\bm\cbasis} - \L\|_2^2 \lesssim \minimaxrate^2$ on
  $\Set{\bm\cbasis \given \He(\pmes_{\bm\cbasis},\pmes_{\L}) \lesssim
    \minimaxrate}$, and by orthogonality of the wavelet basis
  $\|\bm\cbasis + \bm\Xi^{\bm\cbasis} - \bm\cbasis^{\L}\|_2 =
  \|\L_{\bm\cbasis} - \L\|_2 \lesssim \minimaxrate$. Moreover,
  we have proven along the way that on the same event
  $\|\bm\cbasis + \bm\Xi^{\bm\cbasis} - \bm\cbasis^{\L}\|_{1,\infty} \lesssim
  2^{J/2}\|\L_{\bm\cbasis} - \L\|_2 + 2^{-J\smooth_0}\log(n) =
  o(1)$.%
\end{proof}

\section{Missing proofs of the main document}
\label{sec:missing-proofs-main}

\subsection{Concentration of the posterior on $\consistencyset$}
\label{sec:sm:proof-lem:consistency}

\begin{proof}[Proof of \cref{main-lem:consistencyset}]
We prove the lemma by showing that
$\|F^{\bm\cbasis}_0\|_{\infty} + \|\sum_{m\geq 1}F^{\bm\cbasis}_m\|_{\infty} \lesssim \|\bm\cbasis + \bm\Xi^{\bm\cbasis} - \bm\cbasis^{\L}\|_{1,\infty}$.
Then, the conclusion of the lemma will follow from \cref{pro:sm:2}. We note that
it is enough to show the inequality for
$ \|\bm\cbasis + \bm\Xi^{\bm\cbasis} - \bm\cbasis^{\L}\|_{1,\infty} \ll 1$ because of
\cref{pro:sm:2}. For $m \geq 1$, we have
$F_m^{\bm\cbasis} = (\cbasis_{\psi(m)} - \cbasis^{\L}_{\psi(m)})(\basis_{\psi(m)} - \EE_{\L}[\basis_{\psi(m)}]) = (\cbasis_{\psi(m)} + \Xi^{\bm\cbasis}_{\psi(m)} - \cbasis_{\psi(m)}^{\L})(\basis_{\psi(m)} - \EE_{\L}[\basis_{\psi(m)}])$,
by choosing $J_0$ sufficiently large so that $\Xi^{\bm\cbasis}_{\psi(m)} = 0$
whenever $m \ne 0$. It follows by
\cref{main-pro:coeff:new}-\eqref{main-item:norm:2},%
\begin{align}
  \label{eq:4}
  \Big\|\sum_{m \geq 1}F_m^{\bm\cbasis}\Big\|_{\infty}%
  &\leq \sup_x\sum_{m\geq 1}|\cbasis_{\psi(m)}^{\bm\cbasis} +
    \Xi^{\bm\cbasis}_{\psi(m)} - \cbasis^{\L}_{\psi(m)}||\basis_{\psi(m)}(x) -
    \EE_{\L}[\basis_{\psi(m)}]|\\
  &\leq \sum_{j > J_0} \max_k|\cbasis_{j,k}^{\bm\cbasis} +
    \Xi^{\bm\cbasis}_{j,k} - \cbasis^{\L}_{j,k}|
    \sup_x\sum_k|\basis_{j,k}(x)- \EE_{\L}[\basis_{j,k}]|\\
  &\lesssim \sum_{j > J_0} \max_k|\cbasis_{j,k}^{\bm\cbasis} +
    \Xi^{\bm\cbasis}_{j,k} - \cbasis^{\L}_{j,k}|2^{j/2}\\
  &\leq \|\bm\cbasis + \bm\Xi^{\bm\cbasis} - \bm\cbasis^{\L}\|_{1,\infty}.
\end{align}
So it remains to show that $\|F^{\bm\cbasis}_0\|_{\infty}$ is also bounded by the
same quantity. Remark that $\EE_{\L}[\exp\{ \L_{\bm\cbasis} - \L\}] = 1$, thus
\begin{align}
  \label{eq:1}
  1%
  &= \EE_{\L}[e^{\L_{\bm\cbasis} -\L }]%\\
  = 1 + \EE_{\L}[\L_{\bm\cbasis} - \L]%
  + \frac{1}{2}\EE_{\L}[(\L_{\bm\cbasis} - \L)^2]\cdot O\big(e^{\|\L_{\bm\cbasis} - \L\|_{\infty}}\big).
\end{align}
But
$\EE_{\L}[\L_{\bm\cbasis} - \L] = \xi^{\bm\cbasis} + \sum_{(j,k) \in \kset}(\cbasis_{j,k} - \cbasis^{\L}_{j,k})\EE_{\L}[\basis_{j,k}]$
and
$\|\L_{\bm\cbasis} - \L\|_{\infty} \lesssim \|\bm\cbasis + \bm\Xi^{\bm\cbasis} -\bm\cbasis^{\L}\|_{1,\infty} \ll 1$
by assumption. Therefore, we have shown
\begin{align}
  \label{eq:8}
  \xi^{\bm\cbasis}%
  &= - \sum_{(j,k) \in \kset}(\cbasis_{j,k} -
    \cbasis^{\L}_{j,k})\EE_{\L}[\basis_{j,k}]%
    + O\big(\|\L_{\bm\cbasis} - \L\|_2^2\big).
\end{align}
Since for $J_0$ taken large enough we have
$\xi^{\bm\cbasis} = \sum_{(j,k)\in B_0}\Xi_{j,k}\basis_{j,k}$, we deduce that
\begin{align}
  \label{eq:11}
  F_0%
  &= \sum_{(j,k)\in B_0}(\cbasis_{j,k} -
    \cbasis^{\L}_{j,k})(\basis_{j,k} - \EE_{\L}[\basis_{j,k}])\\
  &= \xi^{\bm\cbasis} + \sum_{(j,k) \in B_0}(\cbasis_{j,k} -
    \cbasis^{\L}_{j,k})\basis_{j,k}%
    + \sum_{(j,k) \notin B_0}(\cbasis_{j,k} -
    \cbasis^{\L}_{j,k})\EE_{\L}[\basis_{j,k}] +O\big(\|\L_{\bm\cbasis} - \L\|_2^2 \big)\\
  &= \sum_{(j,k) \in B_0}(\cbasis_{j,k} + \Xi_{j,k}^{\bm\cbasis} -
    \cbasis^{\L}_{j,k})\basis_{j,k}%
    + \sum_{(j,k) \notin B_0}(\cbasis_{j,k} + \Xi^{\bm\cbasis}_{j,k} -
    \cbasis^{\L}_{j,k})\EE_{\L}[\basis_{j,k}] + O(\|\L_{\bm\cbasis} -
    \L\|_2^2).
\end{align}
Now we remark that all the terms involved in the last display are bounded in
absolute value by $\|\bm\cbasis + \bm\Xi^{\bm\cbasis} - \bm\cbasis^{\L}\|_{1,\infty}$,
and so is $\|F_0\|_{\infty}$.
\end{proof}

\subsection{Proofs related to the control of the covariance terms}
\label{sec:relat-contr-covar}

\begin{proof}[Proof of \texorpdfstring{\cref{main-lem:4:new}}{Lemma \ref{main-lem:4:new}}]

  We first establish that $\|\SPFC\|_{\infty}$ can be made arbitrary small when
$\PFC \in \mathcal{N}_I$ by taking $K_1$ large enough in \cref{main-ass:1}.

\begin{proposition}
  \label{pro:covctrl:2}
  Suppose \cref{main-ass:1} is satisfied with $K_1 > 0$ sufficiently large. Then
  there is a universal constant $C > 0$ such that
  $\|\SPFC\|_{\infty} \leq C \log(\Gamma\minimaxrate/\gamma)\Gamma \minimaxrate$ for all $\PFC \in \mathcal{N}_I$ and all
  $I \subseteq \NNInts$.
\end{proposition}

In view of the last proposition, we now assume that $K_1$ is taken large enough
so that $\|\SPFC\|_{\infty} \ll \delta$. Since
$\EE_{\L}[\SPF] = \EE_{\L}[\SPFC] = 0$, we have
\begin{equation}
  \covlog(\SPF,\SPFC) = \log \EE_{\L}[e^{\SPF - \log \EE_{\L}[\exp(\SPF)]}e^{\SPFC - \log \EE_{\L}[\exp(\SPFC)]} ].
\end{equation}
We remark that
$\EE_{\L}[e^{\SPF}] = \EE_{\L}[1 + \SPF +
\frac{1}{2}\SPF^2e^{O(\|\SPF\|_{\infty})} ] = 1 +
\frac{1}{2}\EE_{\L}[\SPF^2]e^{O(\delta)}$.  This implies that
$-\log \EE_{\L}[e^{\SPF}] = - \frac{1}{2}\EE_{\L}[\SPF^2]e^{O(\delta)}$ as well,
and
$e^{\SPF - \log \EE_{\L}[\exp(\SPF)]} = 1+ \SPF -
\frac{1}{2}\EE_{\L}[\SPF^2]e^{O(\delta)} + \frac{1}{2}\SPF^2e^{O(\delta)}$. Also
$\EE_{\L}[e^{\SPFC - \log\EE_{\L}[\exp(\SPFC)]}]=1$, and thus
\begin{align}
  \label{eq:covctrl:dr:13}
  e^{\covlog(\SPF,\SPFC)}%
  &= \EE_{\L}\Big[\Big(1 + \SPF + \frac{1}{2}\SPF^2e^{O(\delta)} -
    \frac{1}{2}\EE_{\L}[\SPF^2]e^{O(\delta)}\Big)e^{\SPFC - \log \EE_{\L}[\exp(\SPFC)]} \Big]\\
  &=1 + \EE_{\L}[\SPF e^{\SPFC - \log \EE_{\L}[\exp(\SPFC)]}]%
    + \frac{1}{2}\EE_{\L}[\SPF^2]e^{O(\delta)}%
    - \frac{1}{2}\EE_{\L}[\SPF^2]e^{O(\delta)}\\
  &= 1 + \EE_{\L}[\SPF e^{\SPFC - \log \EE_{\L}[\exp(\SPFC)]}]%%
    + O( \delta \EE_{\L}[\SPF^2]).
\end{align}
Since $\|\SPFC\|_{\infty}\ll \delta$,
\begin{align}
  \label{eq:covctrl:146}
  \EE_{\L}[\SPF e^{\SPFC - \log \EE_{\L}[\exp(\SPFC)]}]^2%
  &\leq \|\SPF\|_{\infty}e^{O(\delta)}\Big|\EE_{\L}[\SPF e^{\SPFC - \log \EE_{\L}[\exp(\SPFC)]}] \Big|.
\end{align}
Therefore,
\begin{align}
  \label{eq:covctrl:18}
  \covlog(\SPF,\SPFC)
  &=(1 + O(\delta))\EE_{\L}[\SPF e^{\SPFC - \log \EE_{\L}[\exp(\SPFC)]}]
    + O(\delta \EE_{\L}[\SPF^2])\\
  &=(1 + O(\delta))\frac{\EE_{\L}[\SPF (e^{\SPFC} - 1)]}{\EE_L[e^{\SPFC}]}
    + O(\delta \EE_{\L}[\SPF^2])\\
  &= O\big(\EE_{\L}[\SPF (e^{\SPFC} - 1)]\big) +  O(\delta \EE_{\L}[\SPF^2]).
\end{align}
To prove the lemma, it is enough to obtain sharp estimates on
$\EE_L[\SPF(e^{\SPFC}-1)]$. To this end, we introduce the function
$A : \Reals \to \Reals$ such that $A(x) = (e^x - 1)/x$, and we define the
following ``covariance'' matrix
\begin{equation}
  \label{eq:covctrl:17}
  \Sigma_{j,k,j',k'}^{\PFC} \coloneqq%
  \EE_{\L}[(\basis_{j,k} - \EE_{\L}[\basis_{j,k}])(\basis_{j',k'} - \EE_{\L}[\basis_{j',k'}])A(\SPFC)].
\end{equation}
The following proposition provides the necessary estimates on
$\Sigma^{\PFC}$ that we require to control $\covlog(\SPF,\SPFC)$.

\begin{proposition}
  \label{pro:covctrl:7}
  Let $\|\SPFC\|_{\infty} \leq 1$ and let $\minsmooth = \smooth$ if $\smooth < 1$ and
  $\minsmooth$ arbitrary in $(1/2,1)$ if $\smooth \geq 1$. Then, the following are
  true for $n$ large enough (but not depending on $I$)
  \begin{enumerate}
    \item\label{item:cov2} There exists a
    constant $C > 0$ depending only on $(R,\smooth)$  and the wavelet
    basis, such that for all $(j',k') \ne (j,k)$,
    \begin{equation}
      \label{eq:covctrl:94}
      |\Sigma_{j,k,j',k'}^{\PFC}| \leq%
      C\big(1 + \|\SPFC\|_{\infty,\infty,\minsmooth}
      \big)\Big(%
      2^{-(j\vee    j')\minsmooth}
      \int|\basis_{j,k}\basis_{j',k'}|%
      + 2^{-j(\minsmooth +1/2)}2^{-j'(\minsmooth + 1/2)}\Big).
    \end{equation}
    \item\label{item:cov3} There exists a constant $C > 0$ depending only on $(R,\smooth)$
    and the wavelet basis, such that for any $(j',k') \in \kset$ and for any
    $j \leq j'$,
    $\sum_k|\Sigma_{j,k,j',k'}^{\PFC}| \leq C2^{j/2}(1 +
    \|\SPFC\|_{\infty,\infty,\minsmooth}) 2^{-j'(\minsmooth +
      1/2)}$.

    \item\label{item:cov4} Suppose \cref{main-ass:1} is satisfied for $K_0$ and $K_1$ large
    enough. Then $\|\SPFC\|_{\infty,\infty,\minsmooth} \leq \gamma$ for all $I \subseteq \NNInts$.%
  \end{enumerate}
\end{proposition}

It follows from the previous definitions that
$\EE_{\L}[\SPF(e^{\SPFC} - 1)] = \EE_{\L}[\SPF \SPFC A(\SPFC)]$. To ease the
notations, we define $f^m_{j,k} \coloneqq \Inner{F_m,\basis_{j,k}}$, so that
$F_m = \sum_{(j,k)\in B_m}f_{j,k}^m(\basis_{j,k} - \EE_{\L}[\basis_{j,k}])$, by
\cref{main-pro:coeff:new}-\eqref{main-item:coeff:1}. Hence, we can rewrite
\begin{align}
  \label{eq:covctrl:30}
  \EE_{\L}[\SPF(e^{\SPFC}-1)]%
  &= \sum_{m\in I^c}\sum_{(j,k)\in B_0}\sum_{(j',k')\in B_m}f_{j,k}^0f_{j',k'}^m\Sigma_{j,k,j',k'}^{\PFC}\1_{0\in I}\\%
  &\quad%
    +\sum_{m\in I}\sum_{(j,k)\in B_m}\sum_{(j',k')\in B_0}f_{j',k'}^0f_{j,k}^m\Sigma_{j,k,j',k'}^{\PFC}\1_{0\notin I}\\
  &\quad+
    \sum_{m\in I}\sum_{m'\in I^c}\sum_{(j,k)\in B_m}\sum_{(j',k')\in B_{m'}} f_{j,k}^mf_{j',k'}^{m'} \Sigma_{j,k,j',k'}^{\PFC}\1_{m\ne 0}\1_{m'\ne 0}.
\end{align}
We finish the proof by giving a bound on the terms
\begin{align}
  \label{eq:covctrl:61}
  R_1%
  &\coloneqq  \sum_{m\in I^c}\sum_{(j,k)\in B_0}\sum_{(j',k')\in B_m}|f_{j,k}^0||f_{j',k'}^m||\Sigma_{j,k,j',k'}^{\PFC}|\1_{0\in I},\\
  R_2%
  &\coloneqq\sum_{m\in I}\sum_{(j,k)\in B_m}\sum_{(j',k')\in B_0}|f_{j',k'}^0||f_{j,k}^m||\Sigma_{j,k,j',k'}^{\PFC}|\1_{0\notin I},\\
  R_3%
  &\coloneqq \sum_{m\in I}\sum_{m'\in I^c}\sum_{(j,k)\in B_m}\sum_{(j',k')\in B_{m'}} |f_{j,k}^m||f_{j',k'}^{m'}|| \Sigma_{j,k,j',k'}^{\PFC}|\1_{m\ne 0}\1_{m'\ne 0}.
\end{align}
We bound each of the terms above in the next paragraphs. For each term, the
bound depends on whether we are in the scenario $\smooth > 1/2$ or not. An
important quantity that shows up everywhere is
$f^{m,*} \coloneqq \sup_{(j,k)\in B_m}|\Inner{F_m,\basis_{j,k}}|$. We remark that
by construction when $m\notin I$%
\begin{align}
    \label{eq:covctrl:169}
  f^{m,*}
  &\coloneqq \sup_{(j,k)\in B_m}|\Inner{F_m,\basis_{\klm}}|%
    \leq  \|F_m\|_2%
    \leq H_I(m).
\end{align}

\paragraph*{Bound on $R_1$ : case $\smooth > 1/2$.}

If $0 \notin I$ this term is obviously
equal to zero. Thus we now assume that $0 \in I$. Then,
\begin{align}
  \label{eq:covctrl:7}
  R_1%
  &\leq f^{0,*} \sum_{m\in I^{c}} f^{m,*} \sum_{(j',k') \in B_{m}} \sum_{j=0}^{J_{0}} \sum_{k}|\Sigma_{j,k,j',k'}^{\PFC}|\\
  &\lesssim f^{0,*} (1\vee  \gamma )2^{J_{0}/2}\sum_{m\in I^{c}} f^{m,*} \sum_{(j',k')\in B_{m}} 2^{-j'(\minsmooth+1/2)}\\
  &\lesssim f^{0,*}(1\vee \gamma )2^{J_{0}/2}\sum_{m\in I^{c}} f^{m,*} 2^{-J_{m}(\minsmooth + 1/2)},
\end{align}
where the second line follows by \cref{pro:covctrl:7}, and the last
because $0\in I$ so for any $m \in I^c$ we have $|B_m| = 1$ and
$(j,k) \in B_{m} \implies j = J_m$. So by \cref{eq:covctrl:169},
\begin{align}
  \label{eq:covctrl:51}
  \sum_{m\in I^{c}} f^{m,*} 2^{-J_{m}(\minsmooth + 1/2)}%
  &\leq  \sum_{m\in I^c} H_I(m) 2^{-J_m(\minsmooth + 1/2)}\\
  &\leq \Gamma\xi^{-1} \sum_{j= J_0+1}^{j_n}\sup_{m\,:\, J_m=j}\{\rho_m\} |\Set{m\in I^c \given J_m = j}|  2^{-j(\minsmooth+1/2)}\\%
  &\qquad+ \sum_{j > j_n} \gamma 2^{-j(\smooth + \minsmooth +1)}|\Set{m\in I^c \given J_m = j}| \\
  &\lesssim \Gamma\xi^{-1} \sum_{j> J_{0}} \sup_{m\,:\, J_m =j}\{\rho_m\} 2^{-j(\minsmooth - 1/2)}%
    +\gamma 2^{-j_n(\smooth + \minsmooth)},
\end{align}
where the last line follows because $|\Set{m\in I^c \given J_m = j}| \lesssim 2^j$ for
$j > J_0$ as there are no more than $\lesssim 2^j$ at each level. We deduce that,
\begin{align}
  \label{eq:covctrl:14}
  R_1%
  &\lesssim f^{0,*}\Gamma\xi^{-1}(1\vee \gamma) 2^{J_{0}/2}\sum_{j > J_0}2^{-j(\minsmooth -1/2)}\sup_{m\,:\, J_{m}=j}\{\rho_{m}\} + f^{0,*}(1\vee \gamma)2^{J_0/2}\gamma 2^{-j_n(\smooth + \minsmooth)}\\
  &\lesssim \|F_0\|_2\sqrt{\frac{\log(n)}{n}} \cdot \frac{\Gamma(1\vee \gamma )}{\xi} 2^{J_{0}(1-\minsmooth)}%
    + \|F_0\|_2 \sqrt{\frac{\log(n)}{n}}\cdot 2^{J_0/2}\Gamma(1\vee \gamma) 2^{-j_n(\minsmooth -1/2)}
\end{align}
because $\rho_{m} = \sqrt{\log(n)/n}$, $\minsmooth > 1/2$, because
$f^{0,*} \leq \|F_0\|_2$, and because of \cref{main-eq:44}. By choosing
$K_1,K_2>0$ in \cref{main-ass:1} sufficiently large we obtain that
$R_1 \ll \delta \Gamma \sqrt{\log(n)/n}\|F_0\|_2\1_{0\in I}$.

\paragraph*{Bound on $R_1$ : case $0 < \smooth \leq 1/2$.}

The \cref{eq:covctrl:14} remains true in this case, but this time we have,
\begin{align}
  \label{eq:covctrl:100}
  \sum_{j>J_{0}}2^{-j(\minsmooth -1/2)}\sup_{m\,:\, J_{m}=j}\{\rho_{m}\}%
  &\leq \sum_{j> J_0} 2^{-j(\smooth-1/2)} \sup_{m\,:\, J_{m}=j}\big\{ 2^{-J_{m}/2} \minimaxrate \big\}\\
  &= \minimaxrate \sum_{j> J_0} 2^{-j\smooth}\\
  &\lesssim  \minimaxrate 2^{-J_0\smooth},
\end{align}
and, by \cref{main-eq:44},
\begin{align}
  \label{eq:covctrl:91}
  \gamma 2^{-j_n(\smooth + \minsmooth)}%
  &=\gamma 2^{-2j_n\smooth}%
  \lesssim \frac{\Gamma^2\minimaxrate^2}{\gamma}
\end{align}
Therefore,
\begin{align}
  \label{eq:covctrl:101}
  R_1%
  &\lesssim \|F_0\|_2 \Gamma 2^{-J_{0}/2}\minimaxrate\Big(\frac{1\vee \gamma}{\xi} 2^{J_{0}(1 - \smooth)}%
    + 2^{J_0}\frac{\Gamma(1\vee\gamma)}{\gamma} \minimaxrate\Big).
\end{align}
So $R_1 \ll \delta \Gamma 2^{-J_0/2}\minimaxrate\|F_0\|_2\1_{0\in I}$ by choosing $K_1,K_2>0$ in
\cref{main-ass:1} sufficiently large.

\paragraph*{Bound on $R_2$ : case $\smooth > 1/2$.}

$R_2$ is obviously equal to
zero of $0\in I$, hence we assume now that $0\notin I$. Then with the same
arguments as for the first term%
\begin{align}
  \label{eq:covctrl:49}
  R_2%
  &\leq f^{0,*} \sum_{m\in I} \sum_{(j,k) \in B_{m}}|f_{j,k}^m| \sum_{j'=0}^{J_{0}} \sum_{k'} |\Sigma_{j,k,j',k'}^{\PFC}|\\
    &\lesssim (1\vee \gamma )2^{J_{0}/2} f^{0,*} \sum_{m\in I}\sum_{(j,k)\in B_m} |f_{j,k}^m| 2^{-J_m(\minsmooth + 1/2)}\\
  \label{eq:covctrl:102}%
  &\lesssim (1 \vee \gamma )2^{J_{0}/2} H_I(0) \sum_{m\in I}\|F_m\|_2 2^{-J_m(\minsmooth + 1/2)},%%
\end{align}
where the second line follows by \cref{pro:covctrl:7}-\eqref{item:cov3} and because
$(j,k) \in B_m\implies j = J_{m}$ when $m\ne 0$ (recall $0\notin I$), the third
last line follows because $0 \notin I$ so $B_m$ is a singleton and also because
$f^{0,*} \leq H_I(0)$ when $0\notin I$. We decompose the sum in the rhs
\cref{eq:covctrl:102} into the sum of the following two terms
\begin{align}
  \label{eq:covctrl:106a}
  S_1&\coloneqq \sum_{m\in I} \|F_m\|_22^{-J_{m}(\minsmooth+1/2)}\1_{J_{m}\le j_{n}}\\
  \label{eq:covctrl:106b}
  S_2&\coloneqq \sum_{m\in I} \|F_m\|_2 2^{-J_{m}(\minsmooth+1/2)} \1_{J_{m}> j_{n}}.
\end{align}
Then, by Cauchy-Schwarz',
\begin{align}
  \label{eq:covctrl:117}
  S_1
  &\leq \Big\{ \sum_{m\in I} \|F_m\|_2^2\1_{J_{m}\leq j_{n}}\Big\}^{1/2}\Big\{\sum_{m\in I} 2^{-J_{m}(2\minsmooth+1)}\1_{J_{m}\le j_{n}} \Big\}^{1/2}.
\end{align}
But, since $0 \notin I$ we have $m \in I \Rightarrow J_{m} > J_{0}$, and since there are
$\lesssim 2^{j}$ wavelets at each level $j$, we deduce that
\begin{align}
  \label{eq:covctrl:85}
  \sum_{m\in I} 2^{-J_{m}(2\minsmooth+1)}%
  &\leq \sum _{j>J_0}|\Set{m\given J_m = j}| 2^{-j(2\minsmooth + 1)}%
    \lesssim \sum_{j>J_0}2^{-2j\minsmooth}%
    \lesssim 2^{-2J_{0}\minsmooth}.
\end{align}
Hence,
\begin{align}
  \label{eq:covctrl:118}
  S_1
  &\lesssim 2^{-J_0\minsmooth}\Big\{ \sum_{m\in I} \|F_{m}\|_2^2\1_{J_{m}\leq j_{n}}\Big\}^{1/2}.%
\end{align}
On the other hand,
\begin{align}
  \label{eq:covctrl:119}
  S_2
  &\leq \Big\{ \sum_{m\in I} \|F_m\|_2^2\1_{J_{m}> j_{n}}\Big\}^{1/2}\Big\{\sum_{m\in I} 2^{-J_{m}(2\minsmooth+1)}\1_{J_{m}> j_{n}} \Big\}^{1/2}\\
  &\lesssim \Big\{ \sum_{m\in I} \|F_m\|_2^2\1_{J_{m}> j_{n}}\Big\}^{1/2}\Big\{\sum_{j> j_{n}}2^{-2j\minsmooth} \Big\}^{1/2}\\
  &\lesssim 2^{-j_n\minsmooth}  \Big\{ \sum_{m\in I} \|F_m\|_2^2\1_{J_{m}> j_{n}}\Big\}^{1/2}.
\end{align}
Since $H_I(0) = \Gamma \sqrt{\log(n)/n}$ in this case,
\begin{multline}
  \label{eq:covctrl:161}
  R_2%
  \lesssim
    \Gamma(1 \vee \gamma)2^{J_{0}(1/2 - \minsmooth)} \sqrt{\frac{\log(n)}{n}}%
    \Big\{ \sum_{m\in I} \|F_{m}\|_2^2\1_{J_{m}\leq j_{n}}\Big\}^{1/2}\\%
  %&\quad%
    + \Gamma(1 \vee \gamma)2^{J_0/2}\cdot 2^{-j_n\minsmooth} \sqrt{\frac{\log(n)}{n}} \Big\{ \sum_{m\in I} \|F_m\|_2^2\1_{J_{m}> j_{n}}\Big\}^{1/2}.
\end{multline}
Note that in this case $\minsmooth > 1/2$, so by choosing the constants $K_0$
and $K_1$ sufficiently large, we obtain that (note that $\xi^{-\1_{0\in I}} = 1$ here)
\begin{equation}
  \label{eq:covctrl:136}
  R_2%
  \ll \delta \Gamma \xi^{-\1_{0\in I}}\sqrt{\frac{\log(n)}{n}} \Big\{ \sum_{m\in I} \|F_{m}\|_2^2\1_{J_{m}\leq j_{n}}\Big\}^{1/2}%
  + \frac{\delta}{\sqrt{n}}\Big\{ \sum_{m\in I} \|F_m\|_2^2\1_{J_{m}> j_{n}}\Big\}^{1/2}.
\end{equation}

\paragraph*{Bound on $R_2$ : case $0 < \smooth \leq 1/2$.}

We note that the equation \cref{eq:covctrl:102} and the decomposition of
\cref{eq:covctrl:106a,eq:covctrl:106b} remain true in this case. We then have,
\begin{align}
  \label{eq:covctrl:120}
  S_1
  &\leq \Big\{ \sum_{m\in I}\|F_m\|_2^2\1_{J_{m}\leq j_{n}} \Big\}^{1/2}%
    \Big\{ \sum_{m\in I}2^{-J_{m}(2\smooth + 1)}\1_{J_{m}\leq j_{n}} \Big\}^{1/2}
\end{align}
Since $0 \notin I$ we have that $m \in I \implies J_{m} \geq J_{1} > J_{0}$, so that the previous
is in fact bounded by
\begin{align}
  \label{eq:covctrl:121}
  S_1
  &\leq 2^{-J_{0}\smooth}\Big\{ \sum_{m\in I}\|F_m\|_2^2\1_{J_{m}\leq j_{n}} \Big\}^{1/2}%
    \Big\{ \sum_{m\in I}2^{-J_{m}}\1_{J_{m}\leq j_{n}} \Big\}^{1/2}.
\end{align}
On the other hand,
\begin{align}
  \label{eq:covctrl:122}
  S_2
  &\leq \Big\{\sum_{m\in I}\|F_m\|_2^2\1_{J_{m}> j_{n}}\Big\}^{1/2}%
    \Big\{\sum_{m\in I}2^{-J_{m}(2\smooth +1)}\1_{J_{m}>j_{n}}\Big\}^{1/2}.
\end{align}
Since $H_I(0) = \Gamma 2^{-J_0/2}\minimaxrate$ in this case,
\begin{multline}
  \label{eq:covctrl:171}
  R_2%
  \lesssim \Gamma (1 \vee \gamma)\cdot 2^{-J_0\smooth}\Big\{ \sum_{m\in I}\|F_m\|_2^2\1_{J_{m}\leq j_{n}}
  \Big\}^{1/2}%
  \Big\{\minimaxrate^2 \sum_{m\in I}2^{-J_{m}}\1_{J_{m}\leq j_{n}}
  \Big\}^{1/2}\\%
%  &\quad%
  + \Gamma (1 \vee \gamma) \minimaxrate\cdot \Big\{\sum_{m\in I}\|F_m\|_2^2\1_{J_{m}>
    j_{n}}\Big\}^{1/2}%
  \Big\{ \sum_{m\in I}2^{-J_{m}(2\smooth + 1)}\1_{J_{m} > j_{n}} \Big\}^{1/2}.
\end{multline}
So by choosing $K_0$ and $K_1$ large enough,
\begin{multline}
  \label{eq:covctrl:137}
  R_2 \ll%
  \delta \Gamma\xi^{-\1_{0\in I}} \Big\{ \sum_{m\in I}\|F_m\|_2^2\1_{J_{m}\leq j_{n}}
  \Big\}^{1/2}\Big\{\minimaxrate^2 \sum_{m\in I}2^{-J_{m}}\1_{J_{m}\leq j_{n}}
  \Big\}^{1/2}\\
  + \delta \gamma \Big\{\sum_{m\in I}\|F_m\|_2^2\1_{J_{m}>
    j_{n}}\Big\}^{1/2}%
  \Big\{ \sum_{m\in I}2^{-J_{m}(2\smooth + 1)}\1_{J_{m} > j_{n}} \Big\}^{1/2}.
\end{multline}

\paragraph*{Bound on $R_3$ : case $\smooth > 1/2$.}

% Regarding the third term of the rhs \cref{eq:covctrl:28}, we proceed as follows.
% \begin{multline}
%   \label{eq:covctrl:23}
%   \Big| \sum_{m\in I} \sum_{m'\in I^{c}}\EE_{\L}[F_{m}F_{m'}f(\rSPFC)]\1_{m\ne 0}\1_{m'\ne 0} \Big|\\%
%   \leq \sum_{m\in I} \sum_{\klm\in B_{m}} \sum_{m' \in I^{c}}\sum_{\klmalt \in B_{m'}}|f_{m,\klm}| |f_{m',\klmalt}| |\Sigma_{\klm,\klmalt}^{\PFC}|\1_{m\ne 0}\1_{m'\ne 0}
% \end{multline}

We start with the following estimate (using \cref{eq:covctrl:169})
\begin{multline}
  \label{eq:covctrl:126}
  \sum_{m'\in I^{c}}\sum_{(j',k')\in B_{m'}} |f_{j',k'}^{m'}| |\Sigma_{j,k,j',k'}^{\PFC}|\1_{m'\ne 0}\\%
  \begin{aligned}
      &\lesssim \sum_{j'' > J_0} \sum_{m'\in I^{c}}\1_{J_{m'}=j''} \sum_{(j',k') \in B_{m'}} |f_{j',k'}^{m'}|\cdot|\Sigma_{j,k,j',k'}^{\PFC}|\1_{m'\ne 0}\\
      &\leq \sum_{j''>J_0} \sup_{m'\,:\, J_{m'} = j''}\{ H_I(m') \}\sum_{m' \in I^{c}}\1_{J_{m'}=j''}\sum_{(j',k')\in B_{m'}}|\Sigma_{j,k,j',k'}^{\PFC}| \1_{m'\ne 0}\\
      &\leq \sum_{j'>J_0} \sup_{m'\,:\, J_{m'} = j'}\{ H_I(m') \} \sum_{k'}|\Sigma_{j,k,j',k'}^{\PFC}|
  %&\le \xi^{-\1_{0\in I}}\sum_{j> J_0} \sup_{m'\,:\,J_{m'}=j}\{\rho_{m'}\} \sum_{k}|\Sigma_{\klm,(j,k)}^{\PFC}|.
  \end{aligned}
\end{multline}
where the last line follows because for $m' \ne 0$ we have
$(j',k') \in B_{m} \implies j' = J_{m'}$ and $B_{m'}$ is a singleton. By
\cref{pro:covctrl:7}-(\ref{item:cov2},\ref{item:cov4}), because there are no more than
$\lesssim 2^{j}$ wavelets at each level $j \geq 0$, and because
$\sup_x\sum_k|\basis_{j,k}(x)| \lesssim 2^{j/2}$ (see
\cref{main-pro:coeff:new}-\eqref{main-item:norm:2}), we obtain
\begin{align}
  \label{eq:covctrl:8}
  \sum_{k'} |\Sigma_{j,k,j',k'}^{\PFC}|%
  &\lesssim (1\vee\gamma)
    \Big(2^{-(j \vee j')\minsmooth}\int|\basis_{j,k}|\sum_{k'}|\basis_{j',k'}|%
    + 2^{-j(\minsmooth+1/2)}2^{-j'(\minsmooth + 1/2)}2^{j'}%
    \Big)\\
  &\lesssim (1\vee\gamma)%
    2^{-j/2}2^{j'/2}\Big(2^{-(j\vee j')\minsmooth} + 2^{-j\minsmooth}2^{-j'\minsmooth} \Big)\\
  &\lesssim%
     (1\vee\gamma)%
    2^{-j/2}2^{j'/2} 2^{-(j\vee j')\minsmooth}.
\end{align}
Therefore,
\begin{align}
  % \Big| \sum_{m\in I} \sum_{m'\in I^{c}}\EE_{\L}[F_{m}F_{m'}f(\PFC)]\1_{m\ne 0}\1_{m'\ne 0} \Big|\\%
  R_3%
  &\lesssim (1\vee\gamma) \sum_{\substack{m\in I\\m\ne 0}}\sum_{(j,k)\in B_{m}}|f_{j,k}^m|\sum_{j'> J_0}\sup_{m'\,:\, J_{m'}=j'}\{H_I(m')\} 2^{-J_{m}/2}2^{j'/2}2^{-(J_{m} \vee  j')\minsmooth}\\
  \label{eq:covctrl:58}
  &\lesssim (1\vee\gamma) \sum_{\substack{m\in I\\m\ne 0}}\|F_m\|_2 \sum_{j> J_0}\sup_{m'\,:\, J_{m'}=j}\{H_I(m')\} 2^{-J_{m}/2}2^{j/2}2^{-(J_{m} \vee  j)\minsmooth}.
\end{align}
We decompose
summation in the previous display into the sum of the following terms:
\begin{align}
  \label{eq:covctrl:105a}
  T_1&\coloneqq%
       \sum_{\substack{m\in I\\m\ne 0}} 2^{-J_{m}/2}\|F_m\|_2\1_{J_{m}\leq j_{n}} \sum_{j= J_{1}}^{J_{m}}\sup_{m'\,:\, J_{m'}=j}\{H_I(m')\}2^{j/2}2^{-J_{m}\minsmooth},\\
  \label{eq:covctrl:105b}
  T_2&\coloneqq%
       \sum_{\substack{m\in I\\m\ne 0}}  2^{-J_{m}/2}\|F_m\|_2\1_{J_{m}\leq j_{n}}  \sum_{j=J_{m}+1}^{j_{n}}\sup_{m'\,:\, J_{m'}=j}\{H_I(m')\} 2^{j/2}2^{-j\minsmooth},\\
  \label{eq:covctrl:105c}
  T_3&\coloneqq%
       \sum_{\substack{m\in I\\m\ne 0}}  2^{-J_{m}/2}\|F_m\|_2\1_{J_{m}\leq j_{n}}  \sum_{j>j_{n}}\sup_{m'\,:\, J_{m'}=j}\{ H_I(m') \}2^{j/2}2^{- j\minsmooth},\\
  \label{eq:covctrl:105d}
  T_4&\coloneqq%
       \sum_{\substack{m\in I\\m\ne 0}}  2^{-J_{m}/2}\|F_m\|_2\1_{J_{m}> j_{n}} \sum_{j= J_{1}}^{j_{n}}\sup_{m'\,:\, J_{m'}=j}\{ H_I(m') \} 2^{j/2}2^{-J_{m}\minsmooth},\\
  \label{eq:covctrl:105e}
  T_5&\coloneqq%
       \sum_{\substack{m\in I\\m\ne 0}} 2^{-J_{m}/2}\|F_m\|_2\1_{J_{m}> j_{n}}\sum_{j=j_{n}+1}^{J_{m}}\sup_{m'\,:\, J_{m'}=j}\{ H_I(m') \} 2^{j/2}2^{-J_{m}\minsmooth},\\
  \label{eq:covctrl:105f}
  T_6&\coloneqq%
       \sum_{\substack{m\in I\\m\ne 0}} 2^{-J_{m}/2}\|F_m\|_2\1_{J_{m}> j_{n}} \sum_{j> J_{m}}\sup_{m'\,:\, J_{m'}=j}\{ H_I(m') \} 2^{j/2}2^{-j\minsmooth}.
\end{align}
We now bound each terms using the definition of $j_{n}$ and of $H_I(m)$.
Regarding $T_1$ when $J_1 \leq J_{m'} \leq J_{m} \leq j_{n}$ we have
$H_I(m') \leq \Gamma\xi^{-\1_{0\in I}} \sqrt{\log(n)/n}$, thus
\begin{align}
  \label{eq:covctrl:106}
  T_1%
  &\lesssim \Gamma\xi^{-\1_{0\in I}} \sqrt{\frac{\log(n)}{n}} \sum_{m\in I}  2^{-J_{m}\minsmooth} \|F_m\|_2 \1_{J_{m}\leq j_{n}} \1_{m\ne 0}\\
  &\leq%
    \Gamma\xi^{-\1_{0\in I}} \sqrt{\frac{\log(n)}{n}}\Big\{\sum_{m\in I}\|F_m\|_2^2\1_{J_{m}\leq j_{n}} \Big\}^{1/2}%
    \Big\{ \sum_{m\in I}2^{-2J_{m}\minsmooth}\1_{J_{m}\leq j_{n}}\1_{m\ne 0}\Big\}^{1/2}.
\end{align}
But, for $m\ne 0$ we have $J_m > J_0$, and there are no more than
$\lesssim 2^{j}$ blocks of wavelets at each level $j$ so in fact
\begin{align}
  \label{eq:covctrl:113}
  \sum_{m\in I}2^{-2J_{m}\minsmooth}\1_{J_{m} \leq j_{n}}\1_{m\ne 0}%
  \leq \sum_{j= J_0+1}^{j_{n}} \sum_k 2^{-2j\minsmooth}%\\
  \lesssim \sum_{j= J_0+1}^{j_{n}} 2^{j(2\minsmooth - 1)}%\\
  \lesssim 2^{-J_{0}(2\minsmooth - 1)}.
\end{align}
Therefore,
\begin{align}
  \label{eq:covctrl:114}
  T_1
  &\lesssim%
    2^{-J_{0}(\minsmooth -1/2)}\Gamma\xi^{-\1_{0\in I}} \sqrt{\frac{\log(n)}{n}}\Big\{\sum_{m\in I}\|F_m\|_2^2\1_{J_{m}\leq j_{n}} \Big\}^{1/2}.%
\end{align}
Since $\minsmooth > 1/2$ in this case, $2^{-J_0(\minsmooth - 1/2)}$ can be made
$\ll \delta$ by taking $K_0$ in \cref{main-ass:1} sufficiently large. The same goes for $T_2$ because
$\sum_{j=J_{m+1}}^{j_{n}}2^{j/2}2^{-j\minsmooth} \lesssim 2^{-J_{m}(\minsmooth -1/2)}$,
and thus
\begin{align}
  \label{eq:covctrl:107}
  T_2
  &\lesssim%
     2^{-J_0(\minsmooth-1/2)} \Gamma\xi^{-\1_{0\in I}} \sqrt{\frac{\log(n)}{n}}\Big\{\sum_{m\in I}\|F_m\|_2^2\1_{J_{m}\leq j_{n}} \Big\}^{1/2}.%
\end{align}
Regarding $T_3$, when $J_{m'} > j_{n}$ we have
$H_I(m') \leq \gamma 2^{-J_{m'}(\smooth+1/2)}$, and so
\begin{align}
  \label{eq:covctrl:108}
  T_3%
  &\lesssim%
    \gamma \sum_{m\in I} 2^{-J_{m}/2}\|F_m\|_2\1_{J_{m}\leq j_{n}}\sum_{j>j_{n}}2^{- j(\smooth + \minsmooth)} \1_{m\ne 0},\\
  &\lesssim %
    \gamma 2^{-j_{n}(\smooth + 1/2)}2^{-j_{n}(\minsmooth - 1/2)} \sum_{m\in I} 2^{-J_{m}/2} \|F_m\|_2\1_{J_{m} \leq j_{n}} \1_{m\ne 0}\\
  &\leq \gamma 2^{-j_{n}(\minsmooth - 1/2)} \cdot \sqrt{\frac{\log(n)}{n}} \cdot \Big\{ \sum_{m\in I}\|F_m\|_2^2 \1_{J_{m}\leq j_{n}}\Big\}^{1/2}\Big\{\sum_{m\in I}2^{-J_{m}}\1_{J_{m} \leq j_{n}}\Big\}^{1/2}.
\end{align}
But, there are no more than $ \lesssim 2^{j}$ wavelets at each level $j$ so in fact
\begin{align}
  \label{eq:covctrl:115}
  \sum_{m\in I} 2^{-J_{m}}\1_{J_{m} \leq j_{n}}%
  &\leq \sum_{j\geq 0}^{j_{n}} \sum_k 2^{-j}%\\
  \lesssim j_{n}.
\end{align}
Therefore,
\begin{align}
  \label{eq:covctrl:116}
  T_3%
  &\lesssim \gamma \sqrt{\frac{\log(n)}{n}} \cdot  \sqrt{j_{n}}2^{-j_{n}(\minsmooth -1/2)}\Big\{ \sum_{m\in I}\|F_m\|_2^2 \1_{J_{m}\leq j_{n}}\Big\}^{1/2}.
\end{align}
Regarding $T_4$,
\begin{align}
  \label{eq:covctrl:111}
  T_4%
  &\leq%
   \Gamma\xi^{-\1_{0\in I}} \sqrt{\frac{\log(n)}{n}} \sum_{m\in I} 2^{-J_{m}/2}\|F_m\|_2\1_{J_{m}> j_{n}}\sum_{j= J_{1}}^{j_{n}} 2^{j/2}2^{-J_{m}\minsmooth} \1_{m\ne 0},\\
  % &\lesssim \beta_{2}\sqrt{\frac{\log(n)}{n}}2^{j_{n}/2}%
  %   \sum_{m\in I}\Big\{\sum_{\klm\in B_m}|f_{m,\klm}|^2 \Big\}^{1/2} \sqrt{J_{m}}2^{-J_{m}(\minsmooth + 1/2)}\1_{J_{m}>j_{n}}\1_{m\ne 0}\\
  &\lesssim \Gamma\xi^{-\1_{0\in I}} \sqrt{\frac{\log(n)}{n}}2^{j_{n}/2}%
    \Big\{ \sum_{m\in I}\|F_m\|_2^21_{J_{m}> j_{n}} \Big\}^{1/2}%
    \Big\{ \sum_{m\in I} 2^{-J_{m}(2\minsmooth+1)}\1_{J_{m}>j_{n}} \Big\}^{1/2}.
\end{align}
But there are no more than
$\lesssim 2^{j}$ wavelets at each level $j$ so in fact
\begin{align}
  \label{eq:covctrl:47}
  \sum_{m\in I} 2^{-J_{m}(2\minsmooth+1)}\1_{J_{m}>j_{n}}%
  \leq \sum_{j> j_{n}}\sum_k 2^{-j(2\minsmooth + 1)}%
  \lesssim \sum_{j> j_{n}}2^{-2j\minsmooth}%
  \lesssim 2^{-2j_{n}\minsmooth}.
\end{align}
Therefore,
\begin{align}
  \label{eq:covctrl:60}
  T_4%
  &\lesssim
    \Gamma\xi^{-\1_{0\in I}} \sqrt{\frac{\log(n)}{n}}2^{-j_{n}(\minsmooth - 1/2)} \Big\{ \sum_{m\in I}\|F_m\|_2^2\1_{J_{m}> j_{n}} \Big\}^{1/2}%\\
  % &= o\Big(\frac{1}{\sqrt{n}} \Big) \cdot  \Big\{ \sum_{m\in I}\|F_m\|_2^2 \1_{J_{m}> j_{n}} \Big\}^{1/2}.
\end{align}
Regarding $T_5$, for $J_{m'} > j_{n}$ we have
$H_I(m') \leq  \gamma 2^{-J_{m'}(\smooth + 1/2)}$, thus
\begin{align}
  \label{eq:covctrl:63}
  T_5%
  &=\gamma \sum_{m\in I}  2^{-J_{m}(\minsmooth + 1/2)} \|F_m\|_2 \1_{J_{m}> j_{n}}\sum_{j=j_{n}+1}^{J_{m}} 2^{-j\smooth}\1_{m\ne 0} \\
  &\lesssim%
    \gamma 2^{-j_{n}\smooth} \sum_{m\in I}\|F_m\|_2 2^{-J_{m}(\minsmooth + 1/2)}\1_{J_{m}>j_{n}}\\
  &\leq \gamma 2^{-j_{n}\smooth} \Big\{\sum_{m\in I}\|F_m\|_2^2 \1_{J_{m}>j_{n}} \Big\}^{1/2}%
    \Big\{\sum_{m\in I}2^{-J_{m}(2\minsmooth + 1)}\1_{J_{m}>j_{n}} \Big\}^{1/2}\\
  &\lesssim \gamma 2^{-j_{n}(\smooth + \minsmooth)}\Big\{\sum_{m\in I}\|F_m\|_2^2\1_{J_{m}>j_{n}} \Big\}^{1/2},
\end{align}
where the last line follows by \cref{eq:covctrl:47}. But by \cref{main-eq:44} we have
$\gamma 2^{-j_{n}(\smooth+\minsmooth)} = \gamma 2^{-j_{n}(\smooth + 1/2)}2^{-j_{n}(\minsmooth -1/2)} \sim \Gamma \sqrt{\log(n)/n} 2^{-j_{n}(\minsmooth - 1/2)}$.
Therefore,
\begin{align}
  \label{eq:covctrl:73}
  T_5%
  &\lesssim  2^{-j_n(\minsmooth - 1/2)} \Gamma \sqrt{\frac{\log(n)}{n}}%
    \Big\{\sum_{m\in I}\|F_m\|_2^2\1_{J_{m}>j_{n}} \Big\}^{1/2}.
\end{align}
It remains $T_6$ But for $J_{m'} > J_{m} > j_{n}$ we again have that
$H_I(m') \leq  \gamma 2^{-j(\smooth + 1/2)}$, and thus
\begin{align}
  \label{eq:covctrl:76}
  T_6
  &= \gamma \sum_{m\in I}  2^{-J_{m}/2} \|F_m\|_2 \1_{J_{m}> j_{n}}\sum_{j> J_{m}}2^{-j(\smooth + \minsmooth)}\\
  &\lesssim \gamma \sum_{m\in I}\|F_m\|_2 2^{-J_{m}(s + 1/2 + \minsmooth)}\1_{J_{m}>j_{n}}\\
  &\leq \gamma \Big\{ \sum_{m\in I}\|F_{m}\|_2^2\1_{J_{m}> j_{n}} \Big\}^{1/2}%
    \Big\{\sum_{m\in I} 2^{-J_{m}(2\smooth + 1 + 2\minsmooth)}\1_{J_{m}> j_{n}} \Big\}^{1/2}\\
  &\lesssim \gamma 2^{-j_{n}(\smooth + \minsmooth)} \Big\{ \sum_{m\in I}\|F_m\|_2^2 \1_{J_{m}> j_{n}} \Big\}^{1/2},
\end{align}
where the last line follows by the same arguments that led to \cref{eq:covctrl:47}. As
before,
\begin{align}
  \label{eq:covctrl:82}
  T_6%
  &\lesssim 2^{-j_n(\minsmooth - 1/2)}\Gamma  \sqrt{\frac{\log(n)}{n}}%
    \cdot \Big\{ \sum_{m\in I}\|F_m\|_2^2 \1_{J_{m}> j_{n}} \Big\}^{1/2}.
\end{align}
The bound for $R_3$ is obtained by \cref{eq:covctrl:58} and all the estimates on
$T_1,\dots,T_6$, and by taking the constants in \cref{main-ass:1} sufficiently large.

\paragraph*{Bound on $R_3$ : case $0 < \smooth \leq 1/2$.}

We note that the \cref{eq:covctrl:58} remains valid in this case, as well as the
decomposition of \cref{eq:covctrl:105a,eq:covctrl:105b,eq:covctrl:105c,eq:covctrl:105d,eq:covctrl:105e,eq:covctrl:105f}. We
again bound each of these terms. Regarding $T_1$ when $J_1 \leq J_{m'} \leq J_{m} \leq j_{n}$,
we have $H_I(m') \leq \Gamma\xi^{-\1_{0\in I}}  2^{-J_{m'}/2}\minimaxrate$, thus%
\begin{align}
  \label{eq:covctrl:93}
  T_1%
  &\leq%
    \Gamma\xi^{-\1_{0\in I}} \sum_{m\in I} 2^{-J_{m}/2} \|F_m\|_2\1_{J_{m}\leq j_{n}}%
    \sum_{j=J_{1}}^{J_{m}} \minimaxrate 2^{-J_{m}\smooth}\1_{m\ne 0}\\
  &\lesssim \Gamma \xi^{-\1_{0\in I}} \minimaxrate \sum_{m\in I} J_{m} 2^{-J_{m}(\smooth + 1/2)}\|F_m\|_2\1_{J_{m}\leq j_{n}}\1_{m\ne 0}\\
  &\leq \Gamma \xi^{-\1_{0\in I}} \minimaxrate%
    \Big\{ \sum_{m\in I}\|F_m\|_2^2 \1_{J_{m}\leq j_{n}}\1_{m\ne 0} \Big\}^{1/2}%
    \Big\{ \sum_{m\in I} J_{m}2^{-J_{m}(2\smooth + 1)}\1_{J_{m}\leq j_{n}}\1_{m\ne 0} \Big\}^{1/2}\\
  &\leq  \Gamma \xi^{-\1_{0\in I}} J_02^{-J_{0}\smooth} \minimaxrate%
     \Big\{ \sum_{m\in I} \|F_m\|_2^2 \1_{J_{m}\leq j_{n}}\1_{m\ne 0} \Big\}^{1/2}%
    \Big\{ \sum_{m\in I}2^{-J_{m}}\1_{J_{m}\leq j_{n}}\1_{m\ne 0} \Big\}^{1/2},
\end{align}
where the last line follows for $J_{0}$ taken sufficiently large, because
$m \ne 0 \Rightarrow J_{m}> J_{0}$, and thus
$J_m2^{-2J_{m}\smooth} \leq J_02^{-2J_{0}\smooth}$. Regarding $T_2$, we
have using the same arguments
\begin{align}
  T_2%
  &\leq  \Gamma \xi^{-\1_{0\in I}} \minimaxrate \sum_{m\in I} 2^{-J_{m}/2} \|F_m\|_2 \1_{J_{m}\leq j_{n}} \sum_{j=J_{m}+1}^{j_{n}}
     2^{-j\smooth} \1_{m\ne 0}\\
  % &\lesssim \beta_{2}\minimaxrate \sum_{m\in I} 2^{-J_{m}(\smooth+ 1/2)}\sum_{\klm \in B_{m}} |f_{m,\klm}|\1_{J_{m}\leq j_{n}}\1_{m \ne 0}\\
  &\lesssim \Gamma  \xi^{-\1_{0\in I}} \minimaxrate \sum_{m\in I}  2^{-J_{m}(\smooth + 1/2)}\|F_m\|_2\1_{J_{m}\leq j_{n}}\1_{m\ne 0}\\
  & \leq \Gamma \xi^{-\1_{0\in I}} \minimaxrate \Big\{\sum_{m\in I}\|F_m\|_2^2\1_{J_{m}\leq j_{n}}\1_{m\ne 0} \Big\}^{1/2}%
    \Big\{\sum_{m\in I} 2^{-J_{m}(2\smooth + 1)}\1_{J_{m} \leq j_{n}}\1_{m\ne 0} \Big\}^{1/2}\\
  &\leq \Gamma \xi^{-\1_{0\in I}} 2^{-J_{0}\smooth} \minimaxrate \Big\{\sum_{m\in I}\|F_m\|_2^2 \1_{J_{m}\leq j_{n}}\1_{m\ne 0} \Big\}^{1/2}%
    \Big\{\sum_{m\in I}2^{-J_{m}}\1_{J_{m} \leq j_{n}}\1_{m\ne 0} \Big\}^{1/2}.
\end{align}
Now for $T_3$,
\begin{align}
  T_3%
  &\leq  \gamma \sum_{m\in I} 2^{-J_{m}/2} \|F_m\|_2 \1_{J_{m}\leq j_{n}} \sum_{j>j_{n}}2^{-2j\smooth}\1_{m\ne 0}\\
  &\lesssim \gamma  2^{-2j_{n}\smooth}\sum_{m\in I} 2^{-J_{m}/2} \|F_m\|_2 \1_{J_{m}\leq j_{n}}\1_{m\ne 0}\\
  &\leq \gamma 2^{-2j_{n}\smooth}\Big\{ \sum_{m\in I}\|F_m\|_2^21_{J_{m}\leq j_{n}}\1_{m\ne 0} \Big\}^{1/2}%
    \Big\{ \sum_{m\in I}2^{-J_{m}}\1_{J_{m}\leq j_{n}}\1_{m\ne 0} \Big\}^{1/2}.
\end{align}
But we note that $\gamma 2^{-2j_{n}\smooth} \lesssim \Gamma^2\minimaxrate^2/\gamma$ by \cref{eq:covctrl:91},
and hence we have,%
\begin{align}
  \label{eq:covctrl:99}
  T_3%
  &\lesssim \frac{\Gamma \minimaxrate}{\gamma}  \cdot\Gamma \minimaxrate \Big\{ \sum_{m\in I}\|F_m\|_2^2\1_{J_{m}\leq j_{n}}\1_{m\ne 0} \Big\}^{1/2}%
    \Big\{ \sum_{m\in I} 2^{-J_{m}}\1_{J_{m}\leq j_{n}}\1_{m\ne 0} \Big\}^{1/2}.% \\
  % &\lesssim \sqrt{\log(n)}2^{-j_{n}\smooth} \minimaxrate%
  %    \Big\{ \sum_{m\in I}\|F_m\|_2^2 \1_{J_{m}\leq j_{n}}\1_{m\ne 0} \Big\}^{1/2}%
  %   \Big\{ \sum_{m\in I}2^{-J_{m}}\1_{J_{m}\leq j_{n}}\1_{m\ne 0} \Big\}^{1/2},
\end{align}
Now for $T_4$, we have
\begin{align}
  T_4%
  &\leq \Gamma \xi^{-\1_{0\in I}} \minimaxrate \sum_{m\in I}  2^{-J_{m}(\smooth +1/2)}\|F_m\|_2\1_{J_{m}> j_{n}} (j_{n} -J_{1})\1_{m\ne 0}\\
  % &\lesssim \beta_{2}j_{n}\minimaxrate \sum_{m\in I} \sqrt{J_{m}} 2^{-J_{m}(\smooth + 1/2)}\|F_m\|_2\1_{J_{m}>j_{n}}\1_{m\ne 0}\\
  &\leq j_{n}\Gamma \xi^{-\1_{0\in I}} \minimaxrate \Big\{\sum_{m\in I}\|F_m\|_2^2\1_{J_{m} > j_{n}} \Big\}^{1/2}%
    \Big\{ \sum_{m\in I} 2^{-J_{m}(2\smooth + 1)}\1_{J_{m}> j_{n}} \Big\}^{1/2}.% \\
  % &\lesssim \log(n) \minimaxrate%
  %   \Big\{\sum_{m\in I}\|F_m\|_2^2\1_{J_{m} > j_{n}} \Big\}^{1/2}%
  %   \Big\{ \sum_{m\in I}2^{-J_{m}(2\smooth + 1)}\1_{J_{m}> j_{n}} + o\Big(\frac{1}{n}\Big) \Big\}^{1/2},
\end{align}
Similarly,
\begin{align}
  T_5%
  &\leq \gamma \sum_{m\in I} 2^{-J_{m}(\smooth + 1/2)}\|F_m\|_2\1_{J_{m}> j_{n}}\sum_{j=j_{n}+1}^{J_{m}} 2^{-j\smooth} \1_{m\ne 0},\\
  &\lesssim \gamma 2^{-j_{n}\smooth}%
    \sum_{m\in I} 2^{-J_{m}(\smooth + 1/2)}\|F_m\|_2 \1_{J_{m}>j_{n}}\\
  &\lesssim \gamma 2^{-j_n\smooth}%
    \Big\{\sum_{m\in I}\|F_m\|_2^2\1_{J_{m}>j_{n}} \Big\}^{1/2}%
    \Big\{\sum_{m\in I} 2^{-J_{m}(2\smooth + 1)}\1_{J_{m}>j_{n}}\Big\}^{1/2}.% \\
  % &\lesssim  \sqrt{\log(n)} \minimaxrate%
  %   \Big\{\sum_{m\in I}\|F_m\|_2^2\1_{J_{m}>j_{n}} \Big\}^{1/2}%
  %   \Big\{\sum_{m\in I}2^{-J_{m}(2\smooth + 1)}\1_{J_{m}>j_{n}} + o\Big(\frac{1}{n}\Big)\Big\}^{1/2}.
\end{align}
Finally,
\begin{align}
  T_6%
  &\leq \gamma \sum_{m\in I} 2^{-J_{m}/2} \|F_m\|_2\1_{J_{m}> j_{n}}\sum_{j> J_{m}}2^{-2j\smooth} \\
  % &\lesssim \beta_{3}\sum_{m\in I}2^{-J_{m}(2\smooth + 1/2)} \sum_{\klm\in B_{m}}|f_{m,\klm}| \1_{J_{m}>j_{n}}\\
  &\lesssim \gamma  \sum_{m\in I}  2^{-J_{m}(2\smooth + 1/2)}\|F_m\|_2\1_{J_{m}> j_{n}}\\
  &\lesssim \gamma \Big\{\sum_{m\in I}\|F_m\|_2^2\1_{J_{m}>j_{n}} \Big\}^{1/2}%
    \Big\{\sum_{m\in I}2^{-J_{m}(4\smooth + 1)}\1_{J_{m}>j_{n}} \Big\}^{1/2}\\
  &\leq \gamma 2^{-j_{n}\smooth}%
    \Big\{\sum_{m\in I}\|F_m\|_2^2\1_{J_{m}>j_{n}} \Big\}^{1/2}%
    \Big\{\sum_{m\in I}2^{-J_{m}(2\smooth + 1)}\1_{J_{m}>j_{n}} \Big\}^{1/2}.%\\
  % &\lesssim \Gamma \minimaxrate%
  %   \Big\{\sum_{m\in I}\|F_m\|_2^2\1_{J_{m}>j_{n}} \Big\}^{1/2}%
  %   \Big\{\sum_{m\in I}2^{-J_{m}(2\smooth + 1)}\1_{J_{m}>j_{n}} \Big\}^{1/2}.
\end{align}
The bound for $R_3$ is obtained by \cref{eq:covctrl:58} and all the estimates on
$T_1,\dots,T_6$ and by taking the constants in \cref{main-ass:1} sufficiently large.
\end{proof}

\begin{proof}[Proof of \cref{pro:covctrl:2}]
We use that
$F_{m} = \sum_{(j,k) \in B_{m}}\Inner{F_{m},\basis_{j,k}}(\basis_{j,k} - \EE_{\L}[\basis_{j,k}])$
by \cref{main-pro:coeff:new}-\eqref{main-item:coeff:1}. Then, by
\cref{main-pro:coeff:new}-\eqref{main-item:norm:2} we deduce
that,
\begin{align}
  \label{eq:6}
  |\rSPFC(x)|%
  & \leq \sum_{m\in I^{c}}\sum_{(j,k) \in B_{m}}|\Inner{F_{m},\basis_{j,k}}||\basis_{j,k}(x) - \EE_{\L}[\basis_{j,k}]|\\
  &\leq \sum_{j\geq 0} \sup_{m \in I^{c},\, J_{m}=j}\Big\{ \|F_{m}\|_{2}\sum_{k}|\basis_{j,k}(x) - \EE_{\L}[\basis_{j,k}]| \Big\}\\
  &\lesssim \sum_{j\geq 0} \sup_{m\,:\; J_{m}=j}\{H_I(m)2^{J_{m}/2}\}
\end{align}
By the same argument as in the proof of \cref{main-cor:strategybound}, the previous
is seen to be
$\lesssim \max\{(\Gamma/\gamma)^{\frac{2\smooth}{2\smooth + 1}},\, \Gamma(\gamma/\Gamma)^{\frac{1}{2\smooth+1}}\}\minimaxrate$
when $\smooth > 1/2$, and $\lesssim \log(\Gamma \minimaxrate/\gamma)\Gamma\minimaxrate$ otherwise.
\end{proof}

\begin{proof}[Proof of \cref{pro:covctrl:7}, \cref{item:cov2}]
We define the function $h(x) \coloneqq \dens_{\L}(x)A(\SPFC(x))$. With this
definition, we can rewrite
\begin{align}
  \label{eq:covctrl:67}
  \Sigma^{\PFC}_{j,k,j',k'}%
  &= \int h\cdot\big( \basis_{j,k} - \EE_{\L}[\basis_{j,k}] \big)%
    \big( \basis_{j',k'} - \EE_{\L}[\basis_{j',k'}]\big)\\
  &= \int h\cdot \basis_{j,k}\basis_{j',k'}%
    - \EE_{\L}[\basis_{j,k}]\int h \cdot\basis_{j',k'}\\%
  &\quad%
    - \EE_{\L}[\basis_{j',k'}]\int h\cdot \basis_{j,k}%
    + \EE_{\L}[\basis_{j,k}]\EE_L[\basis_{j',k'}]\int h.
\end{align}
By assumption $\|\SPFC\|_{\infty} \leq 1$, thus
$\|A(\SPFC(\cdot))\|_{\infty} \lesssim 1 $ too. Further, for any $g \in \zygmund$ we have
$|\int g\basis_{j,k}| \leq
C\|g\|_{\infty,\infty,\smooth}2^{-j(\smooth + 1/2)}$ for a universal
constant depending eventually on $\smooth$; see for instance
\citet{gine2016mathematical}. Thus, we deduce that for $C > 0$ eventually
depending on $\minsmooth$,
\begin{equation}
  \label{eq:covctrl:69}
  |\Sigma_{j,k,j',k'}^{\PFC}|%
  \leq \Big| \int h \cdot \basis_{j,k}\basis_{j',k'}\Big|%
  + C\max\big\{\|p\|_{\infty,\infty,\minsmooth},\,
  \|h\|_{\infty,\infty,\minsmooth} \big\}%
  2^{-j(\minsmooth + 1/2)}2^{-j'(\minsmooth + 1/2)}.
\end{equation}
We remark that $p = \exp\{\L\}$, and thus $\|p\|_{\infty,\infty,\minsmooth}$ is
in turn bounded by a constant depending only on $(R,\smooth)$. Further $(j,k)
\ne (j',k')$, thus $\int \basis_{j,k}\basis_{j',k'} = 0$, hence for any
$y \in [0,1] $,
\begin{equation}
  \label{eq:covctrl:36}
  \int h\cdot \basis_{j,k}\basis_{j',k'}%
  = \int_{\domain} \big(h(x) - h(y)\big)\basis_{j,k}(x)\basis_{j',k'}(x)\intd x.
\end{equation}
If $S_{j,k,j',k'} \coloneqq \supp \basis_{j,k} \cap \supp \basis_{j',k'} =
\varnothing$, then $\int h\cdot \basis_{j,k}\basis_{j',k'} = 0$ and the result is
immediate. We now consider that $S_{j,k,j',k'} \ne \varnothing$. In this
case, pick $y \in S_{j,k,j',k'}$ arbitrary and remark that for all
$x \in S_{j,k,j',k'}$ we have
$|x - y| \leq C2^{-j \vee j'}$ for a constant $C > 0$
depending only on the wavelet basis. Then by \cref{eq:covctrl:36},
\begin{align}
  \Big| \int h\cdot \basis_{j,k}\basis_{j',k'} \Big|%
  &\leq \sup_{x\ne y} \frac{|h(x) -
    h(y)|}{|x - y|^{\minsmooth}} \int_{\domain}|x -
  y|^{\minsmooth}|\basis_{j,k}(x)\basis_{j',k'}(x)|\,\intd x\\
  \label{eq:covctrl:95}
  &\leq \sup_{x\ne y} \frac{|h(x) - h(y)|}{|x - y|^{\minsmooth}}%
    \cdot 2^{-(j \vee j')\minsmooth}\int |\basis_{j,k}\basis_{j',k'}|.
\end{align}
Then, we obtain the result of the proposition by bounding the
$\minsmooth$-Hölder norm of $h$. For all $x,y \in \domain$,
\begin{align}
  \label{eq:covctrl:54}
  |h(x) - h(y)|
  &\leq |\dens_{\L}(x) - \dens_{\L}(y)|\cdot |A(\SPFC(x))|%
    + \dens_{\L}(y)|A(\SPFC(x)) - A(\SPFC(y))|\\
  &= |e^{\L(x)} - e^{\L(y)}|\cdot |A(\SPFC(x))|%
    + e^{\L(y)}|A(\SPFC(x)) - A(\SPFC(y))|.
\end{align}
Whenever $\L \in \dclass(R,\smooth)$, there is a constant $C > 0$
depending only on $R$ such that
$|e^{\L(x)} - e^{\L(y)}| \leq C|\L(x) - \L(y)|$
and $\sup_x e^{\L(x)} \leq C$. Also $\|\SPFC\|_{\infty} \leq 1$ thus
$\sup_x |A(\SPFC(x))| \lesssim 1$ and $\sup_x|A'(\SPFC(x))| \lesssim  1$. We deduce,
\begin{align}
  \label{eq:covctrl:56}
  |h(x) - h(y)|%
  &\lesssim  C|\L(x) - \L(y)|%
    + C|\SPFC(x) - \SPFC(y)|.% \\
\end{align}
By construction $0 < \minsmooth < 1$, which implies that the $\minsmooth$-Hölder
norm is equivalent to the $\|\cdot\|_{\infty,\infty,\minsmooth}$ norm
\citep[Equations 4.149 and 4.152]{gine2016mathematical}. Then, since
$\L \in \dclass(R,\smooth) \subseteq \dclass(R,\minsmooth)$,
\begin{equation}
  \label{eq:covctrl:59}
  \sup_{x\ne y} \frac{|h(x) - h(y)|}{|x - y|^{\minsmooth}} \lesssim  RC +
  C\|\SPFC\|_{\infty,\infty,\minsmooth}.
\end{equation}
By equivalence of norms, $\|h\|_{\infty,\infty,\minsmooth}$ is also bounded by a
constant (eventually depending on $\minsmooth$) times the last display. Hence,
the conclusion follows by combining \cref{eq:covctrl:69,eq:covctrl:95} with the
last display.
\end{proof}

\begin{proof}[Proof of \cref{pro:covctrl:7}, \cref{item:cov3}]
By the \cref{item:cov2},
\begin{align}
  \label{eq:covctrl:131}
  \sum_k|\Sigma^{\PFC}_{j,k,j',k'}|%
  &\lesssim%
    (1 + \|\SPFC\|_{\infty,\infty,\minsmooth})\Big\{ 2^{-j'\minsmooth} \int|\basis_{j',k'}| \sum_k|\basis_{j,k}|%
    + 2^{-j'(\minsmooth + 1/2)} \sum_k2^{-j(\minsmooth + 1/2)}\Big\}\\
  &\lesssim (1 + \|\SPFC\|_{\infty,\infty,\minsmooth})\Big\{2^{-j'\minsmooth}2^{j/2} \int|\basis_{j',k'}|%
    + 2^{-j'(\minsmooth + 1/2)} 2^{-j \minsmooth}2^{j/2}\Big\},
\end{align}
where the last line follows by \cref{main-pro:coeff:new}-\eqref{main-item:norm:2} and
also because there are no more than $\lesssim 2^j$ wavelets at each level $j$. The
conclusion follows because $\|\basis_{j',k'}\|_1 \lesssim 2^{-j'}$ for all
$(j',k') \in \kset$.
\end{proof}

\begin{proof}[Proof of \cref{pro:covctrl:7}, \cref{item:cov4}]
Note that if $(j,k) \in B_0$ then
\begin{align}
  2^{j (\minsmooth + 1/2)}|\Inner{\rSPFC,\basis_{j,k}}|%
  &\leq 2^{j(\minsmooth + 1/2)}\|\rSPFC\|_{\infty}\|\basis_{j,k}\|_1\\
  &\leq 2^{J_0\minsmooth}\|\rSPFC\|_{\infty}\\
  \label{eq:covctrl:89}
  &\lesssim 2^{J_0\minsmooth} \log(\Gamma\minimaxrate/\gamma)\Gamma \minimaxrate,
\end{align}
where the last line follows by \cref{pro:covctrl:2}. In the case where
$(j,k)\in B_0^c \cap \Set{(j,k) \given j \leq j_n}$, then by
\cref{main-pro:coeff:new}-\eqref{main-item:coeff:2} (since under \cref{main-ass:1} we can
assume wlog that $J_0$ is arbitrarily large),%
\begin{align}
  \label{eq:covctrl:92}
  2^{j (\minsmooth + 1/2)}|\Inner{\rSPFC,\basis_{j,k}}|%
  &\leq  2^{j(\minsmooth + 1/2)}\sup_{m\in I^c,m\ne 0}|\Inner{F_m,\basis_{j,k}}|\\
  &\leq 2^{j(\minsmooth + 1/2)} \Gamma \xi^{-\1_{0\in I}} \sup_{m:J_m=j}\{\rho_m\}\\
  &\leq \xi^{-\1_{0\in I}}%
    \begin{cases}
      2^{j_n(\smooth + 1/2)} \Gamma \sqrt{\log(n)/n} &\mathrm{if}\ \smooth > 1/2,\\
      \Gamma 2^{j_n\smooth} \minimaxrate &\mathrm{if}\ \smooth > 1/2,
    \end{cases}
\end{align}
where the last line follows from the definition of $\rho_m$ and because $j \le j_n$.
Then by definition of $j_n$ in \cref{main-eq:44}, for all
$(j,k)\in B_0^c \cap \Set{(j,k) \given j \leq j_n}$
\begin{align}
  \label{eq:covctrl:97}
  2^{j (\minsmooth + 1/2)}|\Inner{\rSPFC,\basis_{j,k}}|%
  \leq \xi^{-\1_{0\in I}}\gamma.
\end{align}
Finally, in the case where $(j,k)\in B_0^c \cap \Set{(j,k) \given j > j_n}$, then by
\cref{main-pro:coeff:new}-\eqref{main-item:coeff:2}
\begin{align}
  \label{eq:covctrl:109}
  2^{j (\minsmooth + 1/2)}|\Inner{\rSPFC,\basis_{j,k}}|%
  &\leq 2^{j(\minsmooth + 1/2)}\sup_{m\in I^c,m\ne 0}|\Inner{F_m,\basis_{j,k}}|%\\
  \leq \gamma.
\end{align}
Combining \cref{eq:covctrl:89,eq:covctrl:97,eq:covctrl:109}, there is a universal $C > 0$ such that,
\begin{align}
  \label{eq:covctrl:1}
  \|\rSPFC\|_{\infty,\infty,\minsmooth}%
  \leq  \max\{2^{J_0\minsmooth}\log(\Gamma\minimaxrate/\gamma)\Gamma \minimaxrate,\,\gamma\}.
\end{align}
So if $K_1$ in \cref{main-ass:1} is
large enough we have that $ \|\rSPFC\|_{\infty,\infty,\minsmooth} \leq \gamma$.%
\end{proof}

\subsection{Proofs of \texorpdfstring{\cref{main-lem:2,main-lem:3}}{Lemmas
    \ref{main-lem:2} and \ref{main-lem:3}}}
\label{sec:proofs-lem:2-lem:3}

\begin{proof}[Proof of \cref{main-lem:2}]
We prove the lemma by computing an upper bound on $\Pi(\mathcal{A}_I)$ and a
lower bound on $\Pi(\LKL_I)$. We start with the upper bound.
Recall that by construction we have
$F_m^{\bm\cbasis} = \sum_{(j,k)\in B_m}(\cbasis_{j,k} - \cbasis^{\L}_{j,k})(\basis_{j,k} - \EE_{\L}[\basis_{j,k}])$.
Take $m \in I  \cap \Set{m\given J_m > j_n}$, assuming without loss of
generality that this set is not empty. Since $j_n \gg 1$ we have by
\cref{main-pro:coeff:new}%
% \fPROBLEM{Is the fact that the last equality is true that obvious ? I would say
%   it is true up to constant ??}
\begin{align}
  \label{eq:sasthm:19}
  \sup_{(j,k)\in B_m}| \cbasis_{j,k} - \cbasis_{j,k}^{\L}|%
  &= \sup_{(j,k)\in B_m}|\Inner{F_m^{\bm\cbasis},\basis_{j,k}}|%
  = \|F_m^{\bm\cbasis}\|_2.
\end{align}
Hence, if $\PF^{\bm\cbasis} \in \mathcal{A}_I$,  then
\begin{align}
  \label{eq:sasthm:22}
  \sup_{(j,k)\in B_m}| \cbasis_{j,k} - \cbasis_{j,k}^{\L}|%
  \geq \gamma 2^{-J_m(\smooth + 1/2)}.
\end{align}
But by assumption $|\cbasis_{j,k}^{\L}| \leq R 2^{-j(\smooth + 1/2)}$. This
implies that if $K_3 > R$ in \cref{main-ass:1} then $\gamma > R$ and it must the
case that $\PF^{\bm\cbasis} \in \mathcal{A}_I$ implies that $\cbasis_{\psi(m)} \ne 0$
for  $m \in I \cap \Set{m \given J_m > j_n}$. We deduce
that%
\begin{align}
  \label{eq:sasthm:25}
  \Pi(\mathcal{A}_I)%
  \leq \prod_{\substack{m\in I\\J_m>j_n}}\omega_{J_m}%
  \leq \exp\Big\{(1 + \mustar)\log(2)\sum_{m\in I}J_m\1_{J_m>j_n} \Big\}.
\end{align}
The previous bound is true for any $I \subseteq \NNInts$. We remark, however, that
$\omega_{J_m} = 0$ when $J_m > \log(n)/\log(2)$, whence the claim that
$\Pi(\mathcal{A}_I) / \Pi(\LKL_I) = 0$ if
$I \cap \Set{m \given J_m > \log(n)/\log(2)} \ne \varnothing$ (it is always the case
that $\Pi(\LKL_I) > 0$, see below). We now compute a lower bound on $\Pi(\LKL_I)$.
Consider the set
\begin{equation}
  \label{eq:sasthm:34b}
  E \coloneqq%
  \Set*{%
    \PF^{\bm\cbasis}\given%
    \begin{array}{c}
      m \in I \cap \Set{m \given J_m \leq j_n} \implies \sup_{(j,k)\in B_m}\sqrt{|B_m|}|\cbasis_{j,k} - \cbasis_{j,k}^{\L}| \leq \eta_n,\\
      m \in I \cap \Set{m \given J_m > j_n} \implies \cbasis_{\psi(m)} = 0.
    \end{array}
  }.
\end{equation}
We will show that $E \subseteq \LKL_I$ for suitable choice of $\eta_n$, and then we will
bound $\Pi(\LKL_I) \geq \Pi(E)$. Pick $\PF^{\bm\cbasis} \in E$. By
\cref{main-pro:coeff:new}, we deduce that%
\begin{align}
  \label{eq:sasthm:28}
  \EE_{\L}[\SPF^2]%
  &\lesssim \sum_{m\in I}\|F_m\|_2^2\\
  &\lesssim \eta_n^2 \sum_{m\in I} \1_{J_m \leq j_n}%
    + \sum_{m\in I} \sup_{(j,k)\in B_m}|\cbasis_{j,k}^{\L}|^2\1_{J_m > j_{n}}\\
  &\leq \eta_n^2 \sum_{m\in I}\1_{J_m \leq j_n}%
    +  R^2 \sum_{m\in I}2^{-J_m(2\smooth + 1)}\1_{J_m > j_n}.
\end{align}
We deduce from \cref{main-pro:1} that if $K_3$ in the \cref{main-ass:1} is sufficiently
large and if $\eta_n = a \delta \sqrt{\log(n)/n}$ for small enough constant $a > 0$,
then $\EE_{\L}[\SPF^2] \leq \delta^2\natrate_I^2$ for all $\PF^{\bm\cbasis}\in E$.
Similarly, by \cref{main-pro:coeff:new}-\eqref{main-item:norm:2},
\begin{align}
  \label{eq:sasthm:34b}
  \|\SPF\|_{\infty}%
  &= \sup_{x\in \domain}\Big| \sum_{m\in I}\sum_{(j,k)\in B_m}(\cbasis_{j,k} - \cbasis_{j,k}^{\L})(\basis_{j,k}(x) - \EE_{\L}[\basis_{j,k}]) \Big|\\
  &\leq \sum_{j\geq J_0} \sup_{\substack{m\in I\\ J_m = j}}\sup_{(j',k')\in B_m}|\cbasis_{j',k'} - \cbasis_{j',k'}^{\L}| \sup_{x\in \domain} \sum_k|\basis_{j,k}(x) - \EE_{\L}[\basis_{j,k}]| \\
  &\lesssim \eta_n \sum_{j = 0}^{j_n}2^{j/2}  + \sum_{j> j_n}2^{j/2}\cdot R2^{-j(\smooth + 1/2)}\\
  &\lesssim \eta_n 2^{j_n/2}+ R 2^{-j_n\smooth}.
\end{align}
Thus, if $\eta_n$ is taken as above and $K_1$ in \cref{main-ass:1} is large enough, it
is the case that $\|\SPF\|_{\infty} \leq \delta$ for all $\PF^{\bm\cbasis}\in E$. We thus have
proven that $E \subseteq \LKL_I$. Further,
\begin{align}
  \label{eq:sasthm:41}
  \Pi(E)%
  &\geq \prod_{\substack{m\in I\\J_m \leq j_n}}\prod_{(j,k) \in B_m}\Big(\omega_jQ_j\big( \sqrt{|B_m|}|\cbasis_{j,k} - \cbasis_{j,k}^{\L}| \leq \eta_n  \big) \Big) \prod_{\substack{m\in I\\J_m > j_n}}\prod_{(j,k) \in B_m}(1 - \omega_j).
\end{align}
First we note that
\begin{align}
  \prod_{\substack{m\in I\\J_m > j_n}}\prod_{(j,k) \in B_m}(1 - \omega_j)%
  &\geq \exp\Big\{-2 \sum_{m\in I}\sum_{(j,k)\in B_m}2^{-j(1+\mustar)}\1_{J_m\leq j_n} \Big\}\\
  &\geq \exp\Big\{-2 \sum_{j = 0}^{j_n} \sum_k2^{-j(1+\mustar)} \Big\}\\
  \label{eq:sasthm:46}
  &\geq C_1,
\end{align}
for a universal constant $C_1 > 0$, where the third line follows because there
are no more than $\lesssim 2^j$ wavelets at each level $j$, and hence
$\sum_{j = 0}^{j_n} \sum_k2^{-j(1+\mustar)} \lesssim 1$ as $\mustar > 0$. Similarly, since
$\omega_j \geq a_12^{-j(1+b_1)}$,
\begin{align}
  \label{eq:sasthm:68}
  \prod_{\substack{m\in I\\J_m \leq j_n}}\prod_{(j,k) \in B_m}\omega_j%
  &\geq \exp\Big\{-\log \frac{1}{a_1}\sum_{m\in I}|B_m|\1_{J_m \leq j_n} -b_1 \sum_{m\in I}\sum_{(j,k)\in B_m}j\1_{J_m> j_n} \Big\}\\
  &\geq \exp\Big\{-\log \frac{1}{a_1}\sum_{m\in I}|B_m|\1_{J_m \leq j_n} -b_1 \sum_{m\in I}|B_m|J_m\1_{J_m> j_n} \Big\},
\end{align}
because by construction $(j,k)\in B_m \implies j \leq J_m$ (with equality
whenever $m \geq 1$). We remark that $|B_0|J_0 \lesssim J_02^{J_0} \leq \log(n)$ (by choosing
$K_1$ large enough in \cref{main-ass:1}), and $|B_m|J_m = J_m \leq j_n \lesssim \log(n)$ for
all $1 \leq m \leq j_n$. Hence, there is a universal constant $C_2 > 0$ such that
\begin{align}
  \label{eq:sasthm:71}
  \prod_{\substack{m\in I\\J_m \leq j_n}}\prod_{(j,k)\in B_m}\omega_j%
  \geq \exp\Big\{- C_2 \log(n)\sum_{m\in I}\1_{J_m \leq j_n} \Big\}.
\end{align}
Finally, for a universal constant $C_3 > 0$,%
\begin{align}
  \label{eq:sasthm:74}
  Q_j\big(  \sqrt{|B_m|}|\cbasis_{j,k} - \cbasis_{j,k}^{\L}| \leq \eta_n  \big)
  &= F\Big(|X - 2^{j(\smooth_0+1/2)}\cbasis_{j,k}^{\L}| \leq 2^{j(\smooth_0 + 1/2)}\sqrt{\frac{\eta_n}{|B_m|}} \Big)\\
  &\geq C_3 2^{j(\smooth_0+1/2)} \sqrt{\frac{\eta_n}{|B_m|}}\\
  &\geq n^{-1/2}.
\end{align}
The second inequality is true because by
assumption
$2^{j(\smooth_0 +1/2)}|\cbasis_{j,k}^{\L}|\leq R 2^{j(\smooth_0 -\smooth)} \leq R$
since $\smooth \geq \smooth_0$; henceforth the distribution $F$ has density bounded
from below by $C_3/2$ in a neighborhood of
$2^{j(\smooth_0 +1/2)}|\cbasis_{j,k}^{\L}|$, with $C_3$ not depending on
the choice of the wavelet coefficient. The last inequality is true whenever
$K_1$ in \cref{main-ass:1} is taken large enough. Then,
\begin{align}
  \label{eq:sasthm:75}
  \prod_{\substack{m\in I\\J_m \leq j_n}}\prod_{(j,k) \in B_m}Q_j\big(\sqrt{|B_m|} |\cbasis_{j,k} - \cbasis_{j,k}^{\L}| \leq \eta_n  \big)
  &\geq%
    \exp\Big\{- \frac{\log(n)}{2}\sum_{m\in I}|B_m|\1_{J_m \leq j_n} \Big\}.
\end{align}
The conclusion of the lemma follows by combining \cref{eq:sasthm:41,eq:sasthm:46,eq:sasthm:71,eq:sasthm:75}.
\end{proof}

\begin{proof}[Proof of \cref{main-lem:3}]
We use the bound of \cref{main-lem:2} in conjunction with \cref{main-thm:dr:1} to obtain a
clean bound on $\EE_{\L}\Pi(\consistencyset \cap \sliceset_I \mid \obs)$. We remark that
\cref{main-lem:2,main-thm:dr:1} imply that
$\EE_L\Pi(\consistencyset \cap \sliceset_I \mid \obs) = 0$ whenever
$I \cap \Set{m \given J_m > \log(n)/\log(2)} \ne \varnothing$. Hence, the bound in
the lemma holds trivially when
$I \cap \Set{m \given J_m > \log(n)/\log(2)} \ne \varnothing$, and we will now assume
without loss of generality that
$I \cap \Set{m \given J_m > \log(n)/\log(2)} = \varnothing$. We will distinguish
between two cases, according to whether
$\frac{(1+\mustar)\log(2)}{2}\sum_{m\in I}J_m\1_{J_m> j_n} \leq \frac{c_2n\natrate_I^2}{4}$
or not.

\paragraph*{Case
  $\frac{(1+\mustar)\log(2)}{2}\sum_{m\in I}J_m\1_{J_m> j_n} \leq \frac{c_2n\natrate_I^2}{4}$.}

We claim that $n\natrate_I^2 \gtrsim \log(n)$ in this scenario. Indeed, if
$\sum_{m\in I}\1_{J_m> j_n} \geq 1$ then
$n\natrate_I^2 \gtrsim \sum_{m\in I}J_m\1_{J_m > j_n} \geq j_n \gtrsim \log(n)$ by definition of
$j_n$. But if $\sum_{m\in I}\1_{J_m > j_n} =  0$, then it must be the case that
$\sum_{m\in I}\1_{J_m\leq j_n} \geq 1$ since by assumption $I \ne \varnothing$. Therefore,
$n\natrate_I^2 \gtrsim \log(n)$ by the estimate of \cref{main-pro:1}. Then, by picking
$\alpha =1/2$ in \cref{main-thm:dr:1},
\begin{align}
  \label{eq:sasthm:79a}
  \EE_{\L}\Pi(\consistencyset \cap \sliceset_I \mid \obs)^{1+2\delta}%
  &\lesssim \exp\Big\{- \frac{c_2n\natrate_I^2}{2} + \frac{c_12^{J_0}|I|}{2}  \Big\} \frac{\Pi(\mathcal{A}_I)^{1/2}}{\Pi(\LKL_I)^{1/2}}.
\end{align}
Using the estimate of \cref{main-lem:2} and the \cref{main-pro:1}, we find that
\begin{align}
  \frac{1}{2}\log \frac{\Pi(\mathcal{A}_I)}{\Pi(\LKL_I)}%
  &\leq \frac{c_4 \log(n)}{2}\sum_{m\in I}2^{J_0\1_{m=0}}\1_{J_m \leq j_n}%
    - \frac{(1+\mustar)\log(2)}{2} \sum_{m\in I}J_m\1_{J_m > j_n}\\
  &\leq \frac{c_4}{2C\Gamma^2}\Big(\xi^{2\1_{0\in I}} + 2^{J_0}\1_{0\in I}\Big)n\natrate_I^2%
    - \frac{(1+\mustar)\log(2)}{2} \sum_{m\in I}J_m\1_{J_m > j_n}\\
  \label{eq:sasthm:57}
  &\leq \Big\{ \frac{c_4(\xi^{2\1_{0\in I}}+2^{J_0}\1_{0\in I})}{2C\Gamma^2} + \frac{c_2}{4} \Big\}n\natrate_I^2%
    - (1+\mustar)\log(2)\sum_{m\in I}J_m\1_{J_m > j_n}.
\end{align}
Also, by \cref{main-pro:1} again,
\begin{align}
  c_12^{J_0}|I|%
  &= c_12^{J_0}\sum_{m\in I}\1_{J_m \leq j_n}%
    + c_12^{J_0}\sum_{m\in I}\1_{J_m > j_n}\\
  \label{eq:sasthm:62}
  &\leq \frac{c_12^{J_0}}{C\Gamma^2\xi^{-2\1_{0\in I}}\log(n)}n\natrate_I^2%
    + \frac{c_12^{J_0}}{j_n} \sum_{m\in I}J_m\1_{J_m> j_n}
\end{align}
So if $\delta > 0$ is taken small enough and the constants $K_1$ and $K_4$ in
\cref{main-ass:1} are taken large enough, we obtain the bound in the statement of the
lemma by combining \cref{eq:sasthm:79a,eq:sasthm:57,eq:sasthm:62}, and because $j_n \gtrsim \log(n)$ can
be made as small as we want in front of $2^{J_0}$ under \cref{main-ass:1}.%

\paragraph*{Case
  $ \frac{(1+\mustar)\log(2)}{2} \sum_{m\in I}J_m\1_{J_m> j_n}> \frac{c_2n\natrate_I^2}{4}$.}

In this case, by taking $\alpha = \frac{1024\cdot c_0 \delta}{1 + 512\cdot c_0\delta}$ in
\cref{main-thm:dr:1}, we obtain that
\begin{align}
  \EE_{\L}\Pi(\consistencyset \cap \sliceset_I \mid \obs)^{1+2\delta}%
  &\leq \frac{3\exp\big(- \alpha c_2  n\natrate_I^2 +   \alpha c_12^{J_0}|I|\big)}{(1 - e^{-c_2 n \natrate_I^2})^{2\alpha}}  \Big\{\frac{\Pi(\mathcal{A}_I)}{\Pi(\LKL_I)} \Big\}^{1-\alpha}\\
  \label{eq:sasthm:79}
  &\leq \frac{3\exp\big( \alpha c_2 n\natrate_I^2 +   \alpha c_12^{J_0}|I|\big)}{(c_2n\natrate_I^2 )^{2\alpha}}  \Big\{\frac{\Pi(\mathcal{A}_I)}{\Pi(\LKL_I)} \Big\}^{1-\alpha},
\end{align}
where the second line follows using that $1 - e^{-x} \geq xe^{-x}$ for $x \geq 0$.
Using the estimate of \cref{main-lem:2} and the \cref{main-pro:1}
\begin{align}
  \label{eq:sasthm:87}
  \log \frac{\Pi(\mathcal{A}_I)}{\Pi(\LKL_I)}%
  &\leq c_4\log(n) \sum_{m\in I}2^{J_0\1_{m=0}}\1_{J_m\leq j_n}%
    - (1+\mustar)\log(2) \sum_{m\in I}J_m\1_{J_m> j_n}\\
  &\leq \frac{c_4(\xi^{2\1_{0\in I}} + 2^{J_0}\1_{0\in I} ) }{C\Gamma^2}n\natrate_I^2%
    - (1+\mustar)\log(2) \sum_{m\in I}J_m\1_{J_m> j_n}.
\end{align}
Hence, for any $t \in (0,1)$, we have,
\begin{multline}
  \label{eq:sasthm:130}
  \log \frac{\Pi(\mathcal{A}_I)}{\Pi(\LKL_I)}%
    \leq \Big\{- \frac{(1-t)c_2}{2} + \frac{c_4(\xi^{2\1_{0\in I}} + 2^{J_0}\1_{0\in I} ) }{C\Gamma^2} \Big\}n\natrate_I^2\\%
      - t (1+\mustar)\log(2) \sum_{m\in I}J_m\1_{J_m> j_n}.
\end{multline}
We also note that for this case to happen, it must be that
$\Set{m \given J_m > j_n} \ne \varnothing$. Since we have assumed that
$I \cap \Set{m \given J_m > \log(n)/\log(2)} = \varnothing$, by \cref{main-pro:1}
\begin{align}
  \label{eq:sasthm:125}
  n \natrate_I^2
  &\geq Cn \gamma^2\sum_{m\in I}2^{-J_m(2\smooth + 1)}\1_{J_m > j_n}%\\%
  \geq C \gamma^2 n^{-2\smooth}.
\end{align}
We deduce from the previous that,
\begin{align}
  \frac{1}{(c_2n\natrate_I^2)^{2\alpha}}%
  &\leq%
    \frac{1}{(C c_2\gamma^2)^{2\alpha}}\exp\big\{ 4\alpha \smooth \log(n) \Big\}\\
  \label{eq:sasthm:144}
  &\leq \frac{1}{(C c_2\gamma^2)^{2\alpha}}\exp\Big\{ 4\alpha \smooth \frac{\log(n)}{j_n}\sum_{m\in I}J_m\1_{J_m>j_n} \Big\}.
\end{align}
But $j_n \gtrsim \log(\gamma/\Gamma) + \log(n)$ by \cref{main-eq:44}, which implies that
$j_n \gtrsim \log(n)$ when $K_1$ in \cref{main-ass:1} is taken sufficiently large. So
taking in addition $K_3$ sufficiently large, we find that there is a universal
constant $C' > 0$ such that
\begin{align}
  \label{eq:sasthm:145}
  \frac{1}{(c_2n\natrate_I^2)^{2\alpha}}%
  &\leq \exp\Big\{ C'\delta \log(2) \sum_{m\in I}J_m\1_{J_m>j_n} \Big\}.
\end{align}
So if $\delta > 0$ is taken small enough, and the constants $K_1$, $K_3$ and $K_4$ in
\cref{main-ass:1} are taken large enough, we obtain the bound in the statement of the
lemma by combining \cref{eq:sasthm:62,eq:sasthm:79,eq:sasthm:130,eq:sasthm:145} and by taking
$t = \frac{(1+2\delta)(1+\mustar/2 + C'\delta)}{1+\mustar}$.
\end{proof}

\section*{Acknowledgements}

\par This work was supported by U.S. Air Force Office of Scientific Research
grant \#FA9550-15-1-0074. The author also thanks Daniel M. Roy for helpful
discussions and the opportunity to work on this project.

%%% Local Variables:
%%% mode: latex
%%% TeX-master: "bernoulli-supp"
%%% End: